\documentclass[10pt,a4paper]{amsart}
\usepackage[margin=3.3cm]{geometry}
\usepackage{amssymb}
\usepackage{graphicx} 
\usepackage[mathscr]{euscript}
\usepackage{enumerate}
\usepackage{xspace}
\usepackage{color}
\usepackage{enumerate}
\usepackage{enumitem}
\usepackage{amsthm}
\usepackage{dsfont}
\usepackage{bbm}
\usepackage[colorlinks=true, linkcolor = blue, citecolor = blue]{hyperref}
\usepackage[numbers]{natbib}
\usepackage[backgroundcolor=white,bordercolor=red]{todonotes}
\DeclareMathAlphabet{\mathpzc}{OT1}{pzc}{m}{it}
\usepackage[nameinlink]{cleveref}
\numberwithin{equation}{section}
\usepackage{tikz}
\usepackage[%cal=boondox,
scr=boondoxo]{mathalfa}
\usepackage{sidecap}
\begin{document}

\theoremstyle{plain}
\newtheorem{theorem}{Theorem}[section]
\newtheorem{lemma}[theorem]{Lemma}
\newtheorem{proposition}[theorem]{Proposition}
\newtheorem{corollary}[theorem]{Corollary}
\newtheorem{definition}[theorem]{Definition}
\theoremstyle{definition}
\newtheorem{remark}[theorem]{Remark}
\newtheorem{example}[theorem]{Example}
\newtheorem{ass}[theorem]{Standing Assumption}

\crefname{lemma}{lemma}{lemmas}
\Crefname{lemma}{Lemma}{Lemmata}
\crefname{corollary}{corollary}{corollaries}
\Crefname{corollary}{Corollary}{Corollaries}
\crefname{ass}{standing Assumption}{standing Assumptions}
\Crefname{ass}{Standing Assumption}{Standing Assumptions}

\newcommand{\dd}{d}
\newcommand{\cadlag}{c\`adl\`ag }
\newcommand{\1}{\mathbb{I}}
\newcommand{\D}{\mathbb{D}} 
\newcommand{\N}{\mathbb{Z}_+}
\renewcommand{\a}{\mathfrak{A}}
\newcommand{\B}{\mathbb{B}}
\newcommand{\oB}{\overline{\B}}
\newcommand{\K}{\mathbb{K}}
\newcommand{\pp}{\partial^p}
\newcommand{\pl}{\partial^l}
\newcommand{\pt}{\partial^t}
\newcommand{\oK}{\overline{\mathbb{K}}}
\newcommand{\ok}{\overline{K}}
\renewcommand{\c}{\mathfrak{c}}
\newcommand{\C}{\mathsf{c}}
\newcommand{\A}{\mathbb{A}}
\newcommand{\bm}{P_\textup{BM}}
\newcommand{\he}{\widehat{\varepsilon}\hspace{0.05cm} }
\newcommand{\oc}{\xrightarrow{\ \omega\ }}
\newcommand{\osc}{\operatorname{osc}}
\newcommand{\cosc}{\overline{\textup{osc}}}
\renewcommand{\epsilon}{\varepsilon}
\newcommand{\dB}{B^\textup{dis}}
\newcommand{\dQ}{Q} 
\newcommand{\dO}{\mathcal{O}^\textup{dis}}
\expandafter\newcommand\csname r@tocindent4\endcsname{4in}
\setcounter{secnumdepth}{4}
\newtheorem*{convention}{Convention}
\newtheorem{claim}{Claim}
\newcommand*{\bigtimes}{\mathop{\raisebox{-.5ex}{\hbox{\huge{$\times$}}}}}
\renewcommand{\k}{\mathfrak{X}}
\newcommand\llambda{{\mathchoice
		{\lambda\mkern-4.5mu{\raisebox{.4ex}{\scriptsize$\backslash$}}}
		{\lambda\mkern-4.83mu{\raisebox{.4ex}{\scriptsize$\backslash$}}}
		{\lambda\mkern-4.5mu{\raisebox{.2ex}{\footnotesize$\scriptscriptstyle\backslash$}}}
		{\lambda\mkern-5.0mu{\raisebox{.2ex}{\tiny$\scriptscriptstyle\backslash$}}}}}
	\newcommand{\q}{\mathbb{Q}}
	\newcommand{\E}{E}

\title[A PHI for Balanced RWRE]{A Parabolic Harnack Principle for balanced difference equations in random Environments.} %\\
\author[N. Berger]{Noam Berger}
\author[D. Criens]{David Criens}

\address{Technical University of Munich, Center for Mathematics, Germany}
\email{noam.berger@tum.de}
\email{david.criens@tum.de, david.criens@stochastik.uni-freiburg.de}

\keywords{\vspace{1ex} Parabolic Harnack inequality, random walks in random environment, quantitative stochastic homogenization, degenerate random difference equations}

\subjclass[2010]{60K37,	35B27, 35K65, 47B80, 60G42, 60J10}

\date{\today}
\maketitle

\frenchspacing
\pagestyle{myheadings}

\begin{abstract}
We consider difference equations in balanced, i.i.d. environments which are not necessary elliptic. In this setting we prove a parabolic Harnack inequality (PHI) 
for non-negative solutions to the discrete heat equation satisfying a (rather mild) growth condition, and we identify the optimal Harnack constant for the PHI.
We show by way of an example that a growth condition is necessary and that our growth condition is sharp.

Along the way we also prove a parabolic oscillation inequality and a (weak) quantitative homogenization result, which we believe to be of independent interest. 
\end{abstract}

\section{Introduction}
\subsection{Background}
Consider the non-divergence form operator
\begin{align} \label{eq: N-D-F-O}
(L^a f) (x) \triangleq \sum_{i, j = 1}^d a_{ij} (x) \frac{d^2f }{dx_i d x_j} (x), \quad (f, x) \in C^2(\mathbb{R}^d) \times \mathbb{R}^d,
\end{align}
where \(a = (a_{ij})_{i, j = 1}^d\) is a measurable function from \(\mathbb{R}^d\) into the set of symmetric positive definite matrices which is uniformly elliptic, i.e. there is a constant \(0 < \lambda \leq 1\) such that 
\[
\lambda \|y\|^2 \leq \sum_{i, j = 1}^d a_{ij} (x) y_i y_j \leq \frac{1}{\lambda} \|y\|^2, \quad (x, y) \in \mathbb{R}^{d} \times \mathbb{R}^d.
\]
For an open domain $D$ in $\mathbb R^{d + 1}$ a function $u \colon D \to \mathbb R$ is called \emph{caloric} if it solves the (backward) heat equation \(\frac{d}{dt} u = - L^a u\).
In a seminal paper, Krylov and Safonov \cite{Krylov81} proved a \emph{parabolic Harnack inequality (PHI)} for non-negative caloric functions from \([0, R^2] \times B_R (0)\) to~$\mathbb R$.
 
 More precisely, they proved the existence of a positive constant \(C = C(\lambda)\) such that for any radius \(R > 0\) and every (non-negative) caloric function \(u\) on \([0, R^2] \times B_R (0)\) it holds that
\begin{align} \tag{\textup{PHI}}
\max_{Q_+} u \leq C \min_{Q_-} u, 
\end{align}
where \(Q_- \triangleq [0, \frac{1}{4} R^2]\times B_{R/2} (0)\) and \(Q_+ \triangleq [\frac{1}{2} R^2, \frac{3}{4} R^2] \times B_{R/2}(0)\).
The PHI has many important applications such as a priori estimates in parabolic H\"older spaces (see \cite{Krylov81}) or H\"older regularity results (see \cite{Safarov80}).
PHIs for discrete uniformly elliptic heat equations can be found in \cite{10.1215/S0012-7094-98-09122-0,10.1214/009117905000000440}, see also \cite{doi:10.1112/plms/s3-63.3.552} for its elliptic (i.e. time independent) counterpart. A version for uniformly elliptic equations with time-dependent coefficients is given in \cite{DG2019}.

Remarkably, the constant in the PHI of Krylov and Safonov does not depend on the regularity of the coefficient \(a\) but only on the ellipticity constant \(\lambda\). It is not hard to see that as \(\lambda\) goes to zero, the constant goes to infinity, and in particular the proof method of \cite{Krylov81} is not helpful in settings that are not uniformly elliptic.

More recently, there is a growing interest in PHIs for settings which are not necessarily uniformly elliptic.
We mention the paper \cite{hambly2009}, where heat equations arising from random walks on percolation clusters (RWPC) are studied, and the articles \cite{andres16,bella20,10.2969/jmsj/06741413}, where the PHI is proved for equations related to the random conductance model (RCM). In these works the PHI was used to prove a local limit theorem for the corresponding stochastic processes.
 
In contrast to \eqref{eq: N-D-F-O}, the equations associated to RWPC and RCM are in divergence form, i.e. reversible.
Discrete equations in non-divergence form appear in the context of \emph{random walks in balanced random environment (RWBRE)}. An \emph{elliptic Harnack inequality (EHI)} for such equations in fully non-elliptic environments has recently been proved in \cite{Berger18}. To the best of our knowledge, this result is the first of its kind for such a degenerate framework.

\subsection{Purpose of the current article}

Our main result is \Cref{theo: PHI} below, which is  a PHI for random difference equations associated to non-elliptic random walks in balanced i.i.d. random environments, which is the setting from \cite{Berger18}.
More precisely, we prove the PHI for all non-negative caloric functions which satisfy a certain exponential growth condition (see discussion in \Cref{sec:growth} below) and we show by example that it can fail without it. As the EHI holds in full generality, our result points to an interesting difference between parabolic and elliptic frameworks. To the best of our knowledge, a comparable phenomenon has not been reported before.
The Harnack constant in our PHI is optimal in the sense that it can be taken arbitrarily close to its counterpart in the PHI of the limiting Brownian motion from the corresponding invariance principle proven in \cite{Berger2014}.

 \subsection{The growth condition}\label{sec:growth}
As alluded to above, in \Cref{theo: PHI} the PHI is only proven for non-negative caloric functions $f\colon B_R(0)\times[0,R^2] \to \mathbb R$ satisfying the growth condition
\eqref{eq: growth cond}, which roughly states that
\begin{equation*}
	\max f \leq e^{R^{2 - \xi}} \min f,
\end{equation*}
 for $\xi$ arbitrarily small. We also find a counter example to the PHI satisfying 
 \begin{equation*}
	\max f  = e^{R^{2}} \min f,
\end{equation*}
showing that the growth condition \eqref{eq: growth cond} is sharp. We wish to make a few remarks regarding this growth condition:

\begin{remark}
Our growth condition is quite mild. In particular, in most applications (e.g. for local limit theorems) all functions that are considered are such that the maximum to minimum ratio grows like a power of $R$, which easily satisfies our growth condition.
\end{remark}

\begin{remark}
To the best of our knowledge, our paper is the first time that a PHI is proven under such a growth condition. We believe however that this phenomenon, namely that a mild growth condition guarantees an otherwise false PHI, exists in a large variety of models which are not uniformly elliptic. In particular, we believe that for random conductances models which are elliptic but not uniformly elliptic, and where the conductances have a thick enough tail around zero (see, e.g. \cite{BBHK}), a similar phenomenon can hold.
\end{remark}

 \subsection{Proof strategy}

 We now comment on the proof of our PHI. The basic strategy is borrowed from Fabes and Stroock \cite{Fabes1989} and their proof for the continuous uniformly elliptic case.
 In general, our Fabes--Stroock argument relies on two central ingredients which are of independent interest: A parabolic oscillation inequality and a parabolic quantitative homogenization estimate.  The former is used for the iterative scheme in the Fabes--Stroock argument and the latter yields estimates for the exit measure of the random walker, which we use roughly the same way Fabes and Stroock \cite{Fabes1989} used heat kernel estimates.
 In contrast to the setting of Fabes and Stroock, our model lacks connectivity in the sense that the movement of the random walker is restricted by holes in the environment.
 In addition, we have to deal with local degeneracies, as the positive transition probabilities in the random environments might not be bounded away from zero. 
 To control the sizes of the holes in the environments we use percolation estimates which use the i.i.d. structure.
 Due to our parabolic setting the speed of the random walker is a major issue. The growth condition ensures that the random walker reaches certain parts of the environments fast enough. 
 The Fabes--Stroock method was also used in \cite{Berger18} to establish the EHI. In contrast to our setting, the issue of speed plays no role in \cite{Berger18}, which also explains why the EHI holds in full generality.
 
\subsection{Possible future research directions}
Before we turn to the main body of this paper, let us comment on follow up questions which are left for future research.
It is interesting to compare our result to those for the RCM. In \cite{andres16,10.2969/jmsj/06741413} it was shown that the PHI holds under certain moment assumptions on the conductances, which are violated in degenerate cases. Our result suggests that also for the RCM the PHI might hold when restricted to a suitable class of functions. Conversely, the results from \cite{andres16,10.2969/jmsj/06741413} suggest that a full PHI might hold for elliptic RWBRE under suitable moment assumptions on the ellipticity constant. 
We think our PHI is a first step into the direction of a local limit theorem for non-elliptic RWBRE. At this point we stress that our PHI cannot be used directly to solve this question as in \cite{andres16}, because the method there relies on a PHI for adjoint equations. In the reversible (self-adjoint) framework from \cite{andres16} it is clear that the PHI also applies to adjoint equations, but in our non-symmetric setting this is not the case. %We leave it open for future research whether our PHI can be used to establish its adjoint counterpart.

The article is structured as follows. In \Cref{sec: 2} we introduce our setting and state our main results. The proofs are given in the remaining sections, whose structure is explained at the end of \Cref{sec: 2}.
\section{Framework and Main Results} \label{sec: 2}

\subsection{The Framework}
Let \(d \geq 2\) and let \(\{e_i \colon i = 1, \dots, d\}\) be the unit vectors in \(\mathbb{Z}^d\). We set \(e_{d + i} \triangleq - e_i\) for \(i = 1, \dots, d\) and define \(\mathcal{M}\) to be the space of all probability measures on \(\{e_i \colon i = 1, \dots, 2d\}\) endowed with the topology of convergence in distribution. Moreover, we define the product space
\[
\Omega \triangleq \prod_{\mathbb{Z}^d} \mathcal{M} %^{\mathbb{Z}^d}
\]
and its Borel \(\sigma\)-field \(\mathcal{F} \triangleq \mathcal{B}(\Omega)\). An element \(\omega\in \Omega\) is called \emph{environment}.
Let\(P\) be an i.i.d. Borel probability measure on \(\Omega\), i.e.
\[
P \triangleq \bigotimes_{\mathbb{Z}^d} \nu \text{ for some \(\nu\in \mathcal{M}\)}.
\]
We denote the space of all paths \(\N \to \mathbb{Z}^d\), equipped with the product topology, by \(\D\) and the coordinate process by \(X = (X_n)_{n \in \N}\), i.e. \(X_n (\alpha) = \alpha (n)\) for \((\alpha, n) \in \D \times \N\).
For every \(\omega \in \Omega\) and \(x \in \mathbb{Z}^d\) let \(P^x_\omega\) be the (unique) Borel probability measure on \(\D\) which turns \(X\) into a time-homogeneous Markov chain with initial value \(x\) and transition kernel \(\omega\), i.e. 
\[
P_\omega^x(X_0 = x) = 1, \quad P_\omega^x (X_n = y + e_k | X_{n - 1} = y) = \omega(y,e_k), \quad z \in \mathbb{Z}^d, k = 1, \dots, 2d.
\]
The coordinate process \(X\) is typically referred to as \emph{the walk} and the law \(P^x_\omega\) is called the \emph{quenched law of the walk}.
An environment \(\omega \in \Omega\) is called \emph{balanced} if for all \(z \in \mathbb{Z}^d\) and \(k = 1, \dots, d\)
\[
\omega(z, e_k) = \omega(z, - e_k).
\]
The set of balanced environments is denoted by \(\mathsf{B}\). % In the following we use the notation: \([n] \triangleq \{0, \dots, n\}\).
For \(n \in \N\) we set \(\mathcal{F}_n \triangleq \sigma (X_m, m \in [n])\), where \([n] \triangleq \{0, \dots, n\}\).
In the following all terms such as \emph{martingale, stopping time}, etc., refer to \((\mathcal{F}_n)_{n \in \N}\) as filtration. 

\begin{remark}
The Markov property of the walk yields an intuitive characterization for balanced environments. Namely, \(X\) is a \(P_\omega^x\)-martingale for  all \(x  \in \mathbb{Z}^d\) if and only if \(\omega \in \mathsf{B}\). 
\end{remark}

We say that \(\omega\in \Omega\) is \emph{genuinely \(d\)-dimensional} if for every \(k = 1, \dots, 2d\) there exists a \(z \in \mathbb{Z}^d\) such that 
\(
\omega (z, e_k) > 0.
\)
We denote the set of all genuinely \(d\)-dimensional environments by \(\mathsf{G}\).

\begin{example}\label{ex: gdd}
	An example for an environment measure \(P\) with \(P(\mathsf{B} \cap \mathsf{G}) = 1\) is the following:
	\[
	P \Big(\omega  \in \Omega \colon \omega (0, e_i) = \omega(0, - e_i) = \frac{1}{2}\Big) = \frac{1}{d}, \quad i = 1, \dots, d.
	\]
	In this case the environment chooses uniformly at random one of the \(\pm e_i\) directions, see \Cref{fig: ex}.
	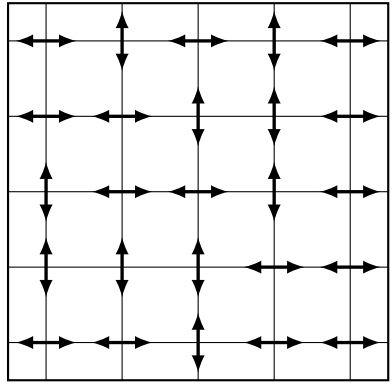
\begin{SCfigure}\label{fig: ex} 
		\renewcommand\thefigure{A}
		\caption{An illustration of \Cref{ex: gdd} restricted to a small box.}
		\begin{tikzpicture}
		\draw[thick] (0,0) rectangle (5,5);
		\draw (0, 0.5) -- (5, 0.5);
		\draw (0, 1.5) -- (5, 1.5);
		\draw (0, 2.5) -- (5, 2.5);
		\draw (0, 3.5) -- (5, 3.5);
		\draw (0, 4.5) -- (5, 4.5);
		\draw (0.5, 0) -- (0.5, 5);
		\draw (1.5, 0) -- (1.5, 5);
		\draw (2.5, 0) -- (2.5, 5);
		\draw (3.5, 0) -- (3.5, 5);
		\draw (4.5, 0) -- (4.5, 5);
		\draw[line width=1.25pt, latex-latex] (0.1, 0.5) -- (0.9, 0.5);
		\draw[line width=1.25pt, latex-latex] (0.1, 3.5) -- (0.9, 3.5);
		\draw[line width=1.25pt, latex-latex] (0.1, 4.5) -- (0.9, 4.5);
		\draw[line width=1.25pt, latex-latex] (1.1, 0.5) -- (1.9, 0.5);
		\draw[line width=1.25pt, latex-latex] (1.1, 2.5) -- (1.9, 2.5);
		\draw[line width=1.25pt, latex-latex] (1.1, 3.5) -- (1.9, 3.5);
		\draw[line width=1.25pt, latex-latex] (2.1, 2.5) -- (2.9, 2.5);
		\draw[line width=1.25pt, latex-latex] (2.1, 4.5) -- (2.9, 4.5);
		\draw[line width=1.25pt, latex-latex] (3.1, 1.5) -- (3.9, 1.5);
		\draw[line width=1.25pt, latex-latex] (3.1, 0.5) -- (3.9, 0.5);
		\draw[line width=1.25pt, latex-latex] (4.1, 0.5) -- (4.9, 0.5);
		\draw[line width=1.25pt, latex-latex] (4.1, 1.5) -- (4.9, 1.5);
		\draw[line width=1.25pt, latex-latex] (4.1, 2.5) -- (4.9, 2.5);
		\draw[line width=1.25pt, latex-latex] (4.1, 3.5) -- (4.9, 3.5);
		\draw[line width=1.25pt, latex-latex] (4.1, 4.5) -- (4.9, 4.5);
		\draw[line width=1.25pt, latex-latex] (0.5,1.1) -- (0.5, 1.9);
		\draw[line width=1.25pt, latex-latex] (0.5,2.1) -- (0.5, 2.9);
		\draw[line width=1.25pt, latex-latex] (1.5,1.1) -- (1.5, 1.9);
		\draw[line width=1.25pt, latex-latex] (1.5,4.1) -- (1.5, 4.9);
		\draw[line width=1.25pt, latex-latex] (2.5,1.1) -- (2.5, 1.9);
		\draw[line width=1.25pt, latex-latex] (2.5,0.1) -- (2.5, 0.9);
		\draw[line width=1.25pt, latex-latex] (2.5,3.1) -- (2.5, 3.9);
		\draw[line width=1.25pt, latex-latex] (3.5,2.1) -- (3.5, 2.9);
		\draw[line width=1.25pt, latex-latex] (3.5,3.1) -- (3.5, 3.9);
		\draw[line width=1.25pt, latex-latex] (3.5,4.1) -- (3.5, 4.9);
		\end{tikzpicture}
		
	\end{SCfigure}
\end{example}

For a finite set \(S \subset \mathbb{Z}^d\) and \(N \in \N\) we say that a function \(u \colon \overline{S} \times [N + 1] \to \mathbb{R}\) is \emph{\(\omega\)-caloric on \(S \times [N]\)} if for every \((x, m) \in S\times [N]\)
\[
u(x, m) = E^x_\omega \big[ u(X_1, 1 + m) \big] = \sum_{k = 1}^{2d} \omega(x, e_k) u(x + e_k, 1 + m).
\]
The following simple observation provides a probabilistic interpretation for the definition of a caloric function.
\begin{lemma}\label{lem: caloric mg}
	Let \(\omega \in \mathsf{B}\) and \(u \colon \overline{S} \times [N + 1] \to \mathbb{R}\). Set
	\[
	\tau_m \triangleq \inf (n \in \N \colon (X_n, n + m) \not \in S \times [N]), \quad m \in [N].
	\]
	The following are equivalent:
	\begin{enumerate}
		\item[\textup{(a)}]
		\(u\) is \(\omega\)-caloric.
		\item[\textup{(b)}]
		For all \((x, m) \in S\times [N]\) the process \((u(X_{n \wedge \tau_m}, n \wedge \tau_m + m))_{n \in \mathbb{Z}_+}\) is a \(P_\omega^x\)-martingale.
	\end{enumerate}
\end{lemma}
\begin{proof}
	The  implication  (b) \(\Rightarrow\) (a) follows from the fact that martingales have constant expectation and \(P^x_\omega\)-a.s. \(\tau_m \geq 1\). 
	For the converse implication, assume that (a) holds and let \(n \in \N\). The Markov property of the walk yields that \(P^x_\omega\)-a.s. on \(\{n < \tau_m\} = \{n + 1 \leq \tau_m\} \in \mathcal{F}_n\)
	\begin{align*}
	E_\omega^x \big[ u(X_{(n + 1) \wedge \tau_m}, (n + 1) \wedge \tau_m + m) \big| \mathcal{F}_n \big] &= E_\omega^x \big[ u(X_{n + 1}, n + 1 + m) \big| \mathcal{F}_n \big]
	\\&= E^{X_n}_\omega \big[ u(X_1, n + 1 + m) \big] 
	\\&= u(X_n, n + m).
	\end{align*}
	Since on \(\{\tau_n \leq n\}\) there is nothing to show,	we conclude that (b) holds. 
\end{proof}

Finally, let us end this section with technical notation: For \(x = (x_1, \dots, x_d) \in \mathbb{R}^d\) define the usual norms:
\begin{align*}
\|x\|_1 \triangleq \sum_{k = 1}^d |x_i|, \qquad \|x\|_2 \triangleq \Big(\sum_{k = 1}^d x^2_k\Big)^\frac{1}{2}, \qquad \|x\|_\infty \triangleq \max_{k = 1, \dots, d} |x_i|.
\end{align*}
For \(R > 0\) and \(y \in \mathbb{R}^d\) let
\begin{align*}
\B_R (y) &\triangleq \big\{x \in \mathbb{R}^d \colon \|x - y\|_2 < R\big\}, \quad B_R \triangleq \B_R \cap \mathbb{Z}^d.
\end{align*}
We also write \(\B_R \triangleq \B_R (0)\) and \(B_R \triangleq B_R (0)\). 
For a set \(G \subset \mathbb{Z}^d\), we define its  discrete boundary by
\[
\partial G \triangleq \big\{x \in \mathbb{Z}^d \backslash G \colon \exists y \in G, \|x - y\|_\infty = 1\big\}.
\]
Furthermore, we set 
\begin{align*}
O_R &\triangleq \big\{x \in B_R \colon \|x - y\|_\infty = 1 \Rightarrow y \in B_R\big\}.
\end{align*}
In case \(R > 1\), \(O_R\) is the biggest subset of \(B_R\) such that \(\overline{O}_R \triangleq O_R \cup \partial O_R = B_R\).
For a space-time point \(\hat{x} = (x, t) \in \mathbb{R}^d \times \mathbb{R}_+\) we define the continuous and discrete parabolic cylinder with radius \(R > 0\) and center \(\hat{x}\) by
\[
\K_R (\hat{x}) \triangleq \B_R (x) \times [t, t + R^2) \subset \mathbb{R}^d \times \mathbb{R}_+, \quad K_R (\hat{x}) \triangleq \K_R (\hat{x}) \cap (\mathbb{Z}^d \times \N).
\]
We also set \(\K_R\triangleq \K_R(0), K_R \triangleq K_R(0)\) and 
\begin{align*}
\pp \K_R&\triangleq \big(\partial \B_R \times (0, R^2] \big) \cup \big( \B_R \times \{R^2\}\big),\\
\pp K_R &\triangleq \big(\partial B_R \times [\lceil R^2\rceil] \big) \cup \big(B_R \times \{\lceil R^2 \rceil\}\big), % \triangleq \pl K_R \cup \pt K_R,
\end{align*}
and 
\(
\oK_R \triangleq \K_R \cup \pp \K_R, \ok_R \triangleq K_R \cup \pp K_R.
\)
Here, \(\partial \B_R\) refers to the boundary of \(\B_R\) in \(\mathbb{R}^d\).
We also define
\begin{align*}
Q_R &\triangleq O_R \times [\lfloor R^2\rfloor- 1],\qquad
\pp Q_R \triangleq \big(\partial O_R \times [\lfloor R^2 \rfloor ] \big) \cup \big(O_R \times \{\lfloor R^2 \rfloor\}\big).
\end{align*}
Moreover, we set
\[
K^-_R \triangleq \big(\B_R \times (0, R^2)\big) \cap \big(\mathbb{Z}^d \times \N\big), \qquad K^+_R \triangleq \big(\B_R \times (2 R^2, 3 R^3)\big) \cap \big( \mathbb{Z}^d \times \N\big).
\]
To capture parities, we define 
\begin{align*}
\Theta^{o/e} (G) &\triangleq \big\{(x, t) \in G \colon \|x\|_1 + t  \text{ is odd/even}\big\}, \quad G \subseteq \mathbb{R}^d \times \N.
\end{align*}
\begin{convention}
	Without explicitly mentioning it, all constants might depend on the measure \(P\) and the dimension \(d\). Moreover, constants might change from line to line. We denote a generic positive constant by \(\c\).
\end{convention}

\subsection{Main Results}
Throughout this chapter, we impose the following:
\begin{ass}\label{SA: PHI}
	\(P(\mathsf{B} \cap \mathsf{G}) = 1\).
\end{ass}
We recall the invariance principle from \cite{Berger2014}:
\begin{theorem}\label{theo: IP} \textup{(\cite[Theorem 1.1]{Berger2014})}
	The quenched invariance principle holds with a deterministic diagonal covariance matrix \(\a\), that is for \(P\)-a.a. \(\omega \in \Omega\) as \(N \to \infty\)
	the law of the continuous \(\mathbb{R}^d\)-valued process
	\[
	\frac{1}{\sqrt{N}} X_{\lfloor tN\rfloor} + \frac{tN - \lfloor tN\rfloor}{\sqrt{N}} \big(X_{\lfloor tN\rfloor + 1} - X_{\lfloor tN\rfloor} \big), \quad t \in \mathbb{R}_+, 
	\]
	under \(P_\omega^x\) converges weakly (on \(C(\mathbb{R}_+, \mathbb{R}^d)\) endowed with the local uniform topology) to the law of a Brownian motion with covariance \(\a\) starting at \(x\).
\end{theorem}
For \(a \in (\sqrt{3}, 2]\), let \(H_a \in (0, \infty)\) be the following Harnack constant for Brownian motion: For every non-negative solution \(u\) to the (backward) heat equation
\[
\frac{du}{dt} + \frac{1}{2} \sum_{i, j = 1}^d \a_{ij} \frac{d^2 u}{dx_i dx_j} = 0
\]
in \(\K_{a R}\) it holds that
\[
\sup_{\B_R \times (2 R^2, 3 R^2)} u \leq H_a \inf_{\B_R \times (0, R^2)} u, 
\]
see \cite[Theorem 1]{doi:10.1002/cpa.3160170106}.
The following \emph{parabolic Harnack inequality (PHI)} is our main result.
\begin{theorem}\label{theo: PHI}
	Fix \(\varepsilon \in (0, 2 - \sqrt{3}), \xi \in (0, \frac{1}{5})\) and \(\mathfrak{w} > 1\).
	There are two constants \(R^*, \delta > 0\) such that for all \(R \geq R^*\) there exists a set \(G \in \mathcal{F}\) such that 
	\(
	P(G) \geq 1 - e^{- R^\delta}
	\)
	and for every \(\omega \in G\), \(p \in \{o, e\}\) and every non-negative \(\omega\)-caloric function \(u\) on \(\overline{K}_{2R}\) satisfying
	\begin{align}\label{eq: growth cond}
	\max_{\Theta^p (\ok_{2R})} u \leq \mathfrak{w}^{R^{2 - \xi}} \min_{\Theta^p (\ok_{2R})} u,
	\end{align}
	it holds that
	\begin{align}\label{eq: PHI}
	\max_{\Theta^p (K^+_R)} u \leq  \frac{(1 + 3 \varepsilon) H_{2 - \varepsilon}}{(1 -\varepsilon)^2} \min_{\Theta^p  (K^-_R)} u.
	\end{align}
\end{theorem}
	\begin{SCfigure} \label{fig: ex2}
	\begin{tikzpicture}
	\caption{An example for the necessity of the growth condition.}
	\draw[thick] (0,0) rectangle (7.5,7.5);
	\draw (0, 0.25) -- (7.5, 0.25);
	\draw (0, 0.75) -- (7.5, 0.75);
	\draw (0, 1.25) -- (7.5, 1.25);
	\draw (0, 1.75) -- (7.5, 1.75);
	\draw (0, 2.25) -- (7.5, 2.25);
	\draw (0, 2.75) -- (7.5, 2.75);
	\draw (0, 3.25) -- (7.5, 3.25);	
	\draw (0, 3.75) -- (7.5, 3.75);
	\draw (0, 4.25) -- (7.5, 4.25);
	\draw (0, 4.75) -- (7.5, 4.75);
	\draw (0, 5.25) -- (7.5, 5.25);
	\draw (0, 5.75) -- (7.5, 5.75);
	\draw (0, 6.25) -- (7.5, 6.25);
	\draw (0, 6.75) -- (7.5, 6.75);
	\draw (0, 7.25) -- (7.5, 7.25);
	\draw (0.25, 0) -- (0.25, 7.5);
	\draw (0.75, 0) -- (0.75, 7.5);
	\draw (1.25, 0) -- (1.25, 7.5);
	\draw (1.75, 0) -- (1.75, 7.5);
	\draw (2.25, 0) -- (2.25, 7.5);
	\draw (2.75, 0) -- (2.75, 7.5);
	\draw (3.25, 0) -- (3.25, 7.5);	
	\draw (3.75, 0) -- (3.75, 7.5);
	\draw (4.25, 0) -- (4.25, 7.5);
	\draw (4.75, 0) -- (4.75, 7.5);
	\draw (5.25, 0) -- (5.25, 7.5);
	\draw (5.75, 0) -- (5.75, 7.5);
	\draw (6.25, 0) -- (6.25, 7.5);
	\draw (6.75, 0) -- (6.75, 7.5);
	\draw (7.25, 0) -- (7.25, 7.5);
	
	%red lines
	
	%vertikal
	\draw[red, line width=1.75pt] (0.25,0) -- (0.25, 2.25);
	\draw[red, line width=1.75pt] (0.25,3.75) -- (0.25, 5.25);
	\draw[red, line width=1.75pt] (0.75,0) -- (0.75, 1.75);
	\draw[red, line width=1.75pt] (0.75,2.25) -- (0.75, 4.25);
	\draw[red, line width=1.75pt] (0.75,5.25) -- (0.75, 7.5);		
	\draw[red, line width=1.75pt] (1.25,0.25) -- (1.25, 2.25);
	\draw[red, line width=1.75pt] (1.25,3.75) -- (1.25, 5.75);
	\draw[red, line width=1.75pt] (1.75,1.25) -- (1.75, 3.75);
	\draw[red, line width=1.75pt] (1.75,5.25) -- (1.75, 7.25);
	\draw[red, line width=1.75pt] (2.25,0) -- (2.25, 2.25);
	\draw[red, line width=1.75pt] (2.25,5.25) -- (2.25, 7.5);
	\draw[red, line width=1.75pt] (2.75,1.25) -- (2.75, 5.75);		
	\draw[red, line width=1.75pt] (3.25,1.25) -- (3.25, 4.25);
	\draw[red, line width=1.75pt] (3.75,0.25) -- (3.75, 2.75);
	\draw[red, line width=1.75pt] (3.75,3.75) -- (3.75, 5.25);
	\draw[red, line width=1.75pt] (3.75,5.75) -- (3.75, 6.75);
	\draw[red, line width=1.75pt] (4.25,0) -- (4.25, 1.75);
	\draw[red, line width=1.75pt] (4.25,2.75) -- (4.25, 6.25);
	\draw[red, line width=1.75pt] (4.75,1.25) -- (4.75, 2.25);
	\draw[red, line width=1.75pt] (4.75,3.75) -- (4.75, 7.5);
	\draw[red, line width=1.75pt] (5.25,0) -- (5.25, 2.75);
	\draw[red, line width=1.75pt] (5.25,5.75) -- (5.25, 7.5);
	\draw[red, line width=1.75pt] (5.75,2.75) -- (5.75, 7.25);
	\draw[red, line width=1.75pt] (6.25, 3.75) -- (6.25, 5.25);
	\draw[red, line width=1.75pt] (6.75,4.75) -- (6.75, 5.75);
	%horizontal
	\draw[red, line width=1.75pt] (0.75,0.25) -- (4.25, 0.25);
	\draw[red, line width=1.75pt] (1.25,1.25) -- (3.75, 1.25);
	\draw[red, line width=1.75pt] (4.25,1.25) -- (5.25, 1.25);
	\draw[red, line width=1.75pt] (0.25,1.75) -- (1.25, 1.75);
	\draw[red, line width=1.75pt] (3.25,1.75) -- (4.75, 1.75);
	\draw[red, line width=1.75pt] (0,2.25) -- (2.75, 2.25);
	\draw[red, line width=1.75pt] (3.25,2.25) -- (4.75, 2.25);
	\draw[red, line width=1.75pt] (3.25,2.75) -- (7.5, 2.75);
	\draw[red, line width=1.75pt] (1.75,3.25) -- (3.25, 3.25);
	\draw[red, line width=1.75pt] (0,3.75) -- (2.75, 3.75);
	\draw[red, line width=1.75pt] (3.25,3.75) -- (7.5, 3.75);
	\draw[red, line width=1.75pt] (0.25,4.25) -- (1.25, 4.25);
	\draw[red, line width=1.75pt] (2.75,4.25) -- (3.75, 4.25);
	\draw[red, line width=1.75pt] (3.75,4.75) -- (7.5, 4.75);
	\draw[red, line width=1.75pt] (0,5.25) -- (6.75, 5.25);
	\draw[red, line width=1.75pt] (0.75,5.75) -- (1.75, 5.75);
	\draw[red, line width=1.75pt] (2.25,5.75) -- (4.25, 5.75);
	\draw[red, line width=1.75pt] (4.75,5.75) -- (7.5, 5.75);
	\draw[red, line width=1.75pt] (3.75,6.25) -- (5.25, 6.25);
	\draw[red, line width=1.75pt] (1.75,6.75) -- (5.75, 6.75);
	\draw[red, line width=1.75pt] (0.75,7.25) -- (2.25, 7.25);
	\draw[red, line width=1.75pt] (5.25,7.25) -- (7.5, 7.25);
	%% Vertikal
	%erste spalte
	\draw[<->, line width=1pt] (0.05, 2.25) -- (0.45, 2.25);
	\draw[<->, line width=1pt] (0.05, 2.75) -- (0.45, 2.75);
	\draw[<->, line width=1pt] (0.05, 3.25) -- (0.45, 3.25);
	\draw[<->, line width=1pt] (0.05, 3.75) -- (0.45, 3.75);
	\draw[<->, line width=1pt] (0.05, 5.25) -- (0.45, 5.25);
	\draw[<->, line width=1pt] (0.05, 5.75) -- (0.45, 5.75);
	\draw[<->, line width=1pt] (0.05, 7.25) -- (0.45, 7.25);
	%zweite splate
	\draw[<->, line width=1pt] (0.55, 1.75) -- (0.95, 1.75);
	\draw[<->, line width=1pt] (0.55, 2.25) -- (0.95, 2.25);
	\draw[<->, line width=1pt] (0.55, 4.25) -- (0.95, 4.25);
	\draw[<->, line width=1pt] (0.55, 4.75) -- (0.95, 4.75);
	\draw[<->, line width=1pt] (0.55, 5.25) -- (0.95, 5.25);
	%dritte spalte
	\draw[<->, line width=1pt] (1.05, 0.25) -- (1.45, 0.25);
	\draw[<->, line width=1pt] (1.05, 2.25) -- (1.45, 2.25);
	\draw[<->, line width=1pt] (1.05, 2.75) -- (1.45, 2.75);
	\draw[<->, line width=1pt] (1.05, 3.75) -- (1.45, 3.75);
	\draw[<->, line width=1pt] (1.05, 5.75) -- (1.45, 5.75);
	\draw[<->, line width=1pt] (1.05, 6.25) -- (1.45, 6.25);
	\draw[<->, line width=1pt] (1.05, 6.75) -- (1.45, 6.75);
	\draw[<->, line width=1pt] (1.05, 7.25) -- (1.45, 7.25);
	%vierte spalte
	\draw[<->, line width=1pt] (1.55, 0.25) -- (1.95, 0.25);
	\draw[<->, line width=1pt] (1.55, 1.25) -- (1.95, 1.25);
	\draw[<->, line width=1pt] (1.55, 3.75) -- (1.95, 3.75);
	\draw[<->, line width=1pt] (1.55, 4.25) -- (1.95, 4.25);
	\draw[<->, line width=1pt] (1.55, 5.25) -- (1.95, 5.25);
	\draw[<->, line width=1pt] (1.55, 7.25) -- (1.95, 7.25);
	%funfte zeile
	\draw[<->, line width=1pt] (2.05, 2.25) -- (2.45, 2.25);
	\draw[<->, line width=1pt] (2.05, 2.75) -- (2.45, 2.75);
	\draw[<->, line width=1pt] (2.05, 3.25) -- (2.45, 3.25);
	\draw[<->, line width=1pt] (2.05, 3.75) -- (2.45, 3.75);
	\draw[<->, line width=1pt] (2.05, 4.25) -- (2.45, 4.25);
	\draw[<->, line width=1pt] (2.05, 5.25) -- (2.45, 5.25);
	\draw[<->, line width=1pt] (2.05, 6.75) -- (2.45, 6.75);
	%sechste zeile
	\draw[<->, line width=1pt] (2.55, 0.25) -- (2.95, 0.25);
	\draw[<->, line width=1pt] (2.55, 0.75) -- (2.95, 0.75);
	\draw[<->, line width=1pt] (2.55, 1.25) -- (2.95, 1.25);
	\draw[<->, line width=1pt] (2.55, 3.25) -- (2.95, 3.25);
	\draw[<->, line width=1pt] (2.55, 5.75) -- (2.95, 5.75);
	\draw[<->, line width=1pt] (2.55, 6.75) -- (2.95, 6.75);
	%siebte zeile
	\draw[<->, line width=1pt] (3.05, 0.25) -- (3.45, 0.25);
	\draw[<->, line width=1pt] (3.05, 1.25) -- (3.45, 1.25);
	\draw[<->, line width=1pt] (3.05, 4.25) -- (3.45, 4.25);
	\draw[<->, line width=1pt] (3.05, 4.75) -- (3.45, 4.75);
	\draw[<->, line width=1pt] (3.05, 5.25) -- (3.45, 5.25);
	\draw[<->, line width=1pt] (3.05, 5.75) -- (3.45, 5.75);
	\draw[<->, line width=1pt] (3.05, 6.25) -- (3.45, 6.25);
	\draw[<->, line width=1pt] (3.05, 6.75) -- (3.45, 6.75);
	\draw[<->, line width=1pt] (3.05, 7.25) -- (3.45, 7.25);
	%achte zeile
	\draw[<->, line width=1pt] (3.55, 0.25) -- (3.95, 0.25);
	\draw[<->, line width=1pt] (3.55, 1.75) -- (3.95, 1.75);
	\draw[<->, line width=1pt] (3.55, 2.75) -- (3.95, 2.75);
	\draw[<->, line width=1pt] (3.55, 3.25) -- (3.95, 3.25);
	\draw[<->, line width=1pt] (3.55, 3.75) -- (3.95, 3.75);
	\draw[<->, line width=1pt] (3.55, 5.25) -- (3.95, 5.25);
	\draw[<->, line width=1pt] (3.55, 5.75) -- (3.95, 5.75);
	\draw[<->, line width=1pt] (3.55, 6.75) -- (3.95, 6.75);
	%neunte zeile
	\draw[<->, line width=1pt] (4.05, 1.75) -- (4.45, 1.75);
	\draw[<->, line width=1pt] (4.05, 2.25) -- (4.45, 2.25);
	\draw[<->, line width=1pt] (4.05, 2.75) -- (4.45, 2.75);
	\draw[<->, line width=1pt] (4.05, 4.75) -- (4.45, 4.75);
	\draw[<->, line width=1pt] (4.05, 6.25) -- (4.45, 6.25);
	\draw[<->, line width=1pt] (4.05, 6.75) -- (4.45, 6.75);
	%zehnte zeile
	\draw[<->, line width=1pt] (4.55, 0.25) -- (4.95, 0.25);
	\draw[<->, line width=1pt] (4.55, 0.75) -- (4.95, 0.75);
	\draw[<->, line width=1pt] (4.55, 1.25) -- (4.95, 1.25);
	\draw[<->, line width=1pt] (4.55, 2.25) -- (4.95, 2.25);
	\draw[<->, line width=1pt] (4.55, 2.75) -- (4.95, 2.75);
	\draw[<->, line width=1pt] (4.55, 3.25) -- (4.95, 3.25);
	\draw[<->, line width=1pt] (4.55, 3.75) -- (4.95, 3.75);
	\draw[<->, line width=1pt] (4.55, 5.25) -- (4.95, 5.25);
	\draw[<->, line width=1pt] (4.55, 6.25) -- (4.95, 6.25);
	%elfte zeile
	\draw[<->, line width=1pt] (5.05, 2.75) -- (5.45, 2.75);
	\draw[<->, line width=1pt] (5.05, 3.25) -- (5.45, 3.25);
	\draw[<->, line width=1pt] (5.05, 3.75) -- (5.45, 3.75);
	\draw[<->, line width=1pt] (5.05, 4.75) -- (5.45, 4.75);
	\draw[<->, line width=1pt] (5.05, 5.25) -- (5.45, 5.25);
	\draw[<->, line width=1pt] (5.05, 5.75) -- (5.45, 5.75);
	\draw[<->, line width=1pt] (5.05, 6.75) -- (5.45, 6.75);
	%zwölfte zeile
	\draw[<->, line width=1pt] (5.55, 1.25) -- (5.95, 1.25);
	\draw[<->, line width=1pt] (5.55, 2.25) -- (5.95, 2.25);
	\draw[<->, line width=1pt] (5.55, 2.75) -- (5.95, 2.75);
	\draw[<->, line width=1pt] (5.55, 4.75) -- (5.95, 4.75);
	\draw[<->, line width=1pt] (5.55, 7.25) -- (5.95, 7.25);
	%dreihzehnte zeile
	\draw[<->, line width=1pt] (6.05, 2.25) -- (6.45, 2.25);
	\draw[<->, line width=1pt] (6.05, 2.75) -- (6.45, 2.75);
	\draw[<->, line width=1pt] (6.05, 3.25) -- (6.45, 3.25);
	\draw[<->, line width=1pt] (6.05, 3.75) -- (6.45, 3.75);
	\draw[<->, line width=1pt] (6.05, 5.25) -- (6.45, 5.25);
	\draw[<->, line width=1pt] (6.05, 5.75) -- (6.45, 5.75);
	\draw[<->, line width=1pt] (6.05, 6.25) -- (6.45, 6.25);
	\draw[<->, line width=1pt] (6.05, 6.75) -- (6.45, 6.75);
	\draw[<->, line width=1pt] (6.05, 7.25) -- (6.45, 7.25);
	%vierzehnte zeile
	\draw[<->, line width=1pt] (6.55, 0.75) -- (6.95, 0.75);
	\draw[<->, line width=1pt] (6.55, 1.25) -- (6.95, 1.25);
	\draw[<->, line width=1pt] (6.55, 2.25) -- (6.95, 2.25);
	\draw[<->, line width=1pt] (6.55, 2.75) -- (6.95, 2.75);
	\draw[<->, line width=1pt] (6.55, 3.75) -- (6.95, 3.75);
	\draw[<->, line width=1pt] (6.55, 4.25) -- (6.95, 4.25);
	\draw[<->, line width=1pt] (6.55, 4.75) -- (6.95, 4.75);
	\draw[<->, line width=1pt] (6.55, 5.75) -- (6.95, 5.75);
	\draw[<->, line width=1pt] (6.55, 6.25) -- (6.95, 6.25);
	\draw[<->, line width=1pt] (6.55, 7.25) -- (6.95, 7.25);
	%funfzehnte zeile
	\draw[<->, line width=1pt] (7.05, 0.25) -- (7.45, 0.25);
	\draw[<->, line width=1pt] (7.05, 1.75) -- (7.45, 1.75);
	\draw[<->, line width=1pt] (7.05, 2.75) -- (7.45, 2.75);
	\draw[<->, line width=1pt] (7.05, 3.25) -- (7.45, 3.25);
	\draw[<->, line width=1pt] (7.05, 3.75) -- (7.45, 3.75);
	\draw[<->, line width=1pt] (7.05, 4.25) -- (7.45, 4.25);
	\draw[<->, line width=1pt] (7.05, 4.75) -- (7.45, 4.75);
	\draw[<->, line width=1pt] (7.05, 5.75) -- (7.45, 5.75);
	\draw[<->, line width=1pt] (7.05, 6.25) -- (7.45, 6.25);
	\draw[<->, line width=1pt] (7.05, 7.25) -- (7.45, 7.25);
	
	%erste spalte
	\draw[<->, line width=1pt] (0.25, 0.05) -- (0.25, 0.45);
	\draw[<->, line width=1pt] (0.25, 0.55) -- (0.25, 0.95);
	\draw[<->, line width=1pt] (0.25, 1.05) -- (0.25, 1.45);
	\draw[<->, line width=1pt] (0.25, 1.55) -- (0.25, 1.95);
	\draw[<->, line width=1pt] (0.25, 4.05) -- (0.25, 4.45);
	\draw[<->, line width=1pt] (0.25, 4.55) -- (0.25, 4.95);
	\draw[<->, line width=1pt] (0.25, 6.05) -- (0.25, 6.45);
	\draw[<->, line width=1pt] (0.25, 6.55) -- (0.25, 6.95);
	
	%zweite spalte
	\draw[<->, line width=1pt] (0.75, 0.05) -- (0.75, 0.45);
	\draw[<->, line width=1pt] (0.75, 0.55) -- (0.75, 0.95);
	\draw[<->, line width=1pt] (0.75, 1.05) -- (0.75, 1.45);
	\draw[<->, line width=1pt] (0.75, 2.55) -- (0.75, 2.95);
	\draw[<->, line width=1pt] (0.75, 3.05) -- (0.75, 3.45);
	\draw[<->, line width=1pt] (0.75, 3.55) -- (0.75, 3.95);
	\draw[<->, line width=1pt] (0.75, 5.55) -- (0.75, 5.95);
	\draw[<->, line width=1pt] (0.75, 6.05) -- (0.75, 6.45);
	\draw[<->, line width=1pt] (0.75, 6.55) -- (0.75, 6.95);
	\draw[<->, line width=1pt] (0.75, 7.05) -- (0.75, 7.45);
	
	%dritte spalte
	\draw[<->, line width=1pt] (1.25, 0.55) -- (1.25, 0.95);
	\draw[<->, line width=1pt] (1.25, 1.05) -- (1.25, 1.45);
	\draw[<->, line width=1pt] (1.25, 1.55) -- (1.25, 1.95);
	\draw[<->, line width=1pt] (1.25, 3.05) -- (1.25, 3.45);
	\draw[<->, line width=1pt] (1.25, 4.05) -- (1.25, 4.45);
	\draw[<->, line width=1pt] (1.25, 4.55) -- (1.25, 4.95);
	\draw[<->, line width=1pt] (1.25, 5.05) -- (1.25, 5.45);
	
	%vierte spalte
	\draw[<->, line width=1pt] (1.75, 0.55) -- (1.75, 0.95);	
	\draw[<->, line width=1pt] (1.75, 1.55) -- (1.75, 1.95);	
	\draw[<->, line width=1pt] (1.75, 2.05) -- (1.75, 2.45);	
	\draw[<->, line width=1pt] (1.75, 2.55) -- (1.75, 2.95);	
	\draw[<->, line width=1pt] (1.75, 3.05) -- (1.75, 3.45);
	\draw[<->, line width=1pt] (1.75, 4.55) -- (1.75, 4.95);		
	\draw[<->, line width=1pt] (1.75, 5.55) -- (1.75, 5.95);		
	\draw[<->, line width=1pt] (1.75, 6.05) -- (1.75, 6.45);		
	\draw[<->, line width=1pt] (1.75, 6.55) -- (1.75, 6.95);		
	
	%funfte spalte
	\draw[<->, line width=1pt] (2.25, 0.05) -- (2.25, 0.45);										
	\draw[<->, line width=1pt] (2.25, 0.55) -- (2.25, 0.95);		
	\draw[<->, line width=1pt] (2.25, 1.05) -- (2.25, 1.45);		
	\draw[<->, line width=1pt] (2.25, 1.55) -- (2.25, 1.95);												
	\draw[<->, line width=1pt] (2.25, 4.55) -- (2.25, 4.95);		
	\draw[<->, line width=1pt] (2.25, 5.55) -- (2.25, 5.95);			
	\draw[<->, line width=1pt] (2.25, 6.05) -- (2.25, 6.45);		
	\draw[<->, line width=1pt] (2.25, 7.05) -- (2.25, 7.45);

	%sechste spalte
	\draw[<->, line width=1pt] (2.75, 1.55) -- (2.75, 1.95);
	\draw[<->, line width=1pt] (2.75, 2.05) -- (2.75, 2.45);
	\draw[<->, line width=1pt] (2.75, 2.55) -- (2.75, 2.95);
	\draw[<->, line width=1pt] (2.75, 3.55) -- (2.75, 3.95);
	\draw[<->, line width=1pt] (2.75, 4.05) -- (2.75, 4.45);
	\draw[<->, line width=1pt] (2.75, 4.55) -- (2.75, 4.95);
	\draw[<->, line width=1pt] (2.75, 5.05) -- (2.75, 5.45);
	\draw[<->, line width=1pt] (2.75, 6.05) -- (2.75, 6.45);
	\draw[<->, line width=1pt] (2.75, 7.05) -- (2.75, 7.45);
	
	%siebte spalte
	\draw[<->, line width=1pt] (3.25, 0.55) -- (3.25, 0.95);
	\draw[<->, line width=1pt] (3.25, 1.55) -- (3.25, 1.95);
	\draw[<->, line width=1pt] (3.25, 2.05) -- (3.25, 2.45);
	\draw[<->, line width=1pt] (3.25, 2.55) -- (3.25, 2.95);
	\draw[<->, line width=1pt] (3.25, 3.05) -- (3.25, 3.45);
	\draw[<->, line width=1pt] (3.25, 3.55) -- (3.25, 3.95);
	
	%achte spalte
	\draw[<->, line width=1pt] (3.75, 0.55) -- (3.75, 0.95);
	\draw[<->, line width=1pt] (3.75, 1.05) -- (3.75, 1.45);
	\draw[<->, line width=1pt] (3.75, 2.05) -- (3.75, 2.45);
	\draw[<->, line width=1pt] (3.75, 4.05) -- (3.75, 4.45);
	\draw[<->, line width=1pt] (3.75, 4.55) -- (3.75, 4.95);
	\draw[<->, line width=1pt] (3.75, 6.05) -- (3.75, 6.45);
	\draw[<->, line width=1pt] (3.75, 7.05) -- (3.75, 7.45);
	
	%neuete spalte
	\draw[<->, line width=1pt] (4.25, 0.05) -- (4.25, 0.45);
	\draw[<->, line width=1pt] (4.25, 0.55) -- (4.25, 0.95);
	\draw[<->, line width=1pt] (4.25, 1.05) -- (4.25, 1.45);
	\draw[<->, line width=1pt] (4.25, 3.05) -- (4.25, 3.45);
	\draw[<->, line width=1pt] (4.25, 3.55) -- (4.25, 3.95);
	\draw[<->, line width=1pt] (4.25, 4.05) -- (4.25, 4.45);
	\draw[<->, line width=1pt] (4.25, 5.05) -- (4.25, 5.45);
	\draw[<->, line width=1pt] (4.25, 5.55) -- (4.25, 5.95);
	\draw[<->, line width=1pt] (4.25, 7.05) -- (4.25, 7.45);
	
	%zehnte spalte
	\draw[<->, line width=1pt] (4.75, 1.55) -- (4.75, 1.95);
	\draw[<->, line width=1pt] (4.75, 4.05) -- (4.75, 4.45);
	\draw[<->, line width=1pt] (4.75, 4.55) -- (4.75, 4.95);
	\draw[<->, line width=1pt] (4.75, 5.55) -- (4.75, 5.95);
	\draw[<->, line width=1pt] (4.75, 6.55) -- (4.75, 6.95);
	\draw[<->, line width=1pt] (4.75, 7.05) -- (4.75, 7.45);
	
	%elfte spalte
	\draw[<->, line width=1pt] (5.25, 0.05) -- (5.25, 0.45);
	\draw[<->, line width=1pt] (5.25, 0.55) -- (5.25, 0.95);
	\draw[<->, line width=1pt] (5.25, 1.05) -- (5.25, 1.45);
	\draw[<->, line width=1pt] (5.25, 1.55) -- (5.25, 1.95);
	\draw[<->, line width=1pt] (5.25, 2.05) -- (5.25, 2.45);
	\draw[<->, line width=1pt] (5.25, 4.05) -- (5.25, 4.45);
	\draw[<->, line width=1pt] (5.25, 6.05) -- (5.25, 6.45);
	\draw[<->, line width=1pt] (5.25, 7.05) -- (5.25, 7.45);
	
	%zwölfte spalte
	\draw[<->, line width=1pt] (5.75, 0.05) -- (5.75, 0.45);
	\draw[<->, line width=1pt] (5.75, 0.55) -- (5.75, 0.95);
	\draw[<->, line width=1pt] (5.75, 1.55) -- (5.75, 1.95);
	\draw[<->, line width=1pt] (5.75, 3.05) -- (5.75, 3.45);
	\draw[<->, line width=1pt] (5.75, 3.55) -- (5.75, 3.95);
	\draw[<->, line width=1pt] (5.75, 4.05) -- (5.75, 4.45);
	\draw[<->, line width=1pt] (5.75, 5.05) -- (5.75, 5.45);
	\draw[<->, line width=1pt] (5.75, 5.55) -- (5.75, 5.95);
	\draw[<->, line width=1pt] (5.75, 6.05) -- (5.75, 6.45);
	\draw[<->, line width=1pt] (5.75, 6.55) -- (5.75, 6.95);
	
	%dreizehnte spalte
	\draw[<->, line width=1pt] (6.25, 0.05) -- (6.25, 0.45);
	\draw[<->, line width=1pt] (6.25, 0.55) -- (6.25, 0.95);
	\draw[<->, line width=1pt] (6.25, 1.05) -- (6.25, 1.45);
	\draw[<->, line width=1pt] (6.25, 1.55) -- (6.25, 1.95);
	\draw[<->, line width=1pt] (6.25, 4.05) -- (6.25, 4.45);
	\draw[<->, line width=1pt] (6.25, 4.55) -- (6.25, 4.95);
	
	%vierzehnte spalte
	\draw[<->, line width=1pt] (6.75, 0.05) -- (6.75, 0.45);
	\draw[<->, line width=1pt] (6.75, 1.55) -- (6.75, 1.95);
	\draw[<->, line width=1pt] (6.75, 3.05) -- (6.75, 3.45);
	\draw[<->, line width=1pt] (6.75, 5.05) -- (6.75, 5.45);
	\draw[<->, line width=1pt] (6.75, 6.55) -- (6.75, 6.95);
	
	%funfzehnte spalte									
	\draw[<->, line width=1pt] (7.25, 0.55) -- (7.25, 0.95);
	\draw[<->, line width=1pt] (7.25, 1.05) -- (7.25, 1.45);
	\draw[<->, line width=1pt] (7.25, 2.05) -- (7.25, 2.45);
	\draw[<->, line width=1pt] (7.25, 5.05) -- (7.25, 5.45);
	\draw[<->, line width=1pt] (7.25, 6.55) -- (7.25, 6.95);								
	
	%green box lines
	\draw[green, line width = 1.75pt] (4.25,2.75) rectangle (5.75,3.75);
	\end{tikzpicture}
\end{SCfigure}
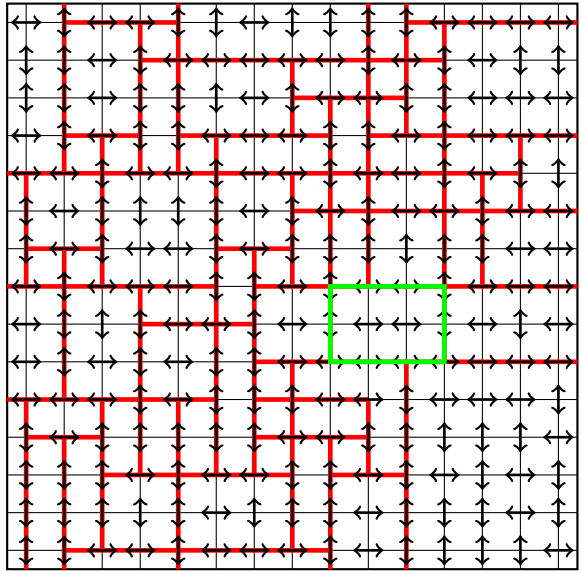

\begin{example} \label{ex: grow needed}
	In the following we provide an example which shows that in non degenerate settings the PHI cannot hold in full generality without a certain growth condition. Let us consider the setting of \Cref{ex: gdd} with \(d = 2\). More precisely, take the environment given by \Cref{fig: ex2}. The red part in \Cref{fig: ex2} is called the \emph{sink}. It is easy to see that once the walk has reached the sink, it cannot exit it. Consequently, by its recursive definition, the values of a caloric function on the sink are not influenced by the values outside of it. 
	For contradiction, assume that the PHI holds for all caloric functions \(u\), i.e. there exists a constant \(C > 0\) independent of \(u\) and \(R\) such that \[\max_{\Theta^p (K^+_R)} u \leq C \min_{\Theta^p (K^+_R)} u.\]
	Denote the points in the green box in \Cref{fig: ex2} by \(x_1\) and \(x_2\).
	Fix \(R\) large enough such that \(2^{R^2} > C\) and take a non-negative caloric function \(u\) on the cylinder with radius \(2R\) which takes the value one on the sink and
	\[
	u (x_1, 4 R^2) \equiv u (x_2, 4 R^2) \triangleq 2^{3 R^2}.
	\]
	Such a caloric function can be defined by recursion. We stress that \(u\) does not satisfy the growth condition \eqref{eq: growth cond}. Using the recursive definition, we note that 
	\[
	\max_{\Theta^p (K^+_R)} u \geq 2^{R^2},
	\]
	and the PHI implies 
	\[
	2^{R^2} \leq C \min_{\Theta^p(K^-_R)} u \leq C,
	\]
	which is a contradiction. We conclude that the PHI does not hold for \(u\).
\end{example}
\begin{remark}
	\begin{enumerate}
		\item[\textup{(i)}]
		The Harnack constant in \Cref{theo: PHI} is optimal in the sense that it can be taken arbitrarily close to \(H_2\). 
		\item[\textup{(ii)}] 
		%The growth condition \eqref{eq: growth cond} is sharp in the sense that there are examples of degenerate settings where the PHI is violated for certain caloric functions satisfying \eqref{eq: growth cond} with \(\xi = 0\), see \Cref{ex: grow needed} below.
		In uniformly elliptic settings the growth condition \eqref{eq: growth cond} is not needed, see \cite{10.1215/S0012-7094-98-09122-0,10.1214/009117905000000440}.
		\item[\textup{(iii)}] Typically PHIs are formulated for forward equations. The time substitution \(t \mapsto 4 R^2 - t\) transforms the PHI for backward equations into a PHI for forward equations.
		\item[\textup{(iv)}] As the following simple example illustrates, it is necessary to compare cylinders of the same parity. Let \(u\) be a solution to the (backward) heat equation for the one-dimensional simple random walk in \(K_R\) with terminal condition 
		\[
		f (x, t) = \begin{cases}1,&|x| + t \text{ odd},\\
		0,&|x| + t \text{ even}.
		\end{cases}
		\]
		The recursive definition of a caloric function shows that
		\(
		u = f, 
		\)
		which implies
		\[
		\max_{K^+} u = 1, \quad \min_{K^-} u = 0.
		\]
		Clearly, \eqref{eq: PHI} does not hold when \(\Theta^p(K^+_R)\) and \(\Theta^p (K^-_R)\) are replaced by \(K^+_R\) and \(K^-_R\), respectively.
		An alternative strategy to deal with the parity issue is to formulate the PHI for
		\[
		\hat{u} (x, n) \triangleq u(x, n + 1) + u(x, n)
		\]
		instead of \(u\).
		This has been done in \cite{hambly2009} for random walks on percolation clusters.
	\end{enumerate}
\end{remark}

The proof of \Cref{theo: PHI} is given in \Cref{sec: Pf PHI}. It borrows arguments by Fabes and Stroock \cite{Fabes1989}. 
A version of the Fabes--Stroock argument has also been used in \cite{Berger18} to prove an elliptic Harnack inequality (EHI) under Standing Assumption \ref{SA: PHI}. Some ideas in the proof of \Cref{theo: PHI} are borrowed from \cite{Berger18}.

Harnack inequalities are important tools in the study of path properties of RWRE. As explained in the introduction, we think that \Cref{theo: PHI} might be the first step in direction of a local limit theorem. The EHI from \cite{Berger18} can for instance be used to prove transience of the RWBRE for \(d \geq 3\) in genuinely \(d\)-dimensional environments.  
Since this result seems to be new, we provide a statement and a  proof, which is similar to those of \cite[Theorem 3.3.22]{Zeitouni2004} and given in \Cref{sec:pf trans}. 
\begin{theorem}\label{theo:trans}
	When \(d \geq 3\), the RWRE is transient for \(P\)-a.a. environments.
\end{theorem}

One key tool for the proof of \Cref{theo: PHI} is the following oscillation inequality, which can be seen as a parabolic version of \cite[Theorem 4.1]{Berger18}.
\begin{theorem}\label{theo: OI}
	There are constants \(R', \delta > 0, \zeta > 1\) and \(\gamma \in (0, 1)\) such that for every \(R \geq R'\) there exists a set \(G \in \mathcal{F}\) such that
	\(
	P(G) \geq 1 - e^{- R^{\delta}},
	\)
	and for every \(\omega \in G\) and every \(\omega\)-caloric function \(u\) on \(\overline{K}_{\zeta R}\) it holds that 
	\[
	\underset{\Theta^p(K_R)}{\osc}\ u \leq \gamma \underset{\Theta^p(K_{\zeta R})}{\osc}\ u, \qquad p \in \{o, e\},
	\]
	where 
	\[
	\underset{G}{\osc}\ u \triangleq \max_{G} u - \min_Gu, \qquad G \subset \mathbb{Z}^d \times \N \text{ finite}.
	\]
\end{theorem}
The proof of \Cref{theo: OI} is based on the explicit construction of a coupling. 
Let us sketch the idea: Suppose that \(\hat{X}\) and \(\hat{Y}\) are coupled space-time walks in a fixed environment \(\omega \in \mathsf{B}\) such that the probability that \(\hat{X}\) and \(\hat{Y}\) leave a subcylinder of \(K_{\zeta R}\) in the same point is bounded from below by a uniform constant \(1 - \gamma > 0\). Denote the corresponding hitting times of the boundary by \(T\) and \(S\), respectively. Then, if \(\hat{X}\) starts, say, at \(\hat{x} \in \Theta^p(K_R)\) and \(\hat{Y}\) starts, say, at \(\hat{y} \in \Theta^p (K_R)\), the optional stopping theorem yields for every \(\omega\)-caloric function \(u\) that
\begin{align*}
u(\hat{x}) - u(\hat{y}) &= E \big[ u(\hat{X}_T) - u(\hat{Y}_S)\big] \\&= E\big[ (u(\hat{X}_T) -  u(\hat{Y}_S) ) \1_{\{\hat{X}_T \not = \hat{Y}_S\}}\big] \\&\leq \underset{\Theta^p (K_{\zeta R})}\osc u\ P(\hat{X}_T \not = \hat{Y}_S) \\& \leq \gamma \underset{\Theta^p (K_{\zeta R})} \osc u.
\end{align*}
This implies the oscillation inequality.

Another key tool for the proof of \Cref{theo: PHI} is the following quantitative homogenization estimate:
Let \(F \colon \oK_1 \to \mathbb{R}\) be a continuous function of class \(C^{2, 3}\) on \(\K_1\) such that 
\begin{align}\label{eq: CP}
\frac{dF}{dt} + \frac{1}{2} \sum_{i, j = 1}^d \a_{ij} \frac{d^2 F}{d x_i d x_j} = 0\quad \textup{ on } \K_1.
\end{align}
The existence of \(F\) is classical.
We define \(F_R \colon \K_R \to \mathbb{R}\) by 
\[
F_R(x, t) \triangleq F \left( \tfrac{x}{R}, \tfrac{t}{R^2}\right), \quad (x, t) \in \K_R.
\]
For \(u \colon \overline{Q}_R \to \mathbb{R}\), we define 
\[
(\mathcal{L}_\omega u) (y, s) \triangleq \sum_{i = 1}^{2d} \omega(y, e_i) u(y + e_i,1 + s) - u(y, s), \quad (y, s) \in Q_R.
\]
For  \(\omega \in \Omega\) let \(G_{\omega} \colon \overline{Q}_R \to \mathbb{R}\) be such that 
\[\begin{cases}
\mathcal{L}_\omega G_{\omega} = 0,& \text{ on } Q_R,\\
G_{\omega} = F_R, &\text{ on } \partial^p Q_R.\end{cases}
\]
It is easy to see that \(G_\omega\) exists in a unique manner: Indeed, first set \(G_\omega = F_R\) on \(\pp Q_R\) and then use \(\mathcal{L}_\omega G_\omega = 0\) on \(Q_R\) to define \(G_\omega\) recursively.
The following quantitative estimate is a parabolic version of \cite[Theorem 1.4]{Berger18}. Quantitative homogenization results for non-linear equations are given in \cite{AS14}.
\begin{theorem}\label{theo: main1}
	For all \(\epsilon \in (0, 1)\) there exist \(R_0 = R_0(\epsilon) \geq \tfrac{1}{\epsilon^2},C_1 = C_1(F)>0\), \(C_2 = C_2(\epsilon) > 0\), \(C_3= C_3(\epsilon) > 0\) and \(\delta > 0\) such that for all \(R \geq R_0\)
	we have 
	\[
	P \Big(\Big\{ \omega \in \Omega \colon \sup_{Q_R} |F_R - G_{\omega}| \leq \epsilon \ C_1 \Big\} \Big)\geq 1 - C_2 e^{- C_3 R^\delta}.
	\]
\end{theorem}
The main tool in the proof of \Cref{theo: main1} is a new parabolic Aleksandrov--Bakelman--Pucci maximum principle, see \Cref{theo: ABP} below.

The remaining article is organized as follows: In \Cref{sec: Pf QE} we prove the quantitative estimate (\Cref{theo: main1}), in \Cref{sec: Exiting Cylinder} we provided estimates on the exit probabilities from a cylinder through a part of the boundary, in \Cref{sec: Pf OI} we prove the oscillation inequality (\Cref{theo: OI}), in \Cref{sec: Pf PHI} we prove the PHI (\Cref{theo: PHI}) and finally in \Cref{sec:pf trans} we prove transience for \(d \geq 3\).

\section{Proof of the Quantitative Estimate: \Cref{theo: main1}}\label{sec: Pf QE}

\subsection{A Parabolic Maximum Principle}
In this section we prove a parabolic version of the Aleksandrov--Bakelman--Pucci (ABP) maximum principle \cite[Theorem 3.1]{Berger2014}. For the uniform elliptic setting related results are given in \cite{10.1215/S0012-7094-98-09122-0, 10.1214/009117905000000440}.
We need further notation:
For \(k \in \N\) set
\begin{align*}
\partial^k O_R &\triangleq \big\{x \not \in O_R \colon \exists_{y \in O_R}, \|x - y\|_\infty \leq k\big\},\\
\partial^k Q_R  &\triangleq \big(\partial^k O_R \times [ \lfloor R^2 \rfloor + k]\big) \cup \big(O_R \times \{\lfloor R^2 \rfloor, \dots, \lfloor R^2 \rfloor + k\}\big),
\\
Q^k_R &\triangleq Q_R \cup \partial^k Q_R.
\end{align*}
Fix a function \(u \colon Q^k_R \to \mathbb{R}\) and define for \((y, s) \in Q_R\) 
\begin{align*}
I_u(y, s) &\triangleq \big\{p \in \mathbb{R}^d \colon u(y, s) - u(x, t) \geq \langle p, y - x\rangle\ \forall\ (x, t) \in Q^k_R \text{ with } t > s\big\},
\\
\Gamma_u& \triangleq \big\{(y, s) \in Q_R \colon I_u(y, s) \not = \emptyset\big\}.
\end{align*}
Let \(\alpha (n)\) be the coordinate that changes between \(X_{n-1}\) and \(X_{n}\) and define 
\begin{align}\label{eq: def T}
T &\triangleq \inf( n \in \N \colon \{\alpha(1), \dots, \alpha(n)\} = \{1,\dots, d\}),\qquad
T^{(k)} \triangleq T \wedge k.
\end{align}
The following theorem will be an important tool at several steps in the proof of \Cref{theo: PHI}.
\begin{theorem}\label{theo: ABP}
	There exists an \(R_o > 0\) such that for all \(R \geq R_o, 0 < k < R\) and all \(\omega \in \mathsf{B}\) the following implication holds: If \(u \leq 0\) on \(\partial^k Q_R\) and for all \(z \in  O_R\) 
	\begin{align}\label{eq: cond}
	P_\omega^z (T > k) < e^{- (\log R)^3},
	\end{align}
	then
	\begin{align}\label{eq: ABP}
	\sup_{Q_R} u  \leq \c R^{\tfrac{d}{d + 1}}\Big( \sum_{(y, s) \in \Gamma_u} \big|E^{y}_\omega \big[u(X_{T^{(k)}}, s + 1 + T^{(k)})\big] - u(y, s + 1) \big|^{d + 1}\Big)^{\tfrac{1}{d + 1}}.
	\end{align}
\end{theorem}

This ABP maximum principle follows by combining arguments from the proofs of \cite[Theorem 3.1]{Berger2014} and \cite[Theorem 2.2]{deuschel2018}. For completeness, we give a proof in  \Cref{app: ABP pf}.

\subsection{Proof of \Cref{theo: main1}}
We start with some general comments:
\begin{enumerate}
	\item[1.] We fix \(\varepsilon > 0, n_0 \in \mathbb{N}, K > 0\) and \(\alpha \in (0, \frac{1}{2})\). Here, \(\epsilon\) is as in the statement of the theorem and \(\alpha = \alpha (\varepsilon)\) is a free parameter, which we choose small enough in the end of the proof. First, we determine \(K = K(d, P)\), then \(\alpha = \alpha (\varepsilon)\), \(n_0 = n_0(\varepsilon)\) and \(R_0 = R_0(\varepsilon)\).
	\item[2.] We denote by \(\C\) any constant which only depends on the dimension \(d\), the function \(F\colon \oK_1 \to \mathbb{R}\) and the environment measure \(P\). The constant might change from line to line.
	\item[3.] We simplify the notation and write \(G\) instead of \(G_\omega\). Moreover, we set \[k = k(R) \triangleq \lfloor \sqrt{R} \rfloor.\]
\end{enumerate}

The idea is to control
\[
H \triangleq F_R - G
\]
with the ABP maximum principle given by \Cref{theo: ABP}.
By definition of \(F_R\) and \(G\), we see that
\[
\begin{cases}
\mathcal{L}_\omega H = \mathcal{L}_\omega F_R,& \text{ on } Q_R,\\
H = 0, &\text{ on }\pp Q_R.
\end{cases}
\]
As \(H\) is only defined on \(Q_R\) instead of \(Q^k_R\) we cannot apply \Cref{theo: ABP} directly. To overcome this problem, we consider an extension \(h\) of \(H\).
Set \(R^* \triangleq R + \sqrt{d} k\) and let \(H' \colon \ok_{R^*} \to \mathbb{R}\) be a solution to
\[
\begin{cases}
\mathcal{L}_\omega H' = (\mathcal{L}_\omega F_R) \1_{Q_R},&\text{ on } K_{R^*},\\
H' = 0,&\text{ on } \pp  K_{R^*}.
\end{cases}
\]
\begin{lemma}\label{lem: some bounds}
	\begin{enumerate}
		\item[\textup{(i)}] \(\max_{Q_R} |H - H'| \leq \max_{\pp Q_R} |H'|.\) 
		\item[\textup{(ii)}] \(\max_{\partial^k Q_R} |H'| \leq \frac{\C}{\sqrt{R}}\).
	\end{enumerate}
\end{lemma}
\begin{proof}
	(i). Fix \((y, s) \in Q_R\), set \(\rho \triangleq \inf (t \in \N \colon (X_t, s + t) \in \pp Q_R)\) and note that \(\mathcal{L}_\omega (H - H') = 0\) on \(Q_R\). 
	Thus, we deduce from the optional stopping theorem that 
	\begin{align*}
	|H (y, s) - H'(y, s)| &= \big| E^{y}_\omega \big[ H(X_\rho, s + \rho) - H'(X_\rho, s + \rho) \big] \big| \\&= \big| E^{y}_\omega \big[ H'(X_\rho, s + \rho) \big] \big| \\&\leq \max_{\pp Q_R} |H'|.
	\end{align*}
	Thus, (i) follows. 
	
	(ii).
	By Taylor's theorem, we obtain for all \((y, s), (x, t) \in Q_R\)
	\begin{equation}\label{eq: Taylor}
	\begin{split}
	F_{R} (y, s) - F_{R} (x, t) = \tfrac{1}{R} \big \langle \nabla F \big (\tfrac{x}{R},\tfrac{t}{R^2}\big), y - x\big\rangle &+ \tfrac{1}{2 R^2}\big \langle y - x, \nabla^2 F \big (\tfrac{x}{R}, \tfrac{t}{R^2}\big) (y- x)\big\rangle \\&+ \tfrac{1}{R^2} \tfrac{d F}{dt}  \big (\tfrac{x}{R}, \tfrac{t}{R^2}\big) (s - t) \\&+ \rho^{\text{t}}_s |s - t|^2 + \rho^{\text{s}}_y \|y - x\|^3_2,
	\end{split}
	\end{equation}
	where \(\rho^\text{s}_y\) is bounded by \(\frac{\C}{R^3}\) and \(\rho^\text{t}_s\) is bounded by \(\frac{\C}{R^4}\).
	Thus, for all \((x, t) \in Q_R\)
	\begin{align}\label{eq: LFR bound}
	\big| (\mathcal{L}_\omega F_R) (x, t) \big| &= \big| E^x_\omega \big[ F_R (X_1, 1 + t) \big] - F_R(x, t) \big|
	\leq \tfrac{\C}{R^2}.
	\end{align}
	We set \begin{align*}
	\tau &\triangleq \inf (t \in \N \colon (X_t, s + t) \in \pp K_{R^*}),\qquad \rho \triangleq \inf (t \in \N \colon X_t \in \partial B_{R^*}).\end{align*}	
	As \((\|X_n\|^2_2 - n)_{n \in \N}\) is a \(P^x_\omega\)-martingale, the optional stopping theorem yields that 
	\begin{align}\label{eq: st boundary bound}
	\max_{x \in \partial^k O_R} E^x_\omega \big[ \rho \big] = \max_{x \in \partial^k O_R} \big(E^x_\omega \big[ \|X_\rho\|^2_2 \big] - \|x\|^2_2\big)
	\leq \c Rk.
	\end{align}
	Fix \((y, s) \in \partial^k Q_{R}\).
	The optional stopping theorem also yields that 
	\begin{equation}\label{eq: H' comp}
	\begin{split}
	H' (y, s) 
	&=E^{y}_\omega \big[H' (X_{\tau}, s+ \tau) \big] - E^y_\omega \Big[ \sum_{t = 0}^{\tau - 1} \mathcal{L}_\omega H'(X_t, s + t) \Big]
	\\&= - E^y_\omega \Big[ \sum_{t = 0}^{\tau - 1} \mathcal{L}_\omega F_R (X_t, s + t) \1_{Q_R} (X_t, s + t) \Big].
	\end{split}
	\end{equation}
	If \(s = \lfloor R^2 \rfloor\), we have \(H' (y, s) = 0\) and if \(s < \lfloor R^2\rfloor\), then \(y \in \partial O_R\) and we deduce from \eqref{eq: LFR bound}, \eqref{eq: st boundary bound} and \eqref{eq: H' comp} that 
	\[
	\big| H' (y, s) \big| \leq \tfrac{\C}{R^2} E^y_\omega \big[ \tau \big] \leq \tfrac{\C}{R^2} E^y_\omega \big [ \rho \big] \leq \tfrac{\C Rk}{R^2} = \tfrac{\C}{\sqrt{R}}.
	\]
	The proof is complete.
\end{proof}

Next, we add a quadratic  penalty term to the function \(H'\). 
Define
\[
h (y, s) \triangleq H'(y, s) + \frac{\C' \varepsilon}{R^2} \|y\|^2_2, \quad (y, s) \in \ok_{R^*}.
\]
We will determine the constant \(\C' = \C'(F) > 0\) in \Cref{lem: few points in upper contact set} below.

To apply \Cref{theo: ABP} to \(h\), we have to control the upper contact set of \(h\) and the \(\omega\)-Laplacian of \(h\).	
In the next lemma we show that \(\C'\) can be chosen such that only a few points are in the upper contact set.
To formulate the lemma, we need more notation: 	Recall that we fixed a constant \(n_0 = n_0(\varepsilon)\). Set 
\[
(M^{(n_0)}_\omega (x))_{ij} \triangleq \frac{1}{n_0} E^x_\omega \big[ (X^{(i)}_{n_0} - x^{(i)})(X_{n_0}^{(j)} - x^{(j)})\big], \quad 1 \leq i, j \leq d,
\]
where \(X_{n_0}^{(k)}\) and \(x^{(k)}\) denote the \(k^\text{th}\) coordinate of \(X_{n_0}\) and \(x\).
Moreover, set 
\[
A_{n_0}(x) \triangleq \big\{ \omega \in \Omega \colon \|M^{(n_0)}_\omega(x) - \a \| < \varepsilon \big\},
\]
where \(\|\cdot\|\) denotes the trace norm, i.e. \(\|M\| \triangleq \textup{tr} (\sqrt{M M^*})\).
Here, \(\a\) is the limiting covariance matrix as given by \Cref{theo: IP}.
Finally, set 
\[
J_{n_0} (R) \triangleq \big\{\hat{x} \in \overline{Q}_R \colon d(\hat{x}, \pp Q_R) > {n_0} \big\}, \quad d \equiv \text{distance function}.
\]
We are in the position to formulate the lemma announced above:
\begin{lemma}\label{lem: few points in upper contact set}
	The constant \(\C'\) can be chosen such that the following holds: Let \(R > \frac{\sqrt{n_0}}{\varepsilon} \vee n_0\). If \((x, t) \in J_{n_0} (R)\) and \(\omega \in A_{n_0} (x)\), then \((x, t) \not \in \Gamma_h\).
\end{lemma}
\begin{proof}
	Take \(R > \frac{\sqrt{n_0}}{\varepsilon} \vee n_0, (x, t) \in J_{n_0} (R)\) and \(\omega \in A_{n_0} (x)\). 
	Recalling \eqref{eq: Taylor} and using that the walk \(X\) is a \(P_\omega^x\)-martingale, we obtain 
	\begin{align*}
	E_{\omega}^x \big[ F_R (X_{n_0}, t + n_0) \big] &- F_R(x, t) 
	\\&= \tfrac{1}{2 R^2} E^x_\omega \big[ \langle X_{n_0} - x, \nabla^2 F \big (\tfrac{x}{R}, \tfrac{t}{R^2}\big) (X_{n_0}- x)\rangle \big] \\&\qquad\qquad+ \tfrac{n_0}{R^2} \tfrac{d F}{dt}  \big (\tfrac{x}{R}, \tfrac{t}{R^2}\big)  + \rho^{\text{t}}_{t + n_0} n_0^2 + E^x_\omega \big[\rho^{\text{s}}_{X_{n_0}} \|X_{n_0} - x\|^3_2\big].
	\end{align*}
	Since \(\omega \in A_{n_0} (x)\), we have
	\begin{align*}
	E^x_\omega \big[ \langle X_{n_0} - x, \nabla^2 F \big (\tfrac{x}{R}, \tfrac{t}{R^2}\big) (X_{n_0}- x)\rangle \big] &= n_0 \textup{tr} \big(\nabla^2 F \big (\tfrac{x}{R}, \tfrac{t}{R^2}\big) M^{(n_0)}_\omega (x) \big) 
	\\&\leq \C n_0 \varepsilon + n_0 \textup{tr} \big( \a \nabla^2 F \big (\tfrac{x}{R}, \tfrac{t}{R^2}\big) \big).
	\end{align*}
	Using that \(\frac{dF}{dt} + \frac{1}{2} \text{tr}(\a \nabla^2 F) = 0\), we obtain 
	\begin{align*}
	\big|E_{\omega}^x \big[ F_R (X_{n_0}, t + n_0) \big] - F_R(x, t)\big| &\leq \tfrac{n_0 \C \varepsilon}{2 R^2} + \tfrac{n_0}{R^2} \big( \tfrac{1}{2} \textup{tr} \big( \a \nabla^2 F \big (\tfrac{x}{R}, \tfrac{t}{R^2}\big) \big) + \tfrac{d F}{dt}  \big (\tfrac{x}{R}, \tfrac{t}{R^2}\big) \big) \\&\qquad\qquad \qquad \qquad+ \rho^{\text{t}}_{t + n_0} n_0^2 + E^x_\omega \big[\rho^{\text{s}}_{X_{n_0}} \|X_{n_0} - x\|^3_2\big]
	\\&= \tfrac{\C n_0 \varepsilon}{2 R^2}  + \rho^{\text{t}}_{t + n_0} n_0^2 + E^x_\omega \big[\rho^{\text{s}}_{X_{n_0}} \|X_{n_0} - x\|^3_2\big]
	\\&\leq \tfrac{\C n_0 \varepsilon}{2 R^2}  + \tfrac{\C n_0^2}{R^4} + \tfrac{\C}{R^3} E^x_\omega \big[ \|X_{n_0} - x\|^3_2\big].
	\end{align*}
	We deduce from the Burkholder--Davis--Gundy inequality that 
	\(
	E^x_\omega \big[ \|X_{n_0} - x\|^3_2 \big] \leq \c n^{\frac{3}{2}}_0.
	\)
	In summary, we have
	\begin{align*}
	\big|E^{\omega}_x \big[ F_R (X_{n_0}, t + n_0) \big] - F_R(x, t)\big| &\leq \tfrac{n_0 \varepsilon}{R^2} \big( \tfrac{\C}{2} + \C \big(\tfrac{n_0}{\varepsilon  R^2} + \tfrac{\C \sqrt{n_0}}{\varepsilon R}\big)\big) \leq \C \tfrac{n_0 \varepsilon}{R^2} \triangleq \C' \tfrac{n_0 \varepsilon}{2 R^2}.
	\end{align*}
	As \(\mathcal{L}_\omega (H' - F_R) = 0\) on \(Q_R\), we have
	\[
	E^x_\omega \big[ H' (X_{n_0}, t + n_0)\big] - H'(x, t) = E^x_\omega \big[F_R (X_{n_0}, t + n_0)\big] - F_R(x, t).
	\]
	We obtain 
	\begin{align*}
	E^x_\omega \big[ h (X_{n_0}, t + n_0) \big] &- h(x, t) \\&= E^x_\omega \big[ F_R (X_{n_0}, t + n_0) \big] - F_R (x, t) + \C' \tfrac{\varepsilon}{R^2} E^x_\omega \big[ \|X_{n_0}\|^2_2 - \|x\|^2_2 \big]
	\\&= E^x_\omega \big[ F_R (X_{n_0}, t + n_0) \big] - F_R (x, t) + \C' \tfrac{n_0 \varepsilon}{R^2}
	\\&\geq - \C' \tfrac{n_0 \varepsilon}{2 R^2} + \C' \tfrac{n_0 \varepsilon}{R^2} = \C' \tfrac{n_0 \varepsilon}{2R^2} > 0.
	\end{align*}
	Using this inequality and the fact that martingales have constant expectation, we obtain for all \(p \in \mathbb{R}^d\)
	\[
	E^x_\omega \big[ h (X_{n_0}, t + n_0) +  \langle p, x - X_{n_0} \rangle  \big] - h(x, t) > 0.
	\]
	Thus, for all \(p \in \mathbb{R}^d\) there exists a \(y\) in the \(P^x_\omega\)-support of \(X_{n_0}\) such that 
	\[
	\langle p, x - y\rangle > h(x, t) - h(y, t + n_0).
	\]
	We conclude that \((x, t) \not \in \Gamma_h\). The proof is complete.
\end{proof}
\Cref{lem: few points in upper contact set} suggests that we should restrict our attention to environments which are in \(A_{n_0} (x)\) for many \(x \in O_R\). Motivated by this observation, we define 
\[
A^{(1)} = A^{(1)}_R (\alpha) \triangleq \Big\{ \omega \in \Omega \colon \frac{1}{|B_R|} \sum_{x \in B_R} \1_{\{\omega \in A_{n_0} (x)\}} > 1 - 2 \alpha \Big\},
\]
where \(\alpha = \alpha(\varepsilon)\) is one of the free constants fixed in the beginning of the proof.

Next, we control the \(\omega\)-Laplacian of \(h\):
\[
(L_\omega h) (y, s) \triangleq E^{y}_\omega \big[ h(X_{T^{(k)}}, s + 1 + T^{(k)}) \big] - h(y,s + 1), \quad (y, s) \in K_R,
\]
where \(T^{(k)} = T \wedge k\) is the stopping time defined in \eqref{eq: def T}.
\begin{lemma}\label{lem: omega-Laplace bound}
	For all \((x, s) \in K_R\)
	\[
	\big| (L_\omega h)(x, s) \big| \leq \frac{\C E^x_\omega [T]}{R^2}.
	\]
\end{lemma}	
\begin{proof}
	Taylor's theorem yields that \(| (\mathcal{L}_\omega F_R)(x, t)| \leq \frac{\C}{R^2}\) for all \((x, t) \in K_R\).
	For \(f (x) = \|x\|^2_2\), note that 
	\(
	(\mathcal{L}_\omega f) (x, t) = E^x_\omega\big[ \|X_1\|^2_2 \big] - \|x\|^2_2 = 1.
	\)
	Consequently, we obtain that 
	\begin{align*}\
	|\mathcal{L}_\omega h  | = |(\mathcal{L}_\omega F_R ) \1_{Q_R} + \C' \tfrac{\varepsilon}{R^2}| \leq \tfrac{\C}{R^2}.
	\end{align*}
	We deduce from the optional stopping theorem that 
	\begin{equation*}
	\begin{split}
	\big| (L_\omega h)(x, s) \big| &= \Big| E^x_\omega \Big[ \sum_{t = 0}^{T^{(k)} - 1} (\mathcal{L}_\omega h) (X_t, s + 1 + t) \Big]\Big|
	\leq E^x_\omega [ T ] \tfrac{\C}{R^2}.
	\end{split}
	\end{equation*}
	This completes the proof.
\end{proof}
\Cref{lem: omega-Laplace bound} shows that the \(\omega\)-Laplacian of \(h\) can be controlled via \(x \mapsto E^x_\omega [T]\). Motivated by this observation, we define 
\[
A^{(2)}_R \triangleq \Big\{ \omega \in \Omega \colon \frac{1}{|B_R|} \sum_{x \in B_R} \big| E^x_\omega [T] \big|^{d + 2} \leq K \Big\},
\]
where \(K\) is one of the constants we fixed in the beginning.

As a last step before we apply \Cref{theo: ABP}, we introduce the following:
\[
A^{(3)}_R \triangleq \Big\{ \omega \in \Omega \colon P^x_\omega (T > k) < e^{- (\log R)^3} \text{ for all } x \in O_R \Big\},
\]
which is in conjunction with the statement of \Cref{theo: ABP}.

We are in the position to complete the proof of \Cref{theo: main1}. Take \(\omega \in A^{(1)} \cap A^{(2)} \cap A^{(3)}\) and let \(R\) be such that \(R_o \vee \frac{n_0}{\alpha} \vee \frac{\sqrt{n_0}}{\varepsilon} \vee n_0 \leq R\), where \(R_o\) is as in \Cref{theo: ABP}. 
Note that 
\[
\llambda (\B_R \backslash \B_{R - n_0}) = \c \int_{R - n_0}^R r^{d - 1} dr \leq \c n_0 R^{d- 1}.
\]
As \(\omega \in A^{(1)}_R\), we have for \(s < \lceil R^2 \rceil - n_0\)
\begin{align}\label{eq: B1}
\frac{1}{|B_R|} \sum_{x \in B_R} \1_{\{(x, s) \not \in J_{n_0}(R)\text{ or }\omega \not \in A_{n_0}(x)\}} \leq \tfrac{\C n_0}{R} + 2 \alpha \leq (\C + 2) \alpha = \C\alpha.
\end{align}
Moreover, because \(\omega \in A^{(2)}_R\), we deduce from \Cref{lem: omega-Laplace bound} that
\begin{align}\label{eq: B2}
\frac{1}{|B_R|} \sum_{y \in B_R} \big| (L_\omega h) (y, s) \big|^{d + 2} &\leq \frac{\C}{R^{2 (d + 2)}} \frac{1}{|B_R|} \sum_{y \in B_R} \big| E^y_\omega  [T] \big|^{d + 2} \leq \frac{\C  K}{R^{2 (d + 2)}},
\intertext{and}
\frac{1}{|B_R|} \sum_{x \in B_R} \big| (L_\omega h)(y, s) \big|^{d + 1}
&\leq \frac{\C }{R^{2 (d + 1)}} \frac{1}{|B_R|} \sum_{y \in B_R} \big| E^y_\omega  [T] \big|^{d + 2} \leq \frac{\C  K}{R^{2 (d + 1)}}. \label{eq: B3}
\end{align}
Furthermore, \Cref{lem: some bounds} yields that 
\begin{align}\label{eq: B4}
\max_{\partial^k Q_R} h \leq \max_{\partial^k Q_R} H' +   \C' \varepsilon \tfrac{(R + k)^2}{R^2} \leq \C \big( \tfrac{1}{\sqrt{R}} + \varepsilon\big).
\end{align}
Because \(\omega \in A^{(3)}_R\), we can apply \Cref{theo: ABP} and obtain that 
\begin{align}\label{eq: some eq}
\max_{Q_R} h - \max_{\partial^k Q_R} h \leq \C R^{\frac{d}{d + 1}} \Big( \sum_{(y, s) \in \Gamma_h} \big| (L_\omega h)(y, s) \big|^{d + 1} \Big)^{\frac{1}{d + 1}}.
\end{align}
Using \Cref{lem: few points in upper contact set}, we obtain
\begin{align*}
\eqref{eq: some eq}&= \C R^{\frac{2d}{d + 1}} \Big( \frac{1}{|B_R|} \sum_{(y, s) \in \Gamma_h} \1_{\{(x,s) \not \in J_{n_0}(R)\text{ or }\omega \not \in A_{n_0}(x)\}}\big| (L_\omega h) (y, s) \big|^{d + 1}\Big)^{\frac{1}{d + 1}}
\\&\leq \C R^{\frac{2d}{d + 1}} \Big( \frac{1}{|B_R|} \sum_{(y, s) \in K_R} \1_{\{(x, s) \not \in J_{n_0}(R)\text{ or }\omega \not \in A_{n_0}(x)\}}\big| (L_\omega h) (y, s) \big|^{d + 1}\Big)^{\frac{1}{d + 1}}.
\intertext{Using H\"older's inequality, \eqref{eq: B1}, \eqref{eq: B2} and \eqref{eq: B3}, we further obtain that}
\eqref{eq: some eq} &\leq \C R^{\frac{2d}{d + 1}} \Big(\sum_{s = 0}^{R^2 - n_0 - 1} \Big[ \frac{1}{|B_R|} \sum_{x \in B_R} \1_{\{(x, s) \not \in J_{n_0}(R)\text{ or }\omega \not \in A_{n_0}(x)\}} \Big]^{\frac{1}{d + 2}} \\ &\hspace{4cm} \Big[ \frac{1}{|B_R|} \sum_{y \in B_R} \big| (L_\omega h) (y, s) \big|^{d + 2} \Big]^{\frac{d + 1}{d + 2}} 
\\&\hspace{4cm} +\sum_{s = R^2 - n_0}^{R^2 - 1} \frac{1}{|B_R|} \sum_{x \in B_R} \big| (L_\omega h)(y, s) \big|^{d + 1}\Big)^{\frac{1}{d + 1}}
\\&\leq \C R^{\frac{2d}{d + 1}} \Big( R^2 \big[\C\alpha \big]^{\frac{1}{d + 2}} \big[ \C K \big]^{\frac{d + 1}{d + 2}} R^{- 2 (d + 1)} + n_0 \C K R^{- 2 (d + 1)} \Big)^{\frac{1}{d + 1}}
\\&\leq \C \Big( \alpha^{\frac{1}{d + 2}}K^{\frac{d + 1}{d + 2}} + \alpha K  \Big)^{\frac{1}{d + 1}}.
\end{align*}
Combining this bound with \Cref{lem: some bounds} and \eqref{eq: B4} shows that 
\begin{align*}
\max_{Q_R} H &\leq \max_{Q_R} H' + \tfrac{\C}{\sqrt{R}} 
\\&\leq \max_{Q_R} h + \tfrac{\C}{\sqrt{R}} 
\\&\leq \C  \big( \alpha^{\frac{1}{d + 2}}K^{\frac{d + 1}{d + 2}} + \alpha K  \big)^{\frac{1}{d + 1}} + \C \big( \tfrac{1}{\sqrt{R}} + \varepsilon\big) + \tfrac{\C}{\sqrt{R}}
\\&= \C  \big( \alpha^{\frac{1}{d + 2}}K^{\frac{d + 1}{d + 2}} + \alpha K  \big)^{\frac{1}{d + 1}} + \C  \varepsilon + \tfrac{\C}{\sqrt{R}}.
\end{align*}
Replacing the roles of \(F_R\) with \(G\) yields that 
\[
\max_{Q_R} |H| \leq \C  \big( \alpha^{\frac{1}{d + 2}}K^{\frac{d + 1}{d + 2}} + \alpha K  \big)^{\frac{1}{d + 1}}  + \C  \varepsilon + \tfrac{\C}{\sqrt{R}}.
\]
To complete the proof we determine the constants. 
First, we choose \(K\) according to the following lemma:
\begin{lemma} \textup{(\cite[Lemma 2.3]{Berger18})}
	One can choose \(K\) such that the following holds: There exists a constant \(\delta\) such that \[
	P\big(A^{(2)}_R\big)> 1 - K e^{- R^\delta}.
	\]
\end{lemma}
Next, we choose \(\alpha = \alpha(\varepsilon)\) such that 
\(
( \alpha^{\frac{1}{d + 2}} + \alpha )^{\frac{1}{d + 1}} \leq \varepsilon. 
\)
Then,
\[
\max_{Q_R} | H | \leq \C \varepsilon + \tfrac{\C}{\sqrt{R}}.
\]
We choose \(n_0 =  n_0 (\varepsilon)\) according to the following lemma:
\begin{lemma} \textup{(\cite[Lemma 2.1]{Berger18})}
	There exists a \(n_0 = n_0(\varepsilon)\) and constants \(c = c(n_0) > 0\) and \(C = C(n_0) > 0\) such that 
	\[
	P\big(A^{(1)}_R \cap A^{(3)}_R \big) > 1 - C e^{- c R^{\frac{1}{7}}}.
	\]
\end{lemma}
Now, we choose \(R_0 = R_0 (\varepsilon) \geq R_o \vee \frac{n_0}{\alpha} \vee \frac{1}{\varepsilon^2}  \vee \frac{\sqrt{n_0}}{\varepsilon} \vee n_0\), where \(R_o\) is as in \Cref{theo: ABP}.
In summary, for all \(R \geq R_0\) and \(\omega \in A^{(1)}_R \cap A^{(2)}_R \cap A^{(3)}_R\) 
we have 
\(
\max_{Q_R} |H| \leq \C\varepsilon,
\)
and 
\(
P \big(A^{(1)}_R \cap A^{(2)}_R \cap A^{(3)}_R \big) \geq  1 - C (\varepsilon) e^{- c(\varepsilon) R^{\delta}}.
\)
The proof of \Cref{theo: main1} is complete.
\qed

\section{An Estimate for the Exit Measure}\label{sec: Exiting Cylinder}
Take a Borel set \(\A \subseteq \partial \K_1\) those boundary has zero %surface \(\times\) Lebesgue 
measure, i.e. 
\begin{align}\label{eq: assp meas zero boundary}
\text{meas} \big(\{x \in \partial \K_1 \colon d (x, \A) = 0 = d(x, \partial^p \K_1 \backslash \A)\}\big)= 0, \quad d \equiv \textup{distance function},
\end{align}
and define for \((x, t) \in \mathbb{Z}^d \times \N\)
\begin{equation}\begin{split}\label{eq: A_R}
RA (x, t) \triangleq \big\{ (y, s) \in \pp K_R (x, t) \colon \big(\tfrac{y - x}{\|y - x\|_2 \vee R}, \tfrac{s - t}{\lceil R^2\rceil} \big) \in \A \big\}.
\end{split}
\end{equation}
We also set \(RA \triangleq RA(0)\).
Furthermore, set
\begin{equation}\label{eq: exit prob def}
\begin{split}
\tau_s &\triangleq \inf (t \in \mathbb{R}_+ \colon (X_t, t + s) \not \in \K_1), \quad  s \in [0, 1]\\
\rho_s &\triangleq \inf(t \in \N \colon (X_t, t + s) \not \in K_R), \quad s \in [\lceil R^2\rceil],\\
\chi (x, s) &\triangleq \bm^x ( (X_{\tau_s}, \tau_s + s) \in \A), \quad (x, s) \in \oK_1,\\
\chi_R (x, s) &\triangleq \chi \big(\tfrac{x}{R}, \tfrac{s}{R^2}),\quad (x, s) \in \oK_R,\\
\Phi_R (x, s) &\triangleq P_\omega^x ( (X_{\rho_s}, \rho_s + s) \in RA), \quad (x, s) \in \ok_R.
\end{split}
\end{equation}
Here, \(\bm^x\) denotes the law of a Brownian motion with covariance matrix \(\a\) and starting value \(x\).
\begin{corollary}\label{coro: QE}
	For every \(\epsilon > 0\) and \(\theta \in (0, 1)\) there exist \(R_o = R_o(\A,\varepsilon, \theta) > 0, \c_1 = \c_1(\A,\varepsilon, \theta), \c_2 = \c_2(\A,\varepsilon, \theta)\) and \(\delta > 0\) such that for all \(R \geq R_o\) 
	\[
	P \Big( \Big\{ \omega \in \Omega \colon \sup_{K_{\theta R}} | \chi_{R} - \Phi_{R}| \leq \varepsilon \Big\} \Big) \geq 1 - \c_1e^{- \c_2R^\delta}.
	\]
\end{corollary} 
\begin{proof}
	\emph{Step 1:} Fix a small number \(\gamma > 0\) and define 
	\begin{align*}
	\A^+_\gamma &\triangleq \big\{x \in \partial^p \K_1 \colon d(x, \A) \leq \gamma \big\},\\
	\A^-_\gamma &\triangleq \big\{ x \in \A \colon d(x,\partial^p \K_1 \backslash \A) \geq \gamma\big\}.
	\end{align*}
	Note that
	\[
	\A^-_{2 \gamma} \subseteq \A^-_\gamma \subseteq \A\subseteq \A^+_\gamma \subseteq \A^+_{2 \gamma}.
	\]
	Let \(f^{(1)}, f^{(2)} \colon \pp \K_1 \to [0, 1]\) be sufficiently smooth functions 
	such that 
	\[
	f^{(1)} = 
	\begin{cases}
	1,&\text{on } \A^-_{\gamma},\\
	0,&\text{on } \pp \K_1 \backslash \A^-_{2 \gamma},
	\end{cases}
	\qquad 
	f^{(2)} = \begin{cases} 1,& \text{on } \A^+_{\gamma},\\
	0,& \text{on } \pp \K_1 \backslash \A^+_{2\gamma}.
	\end{cases}
	\]
	For \(k = 1, 2\) let \(J^{(k)} \colon \oK_1 \to \mathbb{R}\) be a solution to the boundary value problem
	\begin{align*}
	\begin{cases}
	\frac{1}{2} \sum_{i, j = 1}^d \a_{ij} \frac{\dd^2 J^{(k)}}{\dd x_i \dd x_j} + \frac{\dd J^{(k)}}{\dd t} = 0,& \text{on } \K_1, \\
	J^{(k)} = f^{(k)},&\text{on } \pp \K_1.
	\end{cases}
	\end{align*}
	The optional stopping theorem yields that 
	\begin{align}\label{eq: J stoch rep}
	J^{(k)} (x, s) = E_\textup{BM}^x [f^{(k)}(X_{\tau_s}, \tau_s + s)], \quad  (x, s) \in \oK_1, k = 1, 2.
	\end{align}
	Next, we set 
	\[
	F^{(k)}_{R + 1} (x, t) \triangleq J^{(k)} \big(\tfrac{x}{R + 1}, \tfrac{t}{(R + 1)^2}\big), \quad (x, t) \in \overline{Q}_{R + 1}.
	\]
	Note that 
	\[
	\bigcap_{\gamma> 0} \A^+_{2 \gamma} = \{x \in \pp \K_1 \colon d(x, \A) = 0\},
	\]
	which implies that 
	\(
	\bigcap_{\gamma > 0}\A^+_{2 \gamma} \backslash \A \subseteq \partial \A.
	\)
	Thus, due to \eqref{eq: assp meas zero boundary} and % in view of \cite[Corollary 2]{MS18} and 
	\eqref{eq: J stoch rep}, we obtain that
	\begin{align*}
	\max_{(x, t) \in K_{\theta (R + 1)}} (F^{(2)}_{R + 1} &(x,t) - \chi_{R + 1}(x,t)) 
	\\&\leq \max_{(x, t) \in \K_{\theta}} \bm^{x} ((X_{\tau_t}, \tau_{t} + t ) \in \A^+_{2\gamma} \backslash \A) \to 0 \text{ as } \gamma\searrow 0.
	\end{align*}
	Next, note that 
	\[
	\bigcup_{\gamma > 0} \A_{\gamma}^- = \A\cap \{x \in \pp \K_1 \colon d(x, \pp \K_1 \backslash \A) > 0\},
	\]
	which implies that \(\A \backslash \bigcup_{\gamma > 0} \A^-_\gamma \subseteq \partial \A\).
	Due to \eqref{eq: assp meas zero boundary} and \eqref{eq: J stoch rep}, 
	we obtain
	\begin{align*}
	\max_{(x, t) \in K_{\theta (R +  1)}} (\chi_{R + 1} &(x, t) - F^{(1)}_{R + 1}(x,  t)) 
	\\&\leq \max_{(x, t) \in \K_{\theta}} \bm^x ((X_{\tau_t}, \tau_t + t) \in \A\backslash \A^{-}_\gamma)
	\to 0 \text{ as }  \gamma \searrow 0.
	\end{align*}
	Consequently, there exists a \(\gamma = \gamma (\varepsilon, \theta) > 0\) such that the following holds:
	\begin{align}\label{eq: Lem (i) ersatz}
	F^{(2)}_{R + 1} - \varepsilon \leq \chi_{R + 1} \leq F^{(1)}_{R + 1} + \varepsilon\text{ on }K_{\theta (R + 1)}.
	\end{align} 
	Take this
	\(\gamma\).
	Note that the function \(\chi\) is uniformly continuous on \(\oK_\theta\), as it is continuous on \(\K_1\). Thus, 
	assuming	that \(R^o\) is large enough,
	we have
	\[
	\max_{\K_{\theta R}} | \chi_R - \chi_{R + 1}| \leq \varepsilon.
	\]
	Now, it follows from \eqref{eq: Lem (i) ersatz} that 
	\begin{align}\label{eq: Lem (i) ersatz2}
	F^{(2)}_{R + 1} - 2\varepsilon \leq \chi_{R} \leq F^{(1)}_{R + 1} + 2\varepsilon\text{ on }K_{\theta R}.
	\end{align}
	\emph{Step 2:} For \(P\)-a.a. \(\omega \in \Omega\) and \(k = 1,  2\) we define \(G^{(k)}_{R + 1} \colon K_R \to \mathbb{R}_+\) as solutions to the following boundary value problem:
	\[
	\begin{cases}
	\mathcal{L}_\omega G^{(k)}_{R + 1} = 0,& \text{on } Q_{R + 1},\\
	G^{(k)}_{R + 1} = F^{(k)}_{R + 1}, &\text{on } \pp Q_{R + 1}.
	\end{cases}
	\]
	For \(k = 1, 2\), let \(C^{(k)}_1 = C_ 1^{(k)} (J^{(k)}) = C_1^{(k)} (\A, \varepsilon, \theta) > 0\) be the constant from \Cref{theo: main1} and set 
	\[
	\he \triangleq \frac{\varepsilon}{C^{(1)}_1 \vee C^{(2)}_2}.
	\]
	Using \Cref{theo: main1} with \(\he\) instead of \(\varepsilon\) yields that there exists a set \(G = G(\A, \varepsilon, \theta, R) \in \mathcal{F}\) such that, after eventually enlarging \(R_o\), for all \(\omega \in G,\) all \(R \geq R_o\) and \(k = 1, 2\)
	\begin{align}\label{eq: quanti estimate coro ineq}
	\max_{Q_{R + 1}} |F^{(k)}_{R + 1} - G^{(k)}_{R + 1}| \leq C^{(k)} \he \leq \varepsilon.
	\end{align}
	\emph{Step 3:} 
	In this step we show that
	\begin{align}\label{eq: (iii) ersatz}
	G^{(1)}_{R + 1} - 2\varepsilon \leq \Phi_{R} \leq G^{(2)}_{R + 1} + 2\varepsilon\text{ on } K_{R}.
	\end{align}
	For \((x, t) \in \partial K_R \backslash RA\), 
	we obtain for sufficiently large \(R_o\) that
	\(
	F^{(1)}_{R + 1} (x, t) 
	< \varepsilon.
	\)
	To see this, recall that \(J^{(1)} = 0\) on \(\pp \K_1 \backslash \A\) and note that for \((x, t) \in \partial K_R \backslash RA\)
	\[
	F^{(1)}_{R + 1} (x, t) = 
	\big|J^{(1)} \big(\tfrac{x}{R + 1}, \tfrac{t}{(R + 1)^2}\big) - J^{(1)} \big(\tfrac{x}{\|x\|_2 \vee R}, \tfrac{t}{\lceil R^2 \rceil}\big) \big|.
	\]
	Since
	\[
	\big\| \tfrac{x}{R + 1} - \tfrac{x}{\|x\|_2 \vee R} \big\|_2 \vee \big| \tfrac{t}{(R + 1)^2} - \tfrac{t}{\lceil R^2\rceil} \big|^\frac{1}{2} \leq \big(1 - \tfrac{R}{R + 1} \big) \vee \big(1 - \tfrac{\lceil R^2 \rceil}{(R + 1)^2} \big)^\frac{1}{2} \to 0 \text{ as } R \to \infty,
	\]
	the uniform continuity of \(J^{(1)}\) on \(\oK_1\) yields the claim. 	
	In the same manner, eventually enlarging \(R_o\) again, we obtain
	\(
	1 - F^{(2)}_{R + 1} (x, t) \leq \varepsilon
	\)
	for \((x, t) \in RA\).
	In summary, 
	\[
	F^{(1)}_{R + 1} - \varepsilon \leq \1_{RA} \leq F^{(2)}_{R + 1} + \varepsilon \quad \text{ on } \partial K_R.
	\]
	Together with \eqref{eq: quanti estimate coro ineq}, we conclude that on \(\partial K_R\)
	\[
	G^{(1)}_{R + 1} - 2 \varepsilon \leq \1_{RA} \leq G^{(2)}_{R + 1} + 2  \varepsilon.
	\]
	Using once again the optional stopping theorem yields \eqref{eq: (iii) ersatz}.
	\\\noindent\\
	\emph{Step 4:}
	Due to \eqref{eq: Lem (i) ersatz2} and \eqref{eq: (iii) ersatz}, we obtain that on \(K_{\theta R}\)
	\[
	G^{(1)}_{R + 1} - F^{(1)}_{R + 1} - 4\varepsilon \leq \Psi - \chi_{R} \leq G^{(2)}_{R + 1} - F^{(2)}_{R + 1} + 4\varepsilon.
	\]
	Finally, with \eqref{eq: quanti estimate coro ineq}, we conclude that on \(K_{\theta R}\)
	\begin{align*}
	|\Psi - \chi_{R}| 
	&\leq 4 \varepsilon + |G^{(1)}_{R + 1} - F^{(1)}_{R + 1}| + |G^{(2)}_{R + 1} - F^{(2)}_{R + 1}|
	\leq 6 \varepsilon.
	\end{align*}
	The proof is complete.
\end{proof}
\section{Proof of the Oscillation Inequality: \Cref{theo: OI}}\label{sec: Pf OI}
\subsection{An Oscillation Inequality on a Small Scale}
The main result of this section is the following oscillation inequality on a small scale:
\begin{proposition}\label{prop: osc s c}
	There exist constants \(\alpha > 0,  c \in \mathbb{N}\) such that for all \(R \geq 1\) there is a constant \(C \in (0, 1)\) and a set \(G \in \mathcal{F}\) with \(P(G) \geq 1 - \c R^{3d} e^{- R^\alpha}\) such that for all \(\omega \in G, p \in \{o, e\}\) and every \(\omega\)-caloric function \(u \colon \ok_{(c + 3) R} \to \mathbb{R}\) the following oscillation inequality holds:
	\begin{align}\label{eq: smsc osc}
	\underset{\Theta^p(K_{R})}{\osc}\ u \leq C \underset{\Theta^p(K_{(c + 3) R})}{\osc}\ u.
	\end{align}
\end{proposition}
\begin{proof}
	To prove this result we need input from \cite{Berger18}:
	For \(x, y \in \mathbb{Z}^d\) we write \(x \oc y\) in case
	\[
	P^x_\omega (\exists_{n \in \mathbb{N}} \colon X_n = y) > 0.
	\]
	We call a set \(A \subseteq \mathbb{Z}^d\) to be \emph{strongly connected} w.r.t. \(\omega \in \Omega\) if \(x \oc y\) for  every \(x, y \in A\).
	Moreover, we call a set \(A \subseteq \mathbb{Z}^d\) to be a \emph{sink} w.r.t. \(\omega\in \Omega\) if it is strongly connected w.r.t. \(\omega\) and for every \(x \in A\) and \(y \not \in A\)
	\[
	P_\omega^x (\exists_{n \in \mathbb{N}} \colon X_n = y) = 0.
	\]
	In other words, a sink is a strongly connected set from which the walk cannot escape. Due to \cite[Proposition 1.13]{Berger18},
	for \(P\)-a.a. \(\omega \in \Omega\) there exists a unique sink \(\mathcal{C}_\omega\).
	
	We now turn to the main proof of \Cref{prop: osc s c}. Fix two parameters \(c \in \mathbb{N}\) and \(\xi > 0\) and a radius \(R \geq 1\), and define 
	\begin{align*}
	\mathcal{E} = \mathcal{E}(R) &\triangleq  \big\{ \omega \in \Omega \colon\forall_{k = 1, \dots, d} \not\hspace{-0.05cm}\exists_{z \in B_{(c + 3) R}} \ \omega(z, e_k) \in (0, \xi)\big\},\\
	\mathcal{H} = \mathcal{H} (R) &\triangleq \big\{\omega \in \Omega \colon \forall_{z \in B_{R}} \not\hspace{-0.05cm}\exists_{x \in \mathbb{Z}^d} \text{ such that } z \oc x, x \not \in \mathcal{C}_\omega, \|x - z\|_\infty = \lfloor R\rfloor \big\},
	\\ \mathcal{S} = \mathcal{S}(R) &\triangleq \big\{ \omega \in \Omega \colon \forall_{x, y \in \mathcal{C}_\omega \cap B_{2 R}}\ \text{dist}_\omega (x, y) \leq cR \big\}.
	\end{align*}
	Providing an intuition, we have the following:
	\begin{enumerate}
		\item[-]
		If \(\omega \in \mathcal{E}\), the walk in \(\omega\) is elliptic in \(\mathcal{C}_\omega \cap B_{(c + 3) R}\).
		\item[-] If \(\omega \in \mathcal{H}\), when starting in \(B_R\) the worst case is that the walk in \(\omega\) is in a hole of the sink \(\mathcal{C}_\omega\) with radius \(\lfloor R\rfloor\).
		\item[-]
		If \(\omega \in \mathcal{S}\), all points in \(\mathcal{C}_\omega \cap B_{2R}\) can be reached by a walk in \(\omega\) in at least \(\lfloor cR\rfloor\) steps.
	\end{enumerate}
	We set \(G = G(R) \triangleq \mathcal{E} \cap \mathcal{S} \cap  \mathcal{H}\)
	and take \(\xi = \xi(R)\) small enough such that 
	\(
	P(\mathcal{E}^c) \leq R^{3d} e^{- R^\alpha}.
	\)
	Due to \cite[Proposition 3.1]{Berger18}, there exists a constant \(\alpha > 0\) (only depending on the dimension) such that
	\(
	P(\mathcal{H}^c) 
	\leq \c R^d e^{- R^\alpha}.
	\)
	Moreover, due to \cite[Proposition 3.2]{Berger18}, we can choose \(c\) (depending only on \(P\)) in the definition of the set \(\mathcal{S}\) (and the statement of the proposition) such that
	\(
	P(\mathcal{S}^c) \leq \c R^{3d} e^{- R^\alpha}.
	\)
	In summary, we have
	\begin{align*}
	P(G) &= 1 - P(G^c) 
	\geq 1 - P(\mathcal{E}^c) - P(\mathcal{H}^c) - P(\mathcal{S}^c) \geq 1 - \c  R^{3d}  e^{- R^\alpha}.
	\end{align*}
	
	It is left to show that the oscillation inequality \eqref{eq: smsc osc} holds \(P\)-a.e. on \(G\).
	Let \(\omega \in G\cap \mathsf{B}\) and fix \((x, t), (y, s) \in \Theta^p(K_R)\). Furthermore,  let \((Z_n)_{n \in \mathbb{N}}\) and \((Y_n)_{n \in \mathbb{N}}\) be independent walks in \(\omega\) such that \(Z_0 = x\) and \(Y_0 = y\). With abuse of notation, we denote the underlying probability measure by \(P_\omega\).
	Let \(u\) be an \(\omega\)-caloric function on \(\ok_{(c + 3) R}\). 
	Denote 
	\begin{align*}
	\tau &\triangleq \inf(n \in \N \colon (Z_n, n + t) \not \in K_{(c + 3)R}),\\
	\rho &\triangleq \inf(n \in \N\colon (Y_n, n + s) \not \in K_{(c + 3)R}).
	\end{align*}
	Then, 
	\begin{align*}
	u(x, t) - u(y, s) &= E_\omega \big[u(Z_{\tau}, \tau + t) - u(Y_{\rho}, \rho + s)\big]
	\\&= E_\omega \big[\big(u(Z_{\tau}, \tau + t) - u(Y_{\rho}, \rho + s)\big)\1_{\{(Z_{\tau}, \tau + t) \not = (Y_{\rho}, \rho + s)\}}\big]
	\\&\leq \underset{\Theta^p(K_{(c + 3) R})}{\osc} u\  P_\omega((Z_{\tau}, \tau + t) \not = (Y_{\rho}, \rho + s)).
	\end{align*}
	The oscillation inequality follows in case there exists a constant \(C = C(R) > 0\) such that 
	\begin{align}\label{eq: osc sc constant}
	P_\omega((Z_{\tau}, \tau + t) = (Y_{\rho}, \rho + s)) \geq C.
	\end{align}
	\noindent
	\emph{Case 1:} \(x, y \in \mathcal{C}_\omega\). As \(\omega \in \mathcal{E} \cap \mathcal{S}\), we can guide the space-time walks to meet at some point and afterwards to proceed together. 
	Thus, \[\xi^{2 (c + 3)^2 R^2} \leq P_\omega((Z_{\tau}, \tau + t) = (Y_{\rho}, \rho + s)).\]
	\noindent
	\emph{Case 2: \(x \not \in \mathcal{C}_\omega\) or \(y \not \in \mathcal{C}_\omega\)}. In this case we first bring the walks into the sink. Since \(\omega \in \mathcal{H}\), the worst case is that the initial points \(x, y\) are in a hole of the sink of radius \(\lfloor R\rfloor\). Furthermore, using \(\omega\in \mathsf{B}\), the walk can step in direction of the sink with probability at least \(\frac{1}{2 d}\). Consequently, again guiding the walks and using that \(\omega \in \mathcal{E} \cap \mathcal{S}\), we obtain that \[(2d)^{- 2(c + 3)^2 R^2} \xi^{2 (c + 3)^2 R^2} \leq P_\omega((Z_{\tau}, \tau + t)= (Y_{\rho}, \rho + s)).\]
	\noindent 
	Hence, \eqref{eq: osc sc constant} holds with \(C \equiv (2d)^{- 2(c + 3)^2 R^2} \xi^{2 (c + 3)^2 R^2} > 0\). The proof is complete.
\end{proof}
The following is an application of \Cref{prop: osc s c}:
\begin{corollary}\label{coro: TV bound}
	There exist constants \(\alpha > 0,  c \in \mathbb{N}\) such that for all \(R \geq 1\) there is a constant \(C = C(R) \in (0, 1)\) and a set \(G = G(R) \in \mathcal{F}\) with \(P(G) \geq 1 - \c R^{3d} e^{- R^\alpha}\) such that for all \(\omega \in G, p \in \{o, e\}\) 
	\[
	\max_{A \subseteq \pp K_{(c + 3)R}} \underset{\Theta^p(K_R)}{\osc}\ \mathfrak{p}_\omega (A) \leq C,
	\]
	where
	\begin{align*}
	\rho_t &\triangleq \inf(n \in \N \colon (X_n, n + t) \not \in K_{(c + 3)R} ),\qquad 
	\mathfrak{p}_\omega^{(x, t)} (A) \triangleq P_\omega^x ( (X_{\rho_t}, \rho_t + t) \in A). 
	\end{align*}
\end{corollary}
\begin{proof}
	The Markov property of the walk in the environment \(\omega\) yields that \(\mathfrak{p}_\omega (A)\) is \(\omega\)-caloric in \(K_{(c + 3)R}\) and consequently, \Cref{prop: osc s c} implies  the claim.
\end{proof}

\subsection{Multi-Scale Structure}\label{sec: MSS}
Let \(R_o, M, K, N, \varepsilon_o > 0\) be parameters which we will determine later. The constant \(M\) will be taken large (at least 10, say) and \(N, K \in \mathbb{N}\).

Furthermore, let  \(\{\A_1,  \dots, \A_N\}\) be a covering of \(\partial  \K_1\) intersecting only in their boundaries, which are supposed to have no measure, i.e. 
\[
\textup{meas } (\{ x \in \pp \K_1 \colon d (x, \A_i) = d (x, \pp \K_1 \backslash \A_i) = 0 \})= 0.
\] 
Moreover, we assume that 
\begin{align}\label{eq: diam assp}
\forall_{i = 1, \dots, N} \ \exists \hat{z}_i \in \pp \K_1 \colon \A_i \subset \K_{\frac{1}{4 M^2}} (\hat{z}_i).
\end{align}
In the following we will denote space-time points  by \(\hat{x}, \hat{y}, \hat{z},\) etc.
For \(R \geq 1, j \in \{1,  \dots, N\}\), \(\hat{z} \in \mathbb{Z}^d \times \mathbb{Z}_+\) and \(s \in \mathbb{R}_+\) we define 
\begin{align*}
\rho_{s}^{\hat{z}, R}  &\triangleq \inf (t \in \N \colon (X_t, t + s) \not \in K_{R} (\hat{z})),\\
\tau_s^{\hat{z}, R} &\triangleq \inf(t \in \mathbb{R}_+ \colon (X_t, t + s) \not \in  \K_R(\hat{z})),
\end{align*}
and
\begin{align*}
\mathfrak{p}_\omega^{(x, t), \hat{z}, R} (j) &\triangleq P^x_\omega ((X_{\rho_{t}^{\hat{z}, R}}, \rho_t^{\hat{z}, R} + t) \in R A_j ( \hat{z})),\\
\mathfrak{p}_\textup{BM}^{(x, t), \hat{z}, R} (j) &\triangleq \bm^{x}((X_{\tau_{t}^{\hat{z}, R}}, \tau^{\hat{z}, R}_{t} + t) \in R\A_j( \hat{z})),
\end{align*}
where \(R\A_j (x, t) \triangleq \big\{(y,s) \in \pp \K_R( x, t) \colon \big( \frac{y - x}{R}, \frac{s - t}{R^2}\big) \in \A_j\big\}\) and \(R A_j  (\hat{z})\) as in \eqref{eq: A_R}.

\begin{definition} Let \(c \in \mathbb{N}\) and \(C = C(R_o) \in (0, 1)\) be as in \Cref{coro: TV bound}.
	\begin{enumerate}
		\item[\textup{(i)}]
		For \(R \leq R_o\) we say that the cylinder \(K_R(\hat{z})\) is \(\omega\)-\emph{good}, if 
		\[
		\max_{p = o, e} \max \big( \|\mathfrak{p}_\omega^{\hat{x}, \hat{z}, (c + 3)R_o}- \mathfrak{p}_\omega^{\hat{y}, \hat{z}, (c + 3)R_o}\|_{tv} \colon \hat{x}, \hat{y} \in  \Theta^p(K_{R}(\hat{z})) \big) \leq C,
		\]
		where \(\|\cdot\|_{tv}\) denotes the total variation distance.
		\item[\textup{(ii)}]
		For \(R > R_o\) we say that the cylinder \(K_R(\hat{z})\) is \(\omega\)-\emph{good}, if for all \(\hat{x} \in K_R(\hat{z})\)
		\[
		\|\mathfrak{p}_\omega^{\hat{x}, \hat{z}, MR} - \mathfrak{p}_\textup{BM}^{\hat{x}, \hat{z}, MR}\|_{tv} < \varepsilon_o.
		\] 
	\end{enumerate}
\end{definition}
The following lemma shows that in case \(R_o\) is large the probability for a cylinder to be good is high. The lemma follows from \Cref{coro: QE,coro: TV bound}.
\begin{lemma}\label{lem: good prob}
	There exist constants \(R^* = R^*(\A_1, \dots, \A_N, \varepsilon_o, M)\geq 1\) and \(\delta > 0\) such that whenever \(R_o \geq R^*\)
	\begin{align}\label{eq: prob good}
	P\big(\big\{\omega \in \Omega \colon K_R(\hat{z}) \text{ is \(\omega\)-good}\hspace{0.02cm} \big\}\big) \geq 1 - e^{- R^{\delta}} \quad \text{for all } R \geq 1.
	\end{align}
\end{lemma}
In the following, let \(\delta\) be as in \Cref{lem: good prob}.
Our next step is to set up a multi-scale structure. Define 
\[
R_k \triangleq R_o^{K^k}\text{ for } k \in  \N,
\]
and take a constant \[\nu < \frac{\delta}{K}.\]

\begin{definition}
	\begin{enumerate}
		\item[\textup{(i)}]
		A cylinder of radius \(R^2_o\) is called \(\omega\)-\emph{admissible}, if 
		all sub-cylinders of radius \(R_o\) are \(\omega\)-good.
		\item[\textup{(ii)}]
		For \(k \in \mathbb{N}\) a cylinder of radius \(R^2_k\) is called \(\omega\)-\emph{admissible}, if  \begin{enumerate} \item[-] every sub-cylinder of radius \(>  R_{k-1}\) is \(\omega\)-good. \item[-] there are at most \(R_{k}^\nu\) non-\(\omega\)-admissible sub-cylinders of radius \(R^2_{k-1}\).\end{enumerate}
	\end{enumerate}
\end{definition}
\begin{lemma}\label{lem: prob admissible}
	There exists a constant \(R^* = R^*(\A_1, \dots, \A_N, \varepsilon_o, M) \geq 1\) such that for \(R_o \geq R^*\) the following holds: For all \((\hat{z}, k) \in \mathbb{Z}^d \times \N \times \N\)
	\[
	P \big(\big\{\omega  \in \Omega \colon K_{R^2_k}(\hat{z}) \text{ is \(\omega\)-admissible}\hspace{0.02cm}\big\}\big)  \geq  1- e^{- R_k^{\nu/2}}.
	\]
\end{lemma}
\begin{proof}
	We use induction. For \(k = 0\) the claim  follows from  \Cref{lem: good prob} and a union bound.
	For the  induction step, assume that the claim holds for \(k \in \N\). 
	Denote 
	\begin{align*}
	A &\triangleq \big\{\omega \in  \Omega \colon \text{in \(K_{R^2_{k}}(\hat{z})\) there exists a sub-cylinder }\\&\hspace{5cm} \text{of radius \(>  R_{k-1}\) which is not \(\omega\)-good}\big\},\\
	B &\triangleq \big\{ \omega \in \Omega \colon  \text{in \(K_{R^2_k}(\hat{z})\) there are more than \(R_{k}^\nu\)}\\&\hspace{5cm} \text{non-\(\omega\)-admissible sub-cylinders of radius \(R^2_{k-1}\)} \big\}.
	\end{align*}
	Due to \Cref{lem: good prob}, each sub-cylinder with radius \(> R_{k-1}\) is bad with probability less than
	\(e^{ - R_{k - 1}^\delta} \leq e^{- R_k^{\nu}}.
	\) 
	Thus, due to a union bound, 
	we obtain for \(R_o\) large enough that 
	\[
	P(A) \leq \tfrac{1}{2} e^{- R_k^{\nu/2}}.
	\]
	We denote by \(\mathsf{A}(\hat{x}, R)\) the set of all \(\omega \in \Omega\) such that \(K_R (\hat{x})\) is \(\omega\)-admissible.
	To estimate \(P(B)\), we partition the cylinder \(K_{R_k^2}\) in \(\rho \leq \text{polynomial of } R_k\) subsets \(\{U_1, \dots, U_{\rho}\}\) such that \(\mathsf{A}(\hat{x}, R^2_{k - 1})\) and \(\mathsf{A}(\hat{y}, R^2_{k - 1})\) are independent for all \(\hat{x}, \hat{y} \in U_i,  i = 1, \dots, \rho.\)
	For \(i = 1, \dots, \rho\) we have
	\begin{align*}
	Z_i \triangleq \sum_{\hat{x} \in U_i} \big(1 - \1_{\mathsf{A}(\hat{x}, R^2_{k-1})}\big) &\sim \textup{bin} (|U_i|, 1 - P(\mathsf{A}(0, R^2_{k- 1})) \leq_\textup{st} \textup{bin} (|U_i|, e^{- R^{\nu/2}_{k-1}}),
	\end{align*}
	where \(\leq_\textup{st}\) denotes the usual stochastic order.
	Note the  following:
	\begin{lemma}\label{lem: binom}
		For \(n \in \mathbb{N}\), let \(S_1, \dots, S_n\) be i.i.d. Bernoulli random variables with parameter \(p \in (0, 1)\). Then, for all \(k \in [n]\)
		\[
		P(S_1 + \dots + S_n \geq k) \leq (np)^k.
		\]
	\end{lemma}
	\begin{proof}
		We use induction over \(k \in [n]\). For \(k = 0\) the claim is obvious. Assume the claim holds for \(0 \leq k < n\). Then, 
		\begin{align*}
		P(S_1 + \dots + S_n \geq k + 1) &= P(S_1 + \dots + S_n \geq k + 1, \exists_{m \leq n} \colon S_m = 1)
		\\&\leq \sum_{m = 1}^n P(S_1 + \dots + S_n - S_m \geq k, S_m = 1)		
		\\&= \sum_{m = 1}^n P(S_1 + \dots + S_{n - 1} \geq k) P(S_m = 1)\\
		&\leq \sum_{m = 1}^n P(S_1 + \dots + S_n \geq k) p
		\\&\leq \sum_{m = 1}^n (np)^k p = (np)^{k+1}.
		\end{align*}
		The proof is complete.
	\end{proof}
	Using \Cref{lem: binom} and Chebyshev's inequality, we obtain that 
	\begin{align*}
	P (B) &\leq \sum_{i = 1}^\rho P(Z_i > \rho^{-1} R_{k}^{\nu}) 
	\leq \rho (2 R_k)^{2(d + 2) \rho^{-1} R^{\nu}_{k}} e^{- \rho^{-1} R^\nu_{k} R^{\nu/2}_{k - 1}}
	\\&= \rho e^{\rho^{-1} R^{\nu}_{k} (\log (2 R_k) 2 (d + 2) - R^{\nu/2}_{k - 1})}
	\leq \tfrac{1}{2} e^{- R_k^{\nu/2}},
	\end{align*}
	provided \(R_o\) is sufficiently large. 
	We conclude that 
	\(
	P(A \cup B) \leq e^{- R_k^{\nu/2}}.
	\)
	The proof is complete.
\end{proof}

\subsection{The Coupling}
We use the notation from \Cref{sec: MSS}.
\subsubsection{Definition}\label{sec: coupling def}
In this section we define a coupling, which \emph{success} will prove the oscillation inequality. 
We define the coupling via a (random) sequence \[\{\hat{x}^{(m)}, \hat{y}^{(m)}, \hat{z}^{(m)}, R^{(m)}, Y^{(m)}, Z^{(m)} \colon m \in \N\}.\]
The starting point is a so-called \emph{basic coupling}: 
For \(\hat{x} \in \mathbb{Z}^d\times  \N, R \geq 1\), \(\hat{y} = (\hat{y}_1, \hat{y}_2), \hat{z} = (\hat{z}_1, \hat{z}_2) \in K_R(\hat{x})\), let \(\mathfrak{q}_\omega^{(\hat{x}, R, \hat{y}, \hat{z})}\) be a Borel probability measure on the product space \((\mathbb{Z}^d \times \N) \times (\mathbb{Z}^d \times \N) \times D(\N, \mathbb{Z}^d) \times D(\N, \mathbb{Z}^d)\) such that the generic element \((\hat{Z}^1, \hat{Z}^2, X^1, X^2)\) is sampled as follows:
\begin{enumerate}
	\item[-] If \(R > R_o\), then \(X^1\) and \(X^2\) are two walks in  \(\omega\) starting at \(\hat{y}_1\) and \(\hat{z}_1\) respectively, such that the probability of \((X^1_n, \hat{y}_2 + n)_{n \in \N}\) and \((X^2_n, \hat{z}_2 + n)_{n \in \N}\) leaving \(K_{MR}(\hat{x})\) in the same element of \(\{RM A_1 (\hat{x}), \dots, RMA_N (\hat{x})\}\) is maximized. Moreover, \(\hat{Z}^1\) and \(\hat{Z}^2\) are the points where \((X^1_n, \hat{y}_2 + n)_{n \in \N}\) and \((X^2_n,  \hat{z}_2 + n)_{n \in \N}\) leave \(K_{MR}(\hat{x})\).
	\item[-] If \(R \leq R_o\), then \(X^1\) and \(X^2\) are two walks in  \(\omega\) starting at \(\hat{y}_1\) and \(\hat{z}_1\) respectively, such that the probability of \((X^1_n, \hat{y}_2 + n)_{n \in \N}\) and \((X^2_n, \hat{z}_2 + n)_{n \in \N}\) leaving \(K_{(c + 3)R_o}(\hat{x})\) in the same point is maximized. Moreover, \(\hat{Z}^1\) and \(\hat{Z}^2\) are the points where \((X^1_n, \hat{y}_2 + n)_{n \in \N}\) and \((X^2_n, \hat{z}_2 + n)_{n \in \N}\) leave \(K_{(c + 3)R_o}(\hat{x})\).
\end{enumerate}
Before we turn to the main coupling, let us explain that on good cylinders there is a reasonable probability that the walks leave a cylinder in the same region or point. 
\begin{lemma} \label{lem: suc of bc on good cyl}
	Take \(\hat{x} \in \mathbb{Z}^d \times  \N, R \geq 1\) and \(\hat{y}, \hat{z} \in K_R(\hat{x})\) of the same parity and assume that \(\omega \in \Omega\) is such that \(K_R(\hat{x})\) is \(\omega\)-good.
	\begin{enumerate}
		\item[\textup{(i)}]
		There exist two constant \(\c_1, \c_2 > 0\) only depending on the dimension \(d\) and the covariance matrix \(\a\) such that in case \(R > R_o\)
		\begin{align*}
		\mathfrak{q}_\omega^{(\hat{x}, R, \hat{y}, \hat{z})} \big(\exists \hat{v} \in \pp K_{RM} (\hat{x})\colon \hat{Z}^1, \hat{Z}^2 \in K_{RM^{-1}}(\hat{v})\big) 
		> 1 - \c_1 M^{- \c_2} - 2 \varepsilon_o.
		\end{align*}
		\item[\textup{(ii)}] If \(R \leq R_o\), then 
		\[
		\mathfrak{q}_\omega^{(\hat{x}, R, \hat{y}, \hat{z})} \big(\hat{Z}^1 = \hat{Z}^2\big) > 1 - C,
		\]
		where \(C \in (0, 1)\) is as in the definition of the good cylinder, see \Cref{coro: TV bound}.
	\end{enumerate}
\end{lemma} 
\begin{proof}
	(i). The proof is based on the relation of oscillation, total variation and couplings:
	In view of \cite[Proposition 4.7, Remark 4.8]{levin2009markov} and of assumption \eqref{eq: diam assp}, it suffices to show that 
	\[
	\| \mathfrak{p}_\omega^{\hat{y}, \hat{x}, MR} - \mathfrak{p}^{\hat{z}, \hat{x}, MR}_\omega\|_{tv} \leq \c_1M^{- \c_2} + 2 \varepsilon_o.
	\]
	Take \(k \in \{1, \dots, N\}\).
	Since \(K_R(\hat{x})\) is \(\omega\)-good, we have
	\begin{align}
	|\mathfrak{p}_\omega^{\hat{y}, \hat{x},MR} (k) - \mathfrak{p}^{\hat{z}, \hat{x},MR}_\omega (k)| &\leq |\mathfrak{p}_\omega^{\hat{y}, \hat{x}, MR} (k) - \mathfrak{p}_\textup{BM}^{\hat{y}, \hat{x}, MR} (k)| + |\mathfrak{p}^{\hat{z}, \hat{x},MR}_\omega (k) - \mathfrak{p}_\textup{BM}^{\hat{z}, \hat{x}, MR} (k)| \nonumber
	\\&\hspace{1.5cm} + |\mathfrak{p}^{\hat{y}, \hat{x},MR}_\textup{BM} (k) - \mathfrak{p}_\textup{BM}^{\hat{z}, \hat{x}, MR} (k)| \nonumber
	\\&\leq 2 \varepsilon_o + |\mathfrak{p}^{\hat{y}, \hat{x},MR}_\textup{BM} (k) - \mathfrak{p}_\textup{BM}^{\hat{z}, \hat{x}, MR} (k)|. \label{eq: triangle and good}
	\end{align}
	Furthermore, because \(\hat{v} \mapsto u(\hat{v})\triangleq \mathfrak{p}^{\hat{v}, \hat{x}, MR}_\textup{BM}\) solves the (backward) heat equation \(\frac{d}{dt} u + \tfrac{1}{2} \text{tr} (\a\nabla^2u) = 0\) on \(\K_{MR}(\hat{x})\), \cite[Theorem 6.28]{Lieberman96} yields the existence of two constants \(\c_1, \c_2 > 0\) such that 
	\begin{align*}
	|\mathfrak{p}^{\hat{y}, \hat{x},MR}_\textup{BM} (k) - \mathfrak{p}_\textup{BM}^{\hat{z}, \hat{x}, MR} (k)| \leq \c_1 M^{- \c_1}.
	\end{align*}
	Together with \eqref{eq: triangle and good}, we conclude (i).
	
	(ii). This follows from  \cite[Proposition 4.7, Remark 4.8]{levin2009markov} and the definition of a good cylinder.
\end{proof}
We can (and will) take \(M\) and \(\varepsilon_o\) such that 
\begin{align}\label{eq: 2/3 cond}
1 - \c_1 M^{- \c_2} - 2 \varepsilon_o \geq \tfrac{2}{3}.
\end{align}
In other words, on good cylinders the coupling is \emph{successful} (in some sense) with a reasonable probability.

From now on we fix \(R > R_o, \omega \in \Omega\) and two points \(\hat{y}, \hat{z} \in K_R\) of the same parity. The following are the initial values:
\begin{enumerate} 
	\item[-] \(R^{(0)} \triangleq R\).
	\item[-] \(\hat{x}^{(0)} \triangleq (0, 0)\).
	\item[-] sample \((\hat{y}^{(0)}, \hat{z}^{(0)}, Y^{(0)}, Z^{(0)})\) according to \(\mathfrak{q}_\omega^{0, R, \hat{y}, \hat{z}}\).
\end{enumerate}
Now, we proceed inductively. Namely, once the \(m^{\text{th}}\) element of the  sequence is fixed, we generate the \((m + 1)^{\text{th}}\) element as follows: Set \(R^{(m)}\) and \(\hat{x}^{(m)}\) according to the following rule:
\begin{enumerate}
	\item[-] \emph{Case 1}: \(R^{(m-1)} > R_o\). If there exists a point \(\hat{v}\) in the boundary of the cylinder \(K_{MR^{(m - 1)}}(\hat{x}^{(m - 1)})\) such that \(\hat{y}^{(m-1)}, \hat{z}^{(m-1)} \in  K_{M^{-1}R^{(m-1)}}(\hat{v})\), set \(R^{(m)} \equiv M^{-1}R^{(m-1)}\) and \(\hat{x}^{(m)} \equiv \hat{v}\). Otherwise, take \(R^{(m)} \equiv M R^{(m-1)}\) and \(\hat{x}^{(m)} \equiv \hat{x}^{(m-1)}\).
	\item[-] \emph{Case 2}: \(R^{(m-1)} \leq R_o\). We set \(R^{(m)} \equiv (c + 3) R_o\) and \(\hat{x}^{(m)} \equiv \hat{x}^{(m - 1)}\).
\end{enumerate}
Then, sample \((\hat{y}^{(m)}, \hat{z}^{(m)}, Y^{(m)}, Z^{(m)})\) according to \(\mathfrak{q}_\omega^{(\hat{x}^{(m)}, R^{(m)}, \hat{y}^{(m-1)}, \hat{z}^{(m-1)})}\).
Finally, let \(Y\) and \(Z\) be the walks in \(\omega\) that are obtained from \(Y^{(m)}\) and \(Z^{(m)}\) by pasting.
To simplify our notation, we denote the probability measure underlying the coupling by \(\q\).

\subsubsection{A Technical Lemma}
Fix \(k \in \N, \hat{x} \in \mathbb{Z}^d \times \N\) and let \(R_{k} < R  \leq R_{k + 1}\). Further, define two stopping times:
\begin{align*}
T &\triangleq \inf \big(m \in \mathbb{N} \colon R^{(m)} \leq R_k\big),\\
S &\triangleq \inf \big(m \in \mathbb{N} \colon R^{(m)} \geq R^{2}_{k + 1} \text{ or } \hat{x}^{(m)} \not \in K_{R^2_{k + 1}/2}(\hat{x})\big).
\end{align*}
\begin{remark}\label{rem: only good cylinder}
	If  \(\omega \in \mathsf{A}(\hat{x}, R^{2}_{k + 1})\), then till \(T \wedge S\) the coupling only sees \(\omega\)-good cylinders.
\end{remark}
\begin{lemma}\label{lem: heat kernel}
	\begin{enumerate}
		\item[\textup{(i)}] There exist constants \(\theta, \C > 0\) such that if \(\omega \in \mathsf{A}(\hat{x}, R^{2}_{k + 1})\), then
		\[\q (S < T) \leq \C R_{k}^{- \theta K}.\]
		\item[\textup{(ii)}]
		If \(R_{k + 1} M^{-1} \leq R\) there exist constants \(\rho, \C > 0\) such that if \(\omega \in \mathsf{A}(\hat{x}, R^{2}_{k + 1})\), then for all \(\hat{z} \in K_{R^{2}_{k + 1}}(\hat{x})\)
		\[
		\q (\hat{x}^{(T)} \in K_{R_k}(\hat{z})) \leq \C R_k^{- \rho K}.
		\]
	\end{enumerate}
\end{lemma}
\begin{proof}
	(i). Note that 
	\begin{align}
	\q(S < T) &= \q \Big(\Big\{\max_{m < T} R^{(m)} \geq R^{2}_{k + 1} \text{ or } \exists_{m < T}\colon \hat{x}^{(m)} \not \in K_{R^{2}_{k + 1}/2}(\hat{x})\Big\} \cap \Big\{S < T\Big\}\Big) \nonumber
	\\&\leq \q \Big(\max_{m < S \wedge T} R^{(m)} \geq R^{2}_{k + 1}\Big) + \q \Big( \exists_{m  < S \wedge T}\colon \hat{x}^{(m)} \not \in K_{R^{2}_{k + 1}/2}(\hat{x})\Big). \label{eq: to control}
	\end{align}
	In view of \Cref{rem: only good cylinder}, \cite[Lemma 4.10]{Berger18} yields that 
	\[
	\q \Big(\max_{m < S \wedge T} R^{(m)} \geq R^2_{k + 1}\Big) \leq 2^{- \frac{K \log(R_{k})}{\log (M)}} \equiv R_{k}^{- \theta' K} \text{ with } \theta' \triangleq \frac{\log(2)}{\log(M)}.
	\]
	To control the second term in \eqref{eq: to control}, we first consider the process
	\[
	L_m \triangleq \frac{\log (R^{(m)}) - \log (R)}{\log (M)}, \quad m \leq S \wedge T.
	\]
	Note that \((L_m)_{m \leq S \wedge T}\) has step size one and that it steps down with a probability larger than \(\frac{2}{3}\), see \Cref{lem: suc of bc on good cyl}, \Cref{rem: only good cylinder} and \eqref{eq: 2/3 cond}. Consequently, \((L_m)_{m \leq S \wedge T}\) is stochastically dominated by a biased random walk which steps down with probability \(\frac{2}{3}\). This means that there exists a sequence of i.i.d. random variables \(\xi_1, \xi_2 - \xi_1, \dots\) such that \(\q (\xi_{1} = 1) = 1 - \q(\xi_1 = - 1) = \frac{1}{3}\) and \(\q\)-a.s. on \(\{m \leq S \wedge T\}\)
	\begin{align}\label{eq: coupling L xi}
	L_{m + 1} - L_m \leq \xi_{m + 1} - \xi_m.
	\end{align}
	\begin{lemma}\label{lem: exp S wedge T}
		\[
		\E^\q \big[S \wedge T\big] \leq 3 \Big \lceil \frac{\log (R) - \log (R_k)}{\log (M)} \Big \rceil .
		\]
	\end{lemma}
	\begin{proof}
		We set 
		\[
		\tau_{a} \triangleq \inf (m \in  \mathbb{N} \colon \xi_m \leq - a),\quad a \in \N,
		\]
		It is well-known that 
		\(
		\E^\q[ \tau_a ] = 3 a, 
		\)
		which follows from the fact that \((\xi_m + \frac{m}{3})_{m \in \N}\) is a martingale and the optional stopping theorem.
		Now, set 
		\[
		a \triangleq \Big \lceil \frac{\log (R) - \log(R_k)}{\log (M)} \Big \rceil,
		\]
		and note that \(\q\)-a.s. 
		\(
		S \wedge T = S \wedge T \wedge \tau_a \leq \tau_a.
		\)
		The claim follows.
	\end{proof}
	
	Next, set 
	\[
	B \triangleq \Big\{ \max_{m < S \wedge T} R^{(m)} < R_{k + 1}^{3/2} \Big\}.
	\]
	Using again \cite[Lemma 4.10]{Berger18} yields that 
	\[
	\q(B^c) \leq R_k^{- \theta' K/2}.
	\]
	Denote \(\hat{x}^{(m)} = (x^{(m)}, t^{(m)})\) and \(\hat{x}= (x, t)\).
	By \Cref{lem: exp S wedge T} and Chebyshev's inequality, we obtain
	\begin{align*}
	\q \big(\exists_{m < S \wedge T} \colon\|x^{(m)} - x\|_2 \geq R^{2}_{k + 1}/2, B\big) &\leq \frac{2}{R^{2}_{k + 1}}\E^{\q} \Big[ \sum_{n < S \wedge T} \|x^{(n)} - x^{(n - 1)}\|_2 \1_B \Big]
	\\&\leq \frac{2}{R^{2}_{k + 1}} \E^{\q} \Big[ \sum_{n < S \wedge T} 2M R^{(n - 1)} \1_B \Big]
	\\&\leq \frac{\C}{R_{k + 1}^{1/2}} \E^{\q} \big[S \wedge T\big]
	\\&\leq \frac{\C}{R_{k + 1}^{1/2}} \frac{\log (R_{k + 1})}{\log (M)} 
	\\&\leq \C R_{k}^{-K/8}.
	\end{align*}
	Similarly, we obtain that 
	\begin{align*}
	\q \big(\exists_{m < S \wedge T} \colon |t^{(m)} - t| \geq R_{k + 1}^4/4, B\big) 
	\leq \C R_{k}^{- K/8}.
	\end{align*}
	In summary, we have
	\begin{align*}
	\q \big(\exists_{m < S \wedge T} \colon \hat{x}^{(m)} \not \in K_{R^{2}_{k + 1}/2}(\hat{x})\big) 
	\leq \C R_{k}^{- \theta K},
	\end{align*}
	for some suitable \(\theta > 0\). 
	We proved part (i).
	
	(ii). 
	Using (i) we see that 
	\begin{align*}
	\q (\hat{x}^{(T)} \in K_{R_k}(\hat{z})) &\leq \q (S <  T) + \q (\hat{x}^{(T)} \in K_{R_k}(\hat{z}), T \leq S)
	\\&\leq \C R_{k}^{- \theta K} + \q (\hat{x}^{(T)} \in K_{R_k}(\hat{z}), T \leq S).
	\end{align*}
	It remains to control the second term. 
	
	Let \((\xi_k)_{k \in \N}\) be as in the proof of part (i). For \(m \in \N\) we set
	\[
	\mathcal{R} (m) \triangleq \Big\{\xi_{m + 1} - \xi_m = - 1, \sum_{k = m + 2}^\infty M^{\xi_k - \xi_{m + 1}} < 1\Big\}.
	\]
	We obtain that on \(\mathcal{R} (m) \cap \{m + 2 \leq T \leq S\}\) 
	\begin{align*}
	\|x^{(T)} - x^{(m + 1)}\|_2 &\leq \sum_{k = m + 2}^T \|x^{(k)} - x^{(k-1)}\|_2
	\leq \sum_{k = m + 2}^T 2MR^{(k-1)} \\
	&\leq 2M R^{(m + 1)} + 2M R^{(m + 1)} \sum_{k = m + 2}^{\infty} M^{\xi_{k} - \xi_{m + 1}}
	\\&\leq 4 M R^{(m + 1)} = 4 R^{(m)},
	\end{align*}
	and similarly, 
	\(
	|t^{(T)} - t^{(m + 1)}|^\frac{1}{2} \leq 2 R^{(m)}.
	\)
	Thus, on \(\mathcal{R} (m) \cap \{m + 2 \leq T \leq S, \hat{x}^{(T)} \in K_{R_k}(\hat{z})\}\) we have 
	\(
	\|\hat{z}_1 - x^{(m + 1)}\|_2 \leq 2R_k + 4R^{(m)} \leq 6R^{(m)},
	\)
	and 
	\(
	|\hat{z}_1 - t^{(m + 1)}|^\frac{1}{2} \leq 3 R^{(m)},
	\)
	which happens with probability bounded from above by a constant \(p = p(\varepsilon_o, M) < 1\), because \((x^{(m + 1)}, t^{(m + 1)})\in (x^{(m)}, t^{(m)}) + \pp K_{M R^{(m)}}\) and the definition of a good cylinder.
	Next, we need the following large deviation estimate: 
	\begin{lemma}\label{lem: LDE}
		There exist constants \(\kappa, v > 0\) such that for all \(n \in \mathbb{N}\)
		\[
		\q \Big( \frac{1}{n} \sum_{k = 1}^n \1_{\mathcal{R}(k)} \leq \kappa \Big) \leq e^{- v n}.
		\]
	\end{lemma}
	\begin{proof}
		We call \(m \in \mathbb{N}\) a \emph{renewal}, if
		\[
		\begin{cases}
		\xi_m > \xi_n,&n > m,\\
		\xi_m < \xi_n,&n < m.
		\end{cases}
		\]
		For \(k \in \mathbb{N}\), let \(\tau_k\) be the \(k^\textup{th}\) renewal and note that \(\xi_{\tau_i + 2} - \xi_{\tau_k + 1} \leq - (i - k + 1) = k - i - 1\) for every \(i \geq k\).
		Consequently, we see that
		\begin{align*}
		\sum_{i = \tau_k + 2}^\infty M^{\xi_i - \xi_{\tau_k + 1}} &= M^{- \xi_{\tau_k + 1}} \sum_{i = k}^\infty \sum_{j = \tau_i + 2}^{\tau_{i + 1} + 1} M^{\xi_j} 
		\\&\leq M^{- \xi_{\tau_k + 1}} \sum_{i = k}^\infty (\tau_{i + 1} - \tau_i) M^{\xi_{\tau_i + 2}} 
		\leq \sum_{i = k} (\tau_{i + 1} - \tau_i) M^{k - i - 1}.
		\end{align*}
		This shows that the event \(\mathcal{R}(\tau_k)\) happens in case 
		\(
		(\tau_{i + 1} - \tau_i) M^{k - i - 1} < 2^{k - i - 1}\) for all \(i \geq k.
		\)
		Now, the proof concludes identical to the proof of \cite[Claim 4.12]{Berger18}.
	\end{proof}
	Let
	\[
	Z_N \triangleq \sum_{k = 1}^{N} \1_{\mathcal{R} (k)}, \qquad N \in \mathbb{N}.
	\]
	Note that 
	\[
	T \geq \Big\lceil \frac{K \log (R_k)}{2 \log (M)}  \Big \rceil \equiv n + 2, 
	\]
	because \(R^{(m)} \geq \frac{R}{M^m}\) and \(R \geq R_{k + 1} M^{-1}\). Now, we obtain that 
	\begin{align*}
	\q(\hat{x}^{(T)}\in  K_{R_k}(\hat{z}), T \leq S) 
	&\leq \q(Z_n \leq \kappa n) + \q(\hat{x}^{(T)}\in K_{R_k}(\hat{z}), T \leq S, Z_n >  \kappa n) 
	\\&\leq R_k^{- \frac{v}{4 \log (M)} K} + R_k^{ \frac{\log (p) \kappa}{4 \log (M)} K}.
	\end{align*}
	The claim of (ii) follows.
\end{proof}
\subsubsection{Success of the Coupling}\label{sec: success}
Let \(k \triangleq \max (n \in \mathbb{N} \colon R > R_n)\) and define the following sets:
\begin{align*}
A_1 &\triangleq \bigcap_{i = 0}^{k} \mathsf{A} (\hat{x}^{(T^{(i)})}, R_i^{2}),
\qquad
A_2 \triangleq \bigcap_{i = 0}^k \big\{T^{(i)} \leq S^{(i)}\big\},
\end{align*}
where \(T^{(k + 1)} \triangleq 0\) and 
\begin{align*}
T^{(i)} &\triangleq \inf \big(m  \geq T^{(i + 1)} \colon R^{(m)} \leq R_i\big),\\
S^{(i)} &\triangleq \inf \big(m  \geq T^{(i + 1)} \colon R^{(m)} \geq R_{i + 1}^{2} \text{ or } \hat{x}^{(m)}  \not \in K_{R_{i + 1}^{2}/2}(\hat{x}^{(T^{(i + 1)})})\big),
\end{align*}
and 
\begin{align*}
A_3 &\triangleq \big\{(Z_{\tau}, \tau + \hat{z}_1) = (Y_\rho, \rho +  \hat{y}_1)\big\},\\
A_4 &\triangleq \big\{ \{(Z_{n \wedge \tau}, n \wedge \tau + \hat{z}_1), (Y_{n \wedge \rho}, n \wedge \rho + \hat{y}_1) \colon n \in \N\} \subset K_{(c + 6)M R}\big\},
\end{align*}
where
\begin{align*}
\tau &\triangleq \hat{z}^{(T^{(0)})}_1 +  \inf \big(m \in \N \colon (Z^{(T^{(0)})}_m, m + \hat{z}^{(T^{(0)})}_1) \not \in K_{(c + 3)R_o}(\hat{x}^{(T^{(0)})})\big), \\
\rho &\triangleq \hat{y}^{(T^{(0)})}_1 + \inf \big(m \in \N \colon (Y^{(T^{(0)})}_m, m + \hat{y}^{(T^{(0)})}_1) \not \in K_{(c + 3)R_o}(\hat{x}^{(T^{(0)})})\big).
\end{align*}
Moreover, set \[A \triangleq \bigcap_{i = 1}^4 A_i.\]

We define \(\mathsf{A} (0, R^2)\) to be the set of all environments \(\omega \in \Omega\) for which every sub-cylinder of \(K_{R^2}\) with radius \(R^2_{k}\) is \(\omega\)-admissible and every sub-cylinder with radius \(> R_k\) is \(\omega\)-good.
\begin{lemma}\label{lem: suc global constant}
	If \(\delta < \rho K\) there exist two constants \(R'  \in \mathbb{N}\) and \(\zeta > 0\) such that if \(R_o \geq R'\) and \(\omega \in \mathsf{A}(0, R^{2})\), then
	\[
	\q (A) \geq \zeta.
	\]
\end{lemma}
\begin{proof}
	Using \Cref{lem: heat kernel} and the definition of admissibility, we conclude the existence of a constant \(\kappa > 0\) such that
	\begin{align*}
	\q \Big( \mathsf{A}(\hat{x}^{(T^{(n)})}, R_n^{2})^c \big|\bigcap_{i = n}^{k + 1} \mathsf{A}(\hat{x}^{(T^{(i)})}, R^{2}_{i})\Big) 
	&\leq \C R^{- \kappa}_n, \quad n \in [k],
	\end{align*}
	where with abuse of notation \(\hat{x}^{(0)} \equiv 0\) and \(R_{k + 1} \equiv R\).
	Using the elementary inequality:
	\[
	\prod_{i = 1}^n (1 - a_i) \geq 1 - \sum_{i = 1}^n a_i, \quad a_i \in (0, 1),
	\]
	we obtain that
	\[
	\q \Big( \bigcap_{i = 0}^k \mathsf{A}(\hat{x}^{(T^{(i)})}, R_i^{2})\Big) \geq 1 - \sum_{i = 0}^k \C R_i^{- \kappa} \geq 1 - \sum_{i = 0}^\infty \C  R_i^{- \kappa} = 1 - \sum_{i = 0}^{\infty} R_o^{- \kappa K^i}.
	\]
	Choosing \(R_o\) large enough, we get
	\[
	\q \Big( \bigcap_{i = 0}^k \mathsf{A}(\hat{x}^{T^{(i)}}, R_i^{2})\Big) \geq 1 - \varepsilon, 
	\]
	for a fixed \(\varepsilon \in (0, 1)\). Using \Cref{lem: heat kernel} (i), we also obtain that 
	\[
	\q (A_2) \geq 1 - \C \sum_{i =  0}^\infty R^{- \theta K}_{i} \geq 1 - \varepsilon, 
	\]
	provided \(R_o\) is large enough.
	Let 
	\[
	\mathcal{R} \triangleq \Big\{\sum_{i = 1}^\infty M^{\xi_i} < 1\Big\}.
	\]
	We note that
	\(
	\q (\mathcal{R}) > 0,
	\)
	see \Cref{lem: LDE}.
	Let \(\hat{x} = (x, t)\) and note that on \(A_1 \cap A_2 \cap \mathcal{R}\)
	\begin{align*}
	\|x^{(T^{(0)})} - x\|_2 &\leq \|x^{(1)} - x\|_2 + \sum_{i = 2}^{T^{(0)}} \|x^{(i)} - x^{(i - 1)}\|_2
	\leq 2MR + \sum_{i = 2}^{T^{(0)}} 2M R^{(i - 1)}
	\\&\leq 2MR + 2M R \sum_{i = 1}^{\infty} M^{\xi_{i}}
	\leq 4 MR.
	\end{align*}
	Similarly, we see that \(|t^{(T^{(0)})} - t| \leq 4 M^2R^2\) on \(A_1 \cap A_2 \cap \mathcal{R}\).
	Hence, 
	\(
	\q(A_1 \cap A_2 \cap A_4) \geq \q(A_1 \cap A_2 \cap \mathcal{R}).
	\)
	Due to \Cref{lem: suc of bc on good cyl}, we also have 
	\(
	\q (A_3 | A_1 \cap A_2 \cap A_4) \geq C.
	\)
	Finally, we conclude that 
	\(
	\q (A) 
	\geq  C (\q(\mathcal{R}) - 2 \varepsilon ).
	\)
	Taking \(\varepsilon\) small enough completes the proof.
\end{proof}
\subsection{Proof of \Cref{theo: OI}}
Let \((Z_n)_{n \in \N}\) and \((Y_n)_{n \in \N}\) be the coupled processes as defined in \Cref{sec: coupling def} and let \(\tau\) and \(\rho\) be as in \Cref{sec: success}. Take \(\omega \in \mathsf{A}(0, R^{2})\) and let \(u \colon K_{(c + 6)M R}\to \mathbb{R}\) be \(\omega\)-caloric. Now, we have 
\begin{align*}
u(\hat{z}) - u(\hat{y}) 
&\leq \q((Z_\tau, \tau + \hat{z}_1) \not = (Y_\rho, \rho + \hat{y}_1)) \underset{\Theta^p(K_{(c + 6)MR})}{\osc} u
\leq (1 - \zeta) \underset{\Theta^p(K_{(c + 6) MR})}{\osc} u, 
\end{align*}
where \(\zeta> 0\) is as in \Cref{lem: suc global constant} and \(p\) is the parity of \(\hat{z}\) and \(\hat{y}\). In view of \Cref{lem: prob admissible}, this proves \Cref{theo: OI}. \qed

\section{Proof of the PHI: \Cref{theo: PHI}}\label{sec: Pf PHI}
\subsection{Some Notation}\label{sec: pf PHI notation}
In the following, fix a parity \(p \in \{o, e\}\).
Take \(\nu \in (0, 1)\) and \(N \in \mathbb{N}\), and let \(\{\A_1, \dots, \A_N\}\) be a covering of \(\pp \K_1\) and let \(\{\mathbb{C}_1, \dots, \mathbb{C}_N\}\subset \pp \K_1\). Further, we suppose that
\[
\max_{i = 1, \dots, N} \text{diam} (\A_i) \leq \tfrac{\nu}{4}. 
\]
We assume that the boundary of each \(\A_i\) and \(\mathbb{C}_i\) has zero measure.

Let \(\zeta > 1\) and \(\gamma \in (0, 1)\) be as in \Cref{theo: OI} and let \(\mathcal{O}_{\hat{x}, R}\) be the set of all \(\omega \in \Omega\) such that the oscillation inequality
\[
\underset{\Theta^p (K_R(\hat{x}))} \osc u \leq \gamma \underset{\Theta^p (K_{\zeta R}(\hat{x}))} \osc u
\]
holds for all \(\omega\)-caloric functions \(u\) on \(\ok_{\zeta R} (\hat{x})\).

For \(i = 1, \dots, N\) define \(\chi_{\hat{x}, R, i}\) and \(\Psi_{\hat{x}, R, i}\) as in \eqref{eq: exit prob def} with \(\A\) replaced by \(\A_i\) and \(\oK_R\) replaced by \(\oK_R(\hat{x})\).
Moreover, set
\(
\alpha \triangleq 2 - \varepsilon
\)
and 
\[
\mathcal{U}_{\hat{x}, R} \triangleq \Big\{ \omega \in \Omega \colon  \forall_{\hat{y} \in K_{R}(\hat{x})} \forall_{i = 1, \dots, N}\ \frac{|\Phi_{\hat{x}, \alpha R, i} (\hat{y}) - \chi_{\hat{x}, \alpha R, i} (\hat{y})|}{\chi_{\hat{x}, \alpha R, i} (\hat{y})} \leq \varepsilon\Big\}.
\]
The \(\omega\)-dependence in the above definition stems from \(\Phi\).

For \(i = 1, \dots, N\) define \(\chi^*_{\hat{x}, R, i}\) and \(\Phi^*_{\hat{x}, R, i}\) as in \eqref{eq: exit prob def} with \(\A\) replaced by \(\mathbb{C}_i\) and \(\oK_R\) replaced by \(\oK_R(\hat{x})\).
Fix a \(\theta_1, \dots, \theta_N > 1\) and \(\delta^*\in (0, 1)\), and set 
\[
\mathcal{U}^*_{\hat{x}, R} \triangleq \big\{ \omega \in \Omega \colon  \forall_{\hat{y} \in K_{R}(\hat{x})} \forall_{i = 1, \dots, N}\ |\Phi^*_{\hat{x}, \theta_i R, i} (\hat{y}) - \chi^*_{\hat{x}, \theta_i R, i} (\hat{y})| \leq \delta^*\big\}.
\]

Let \(\kappa \in (0, \frac{1}{2d})\). 
Define the map \(\widetilde{J} \colon \Omega \to \Omega, \widetilde{J}(\omega) \triangleq \widetilde{\omega} \) as follows: For \(x \in \mathbb{Z}^d\) and \(i = 1, \dots, 2d\)
set 
\begin{align*}
\widetilde{\omega} (x, e_i) \triangleq \begin{cases} 0, & \omega(x, e_i) < \kappa,\\
\omega(x, e_i) + \frac{M}{N},& \omega(x, e_i) \geq \kappa,\end{cases}
\end{align*}
where \(N \triangleq \sum_{i = 1}^{2d} \1_{\{\omega(x,  e_i) \geq \kappa\}}\) and \(M \triangleq \sum_{i = 1}^{2d} \omega (x, e_i) \1_{\{\omega(x, e_i) < \kappa\}}\).

Next, take \(\delta \in (\xi, \frac{1}{5})\), where \(\xi \in (0, \frac{1}{5})\) is as in the statement of \Cref{theo: PHI}, and define 
\begin{align*}
\mathcal{J}_{R} &\triangleq \Big\{  \omega \in \Omega \colon \forall_{y \in B_{2R}} \not\hspace{-0.05cm} \exists_{z \in \mathbb{Z}^d} \text{ such that } y \xrightarrow{\ \widetilde{\omega} \ } z, z \not \in \mathcal{C}_{\widetilde{\omega}}, \|z - x\|_\infty = \lfloor R^{\xi}\rfloor \Big\},
\\
\mathcal{I}_R 
&\triangleq \Big\{  \omega \in \Omega \colon \forall_{y \in B_{2R}} \text{ all self-avoiding paths in \(\omega\) with length \(\lfloor \mathfrak{o} R^{\xi/4}\rfloor\)}\\&\hspace{6cm}\text{ and starting value \(y\) have visited \(\mathcal{C}_{\widetilde{\omega}}\)}\Big\},\\
\mathcal{S}_{R} &\triangleq \Big\{ \omega \in \Omega \colon \forall_{x, y \in \mathcal{C}_{\widetilde{\omega}} \cap B_{3 R^\delta}}\ \text{dist}_{\widetilde{\omega}} (x, y) \leq \C R^\delta \Big\},
\end{align*}
where \(\C, \mathfrak{o} > 0\) are constants determined in the following lemmata. 
The proof of the next lemma is given in \Cref{sec: pf of main lemma repair} below.
\begin{lemma}\label{lem: main lemma repair}
	If \(\kappa\) is small enough, \(\mathfrak{o} > 0\) can be chosen such that there are constants \(R' > 0, \c_1, \c_2 > 0 \) and \(\zeta > 0\) such that for all \(R \geq R'\)
	\[
	P(\mathcal{I}_R) \geq 1 - \c_1 e^{- \c_2 R^\zeta}.
	\]
\end{lemma}

Finally, define 
\[
\mathcal{Z}_{R} \triangleq \bigcap_{\hat{y} \in K_{R}} \bigcap_{r \in (R^\delta, R]} \big[ \mathcal{O}_{\hat{y}, r} \cap \mathcal{U}_{\hat{y}, r} \cap \mathcal{U}^*_{\hat{y}, r}\big] \bigcap  \mathcal{J}_{R} \bigcap \mathcal{I}_R \bigcap \mathcal{S}_R.
\]

\begin{lemma}\label{lem: streched expo}
	If \(\kappa\) is small enough, \(\C > 0\) can be chosen such that there are constants \(R' > 0, \c_1, \c_2 > 0 \) and \(\zeta > 0\) such that for all \(R \geq R'\)
	\[
	P(\mathcal{Z}_R) \geq 1 - \c_1 e^{\c_2 R^\zeta}.
	\]
\end{lemma}
\begin{proof}
	If \(\kappa\) is small enough the probability measure \(P \circ \widetilde{J}^{-1}\) is balanced and genuinely \(d\)-dimensional. Thus, the claim follows from \Cref{theo: OI}, \Cref{coro: QE},  \cite[Propositions 3.1 and 3.2]{Berger18} and \Cref{lem: main lemma repair} with a union bound. 
\end{proof}

From now on, we will  assume that \(R \geq R'\) and that \(\omega \in \mathcal{Z}_R \cap \mathsf{B}\). It might be that we enlarge \(R'\) even further. Under these assumptions we will prove the parabolic Harnack inequality, which completes the proof of \Cref{theo: PHI}.

\subsection{The Proof}
Let \(u\) be a non-negative \(\omega\)-caloric function satisfying the growth condition \eqref{eq: growth cond} in \Cref{theo: PHI}.
For contradiction, assume that \(\hat{x}^* \in \Theta^p(K^+_R)\) and \(\hat{y}^* \in \Theta^p(K^-_R)\) satisfy
\begin{align}\label{eq: cont assp PHI}
u(\hat{x}^*) \geq \frac{(1 + 3 \varepsilon)H_{2 - \varepsilon} u(\hat{y}^*)}{(1 - \varepsilon)^2}, \qquad H \equiv H_{2 - \varepsilon} = H_\alpha.
\end{align}
Furthermore, let \(K^+_{(2 - \varepsilon/2) R}\) be the discrete version of \(\B_{(2 - \varepsilon/2) R} \times (2 R^2, (2 - \varepsilon/2)^2 R^2)\) and let \(K^{++}_{2R}\) be the discrete version of \(\B_{2R} \times (1.5R^2,4 R^2)\).

The proof of \Cref{theo: PHI} is based on the following three lemmata:
\begin{lemma}\label{lem: key lemma PHI}
	There exists a constant \(M \geq 1\) such that every subcylinder of \(K_{2R}^{++}\) with radius \(R^\delta\) contains a point \(\hat{z}\) of the same parity as \(\hat{x}^*\) such that 
	\[
	u(\hat{z}) \leq M u (\hat{y}^*).
	\]
	Moreover, there exists a subcylinder of \(K^+_{(2 - \varepsilon/2)R}\) with radius \(R^\delta\) which contains a point \(\hat{x}\) of the same parity as \(\hat{x}^*\) such that 
	\[
	u(\hat{x}) \geq M u (\hat{y}^*) 2^{\c R^{\frac{1 - \delta}{2}}}.
	\]
\end{lemma}
From now on let \(\hat{x} = (x, t)\) be as in \Cref{lem: key lemma PHI}. 
\begin{lemma}\label{lem: strong sink bound}
	For every \(\hat{z} \in \mathcal{C}_{\widetilde{\omega}} \times \mathbb{Z}_+ \cap K_{(2 - \varepsilon/2) R}^+\) with the same parity as \(\hat{x}^*\) it holds that \[u (\hat{z}) \leq M u(\hat{y}^*) \kappa^{- 2R^{2\delta}}.\]
\end{lemma}
Noting that \((1 - \delta)/2 > \frac{2}{5}\) and \(2 \delta < \frac{2}{5}\), we see that \(\hat{x} \not \in \mathcal{C}_{\widetilde{\omega}} \times \mathbb{Z}_+\).
Let 
\begin{align}\label{eq: def T sink}
T \triangleq \inf (n \in \mathbb{Z}_+ \colon (X_n, n + t) \not \in K_{2R} \text{ or } X_n \in \mathcal{C}_{\widetilde{\omega}}).
\end{align}
\begin{lemma}\label{lem: prob strong sink bound}
	\(P_\omega^{x} (X_T \not \in \mathcal{C}_{\widetilde{\omega}}) \leq \mathfrak{w}^{- R^{2 - \xi}}.\)
\end{lemma}
Next, we put these pieces together. The optional stopping theorem and \Cref{lem: strong sink bound,lem: prob strong sink bound} yield that 
\begin{align*}
u (\hat{x}) 
&= E^x_\omega \big[ u (X_T, T + t) \1_{\{X_T \in \mathcal{C}_{\widetilde{\omega}}\}}\big] + E^x_\omega \big[ u (X_T, T + t) | X_T \not \in \mathcal{C}_{\widetilde{\omega}} \big] P_\omega^x (X_T \not \in \mathcal{C}_{\widetilde{\omega}}) 
\\  &\leq M u (\hat{y}^*) \kappa^{- 2R^{2 \delta}} + E^x_\omega \big[ u (X_T, T + t) | X_T \not \in \mathcal{C}_{\widetilde{\omega}} \big] \mathfrak{w}^{- R^{2 - \xi}}.
\end{align*}
Now, rearranging and using \Cref{lem: key lemma PHI} shows that 
\[
E^x_\omega \big[ u (X_T, T + t) | X_T \not \in \mathcal{C}_{\widetilde{\omega}} \big] \geq M u (y^*) \big( 2^{\c R^\frac{1 - \delta}{2}} - \kappa^{-2R^{2 \delta}}\big) \mathfrak{w}^{R^{2 - \xi}}.
\]
Since \(2^{\c R^\frac{1 - \delta}{2}} - \kappa^{-2R^{2 \delta}} > 1\) for large enough \(R\), we obtained a contradiction to the growth assumption \eqref{eq: growth cond}. Except for the proofs of \Cref{lem: key lemma PHI,lem: strong sink bound,lem: prob strong sink bound}, which are given in the next subsection, the proof of \Cref{theo: PHI} is complete. \qed

\subsection{Proof of \Cref{lem: main lemma repair}}\label{sec: pf of main lemma repair}
For \(z \in \mathbb{Z}^d\) and \(n \in \mathbb{N}\) we write 
\[
C_n (z) \triangleq [-n, n]^d + (2n + 1) z.
\]
Adapting terminology from \cite{Berger18}, we call \(C_n (z)\) to be \(\omega\)-\emph{good}, if \(C_{2n} (z)\) contains a unique sink and for every \(x \in C_n(z)\) every self-avoiding path in \(\omega\) of length \(\geq n/10\) reaches the unique sink in \(C_n(z)\), cf. \cite[Lemma 3.6]{Berger18}.

Fix a small \(\epsilon > 0\). Then, by \cite[Lemma 3.6]{Berger18} there exists an \(N \in \mathbb{N}\) such that
\[
P \big(\big\{ \omega \in \Omega \colon C_N(z) \text{ is \(\omega\)-good} \big\}  \big) \geq 1 - \frac{\epsilon}{2}.
\]
Next, take the parameter \(\kappa \in (0, \frac{1}{2d})\) in the definition of \(\widetilde{J}\) small enough such that 
\[
P \big(\big\{ \omega \in \Omega \colon \exists_{i = k, \dots, 2d}\ \omega(0, e_k) < \kappa \big\}\big)\leq  \frac{\varepsilon}{2 |C_{2N}(0)|}.
\]
Let \(\mathcal{C}^{z}_\omega\) be a sink in \(Q_{2N}(z)\) w.r.t. the environment \(\omega\). In case there are several sinks, take one in an arbitrary manner.
Now, we have
\begin{align*}
P\big(\big\{  \omega \in \Omega \colon  \mathcal{C}^{z}_\omega = \mathcal{C}^{z}_{\widetilde{\omega}}\big\} \big) &\geq P \big(\big\{\omega \in  \Omega \colon \forall_{y \in C_{2N} (z)} \forall_{k= 1, \dots, 2d} \  \omega(y, e_k) \geq \kappa \big\} \big) 
\\&\geq 1 - |C_{2N}(z)| P \big( \big\{ \omega \in \Omega \colon  \exists_{k =  1,  \dots, 2d} \ \omega (0, e_k) < \kappa \big\} \big)
\geq 1 - \frac{\varepsilon}{2}.
\end{align*}
We say that \(C_{2N} (z)\) is \emph{very \(\omega\)-good}, if it is \(\omega\)-good and \(\mathcal{C}^{z}_\omega = \mathcal{C}^{z}_{\widetilde{\omega}}\).
Now, define the \(\{0,1\}\)-valued random variables
\[
G_z (\omega) \triangleq \1 \{C_{2N} (z) \text{ is very \(\omega\)-good}\}, \quad z \in \mathbb{Z}^d,
\]
and note that
\[
P(G_z = 1)\geq  1- \varepsilon.
\]
As the environment measure \(P\) is an i.i.d. measure, the random variables \(G_z\) and \(G_y\) are independent whenever \(\|x - y\|_\infty \geq 2\). Consequently, we can apply \cite[Theorem 0.0]{liggett1997} and conclude that in case we have chosen \(\varepsilon\) small enough from the beginning, the family \((G_z)_{z \in \mathbb{Z}^d}\) stochastically dominates supercritical Bernoulli site percolation. 
With abuse of notation, this means that the percolation process and \((G_z)_{z \in \mathbb{Z}^d}\) can be realized on the same probability space such that a.s.
\[
\1 \{z \in \mathcal{D} \} \leq G_z, \quad z \in \mathbb{Z}^d, 
\] 
where \(\mathcal{D}\) is the (a.s. unique) infinite cluster of the supercritical percolation process. 
Thus, we note that a.s.
\[
\bigcup_{z \in \mathcal{D}} \mathcal{C}^z_\omega \subseteq \mathcal{C}^z_{\widetilde{\omega}}.
\]
Denote by \(A_z\) the connected component of \(z\) in \(\mathbb{Z}^d \backslash \mathcal{D}\). In case \(z \in \mathcal{D}\) we have \(A_z = \emptyset\).
Furthermore, set 
\[
K^z_\omega \triangleq \begin{cases}|A_z \cup \partial A_z|,& z \not \in \mathcal{D},\\
1,& z \in \mathcal{D}.\end{cases}
\]
Of course, the \(\omega\)-dependence stems from \(\mathcal{D}\).
\begin{lemma}\label{lem: NoamLem1}
	For a.a. \(\omega\) every self-avoiding path in \(\omega\) with length \(|C_{2N} (0)| K^z_\omega + 1\) and starting value \(z\in \mathbb{Z}^d\) must have visited \(\mathcal{C}_{\widetilde{\omega}}\).
\end{lemma}
\begin{proof}
	Note that every self-avoiding path in \(\omega\) with length \(|C_{2N} (0)| K^z_\omega + 1\) and starting value \(z\) has visited a point \(y\) such that 
	\[
	y \not \in \bigcup_{u \in A_z \cup \partial A_z} C_{N}(u).
	\]
	Consequently, the path has crossed a cube \(C_{N} (u)\) with \(u \in \partial A_z\). As \(\partial A_z \subset \mathcal{D}\), for a.a. \(\omega \in \Omega\) we have \(G_u (\omega) = 1\) and the definition of very \(\omega\)-good implies that the path must have visited \(\mathcal{C}_{\widetilde{\omega}}\).
\end{proof}

The following lemma follows from \cite[Theorem 8, Remark 10]{guozeitouni12}.
\begin{lemma}\label{lem: NoamLem2}
	There exists an \(\alpha > 0\) such that for all \(k \in \mathbb{N}\)
	\[P (|A_z| \geq  k) \leq \c e^{- k^\alpha}.\]
\end{lemma}
Now, for \(\mathfrak{o} \triangleq 2 |C_{2N} (0)|\) and large enough \(R\), \Cref{lem: NoamLem1,lem: NoamLem2} yield that
\begin{align*}
P(\mathcal{I}_R) &\geq P(\mathcal{I}_R, \forall_{z \in B_{2R}}\ |A_z| < \lfloor R^{\xi/4}\rfloor )
\\&= P(\forall_{z \in B_{2R}}\ |A_z| < \lfloor R^{\xi/4}\rfloor )
\\&\geq 1 - \c |B_{2R}| e^{- \lfloor R^{\xi/4}\rfloor^\alpha}.
\end{align*}
This bound completes the proof of \Cref{lem: main lemma repair}. \qed
\subsubsection{Proof of \Cref{lem: key lemma PHI}}
The proof is based on an iterative scheme in the spirit of an argument by Fabes and Stroock \cite{Fabes1989}.
Let \((r_n)_{n \in \mathbb{Z}_+}\) be a sequence of radii defined as follows: 
\[
r_0 \triangleq R, \qquad  r_1 \triangleq \frac{\varepsilon R}{8 \alpha},\qquad r_n \triangleq \frac{r_1}{n^2}.
\]
Note that 
\[
\sum_{n =  0}^\infty \alpha r_n = \alpha R + \frac{\varepsilon R}{8} \sum_{n = 1}^\infty \frac{1}{n^2} \leq \alpha R + \frac{\varepsilon R}{4} < \Big(2 - \frac{\varepsilon}{2}\Big) R,
\]
and that 
\[
\sum_{n = 0}^\infty \alpha^2 r^2_n = \alpha^2 R^2 + \frac{\varepsilon^2 R^2}{64} \sum_{n = 1}^\infty \frac{1}{n^4} < \Big(2 - \frac{\varepsilon}{2}\Big)^2 R^2.
\]
Set \(\k \triangleq \max (n \in \mathbb{N} \colon r_n > R^\delta)\) and note that
\begin{align}\label{eq: K cond}
\frac{r_1}{n^2} > R^\delta \quad \Leftrightarrow\quad \sqrt{\frac{\varepsilon}{8 \alpha}} R^{\frac{1-\delta}{2}} > n\quad \Rightarrow\quad \k \geq \Big\lfloor \sqrt{\frac{\varepsilon}{8 \alpha}}R^{\frac{1-\delta}{2}} \Big\rfloor.
\end{align}
Next, we construct two sequences \((\hat{x}_n)_{n \in [\k]}\) and \((\hat{y}_n)_{n \in [\k]}\) of points in  \(K_{(2-\varepsilon/2)R}\) with the same parity as \(\hat{x}^*\). As initial points we take
\(\hat{x}_0 \triangleq \hat{x}^*\) and \(\hat{y}_0 \triangleq \hat{y}^*\). 

Before we explain mathematically how \(\hat{x}_{n + 1}\) and \(\hat{y}_{n + 1}\) are chosen once \(\hat{x}_n\) and \(\hat{y}_n\) are known,
we describe the idea in an informal manner, see also \Cref{fig: FS explain}.
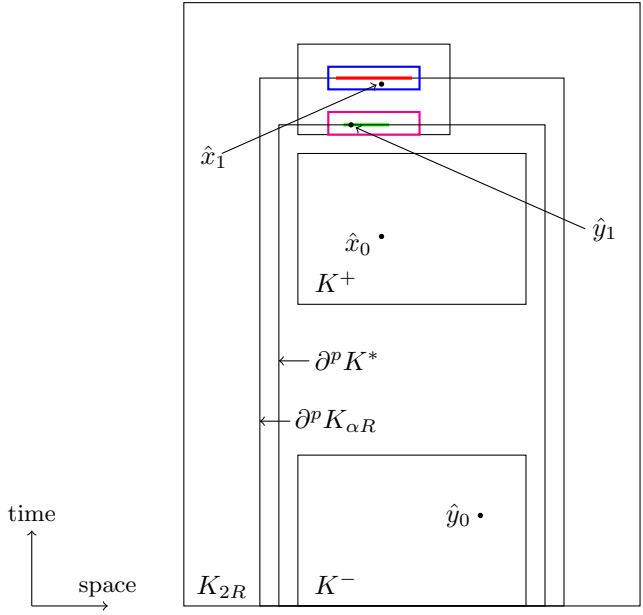
\begin{SCfigure}
	\caption{An illustration of the first step in the iteration procedure.}
	\label{fig: FS explain} 
	\begin{tikzpicture}
	\draw (0,0) rectangle (6,8);
	\draw (1.5, 0) rectangle (4.5, 2);
	\draw (1.5, 4) rectangle (4.5, 6);
	\draw (1,0) rectangle (5,7);
	\node at (0.5, 0.25) { \(K_{2R}\) };
	\node at (2, 0.3) { \(K^-\) };
	\node at (2, 4.3) { \(K^+\) };
	\draw[<-] (1, 2.45) -- (1.4, 2.45);
	\node at (2, 2.45) { \(\pp K_{\alpha R}\) };
	\draw[<-] (1.25, 3.25) -- (1.65, 3.25);
	\node at (2.15, 3.25) { \(\pp K^*\) };
	\fill (3.9, 1.2) circle (0.035);
	\node[left] at (3.9, 1.2) { \(\hat{y}_0\)};
	\fill (2.6, 4.9) circle (0.035);
	\node[left] at (2.6, 4.8) { \(\hat{x}_0\)};
	\draw[red, very thick] (2, 7) -- (3, 7);
	\draw (1.5,6.25) rectangle (3.5,7.45);
	\draw[blue, thick] (1.9,6.85) rectangle (3.1,7.15);
	\draw[magenta, thick] (1.9,6.25) rectangle (3.1,6.55);
	\fill (2.6, 6.92) circle (0.035);
	\draw[->] (0.5,6) -- (2.535,6.91);
	\node[left] at (0.7, 5.95) {\(\hat{x}_1\)};
	\draw[green, very thick] (2.1, 6.38) -- (2.7, 6.38);
	\draw[->] (5.28,5) -- (2.26,6.345);
	\node[right] at (5.25,5) {\(\hat{y}_1\)};
	\fill (2.2, 6.38) circle (0.035);
	\draw (1.25, 0) rectangle (4.75, 6.38);
	\draw[->] (- 2, 0) -- (-2, 1);
	\draw[->] (-2, 0) -- (-1, 0);
	\node[above] at (-2, 1) { \small time};
	\node[above] at (-1, 0) { \small space};
	\end{tikzpicture}
\end{SCfigure}
The initial step is to show the existence of a subset \(\alpha R A_k\) (red in \Cref{fig: FS explain}) of \(\pp K_{\alpha R}\) with the two properties that it can be reached by the space-time walk starting at \(\hat{x}_0\) and that \(\max_{\alpha R A_k} u\) and \(\max_{\alpha R A_k} u/ \min_{\alpha R A_k} u\) are reasonably large compared to \(u(\hat{x}_0)\) and \(u(\hat{x}_0)/ u(\hat{y}_0)\), respectively. 
The oscillation inequality shows the existence of a cylinder \(K^u\) (blue in \Cref{fig: FS explain}) containing \(\alpha R A_k\) in which the ratio \(\max_{K^u} u/\max_{\alpha R A_k} u\) is reasonably large. Using these properties, we obtain that
\[
\max_{K^u} u \gg \max_{\alpha R A_k} u\ \gg u(\hat{x}_0),
\]
where \(b \gg a\) means that \(b\) is in some sense larger than \(a\).
We now take \(\hat{x}_1\) to be the point in \(K^u\) (with the correct parity) where \(u\) attains its maximum. The next step then is to chose \(\hat{y}_1\) and to iterate. Before we comment on how \(\hat{y}_1\) is chosen, let us stress that the sequence \((\hat{x}_n)_{n \in [\k]}\) grows fast and the terminal point \(\hat{x}_\k\) will have the properties as described in the second part of \Cref{lem: key lemma PHI}.
Since we want to iterate, the point \(\hat{y}_1\) should be an element of a shifted version of \(K^u\), say \(K^l\) (magenta in \Cref{fig: FS explain}). Suppose that \(\theta_k R C_k\) (green in \Cref{fig: FS explain}) is a subset of \(K^l\) and part of the boundary of a cylinder \(K^*\). We will chose \(\theta_k\) and \(C_k\) such that the space-time walk starting at \(\hat{y}_0\) has a reasonable probability of exiting \(K^*\) through \(\theta_k R C_k\). Then, we take \(\hat{y}_1\) to be the point in \(\theta_k R C_k\) (with the correct parity) where \(u\) attains its minimum. 
We proceed the iteration up to time \(\k\).

We now make this precise.
The first step is based on the definition of \(\mathcal{U}_{\hat{z}_n, r_n}\). Due the Harnack inequality for Brownian motion, we have  
\[
\chi_{\hat{z}_n, \alpha r_n, i}(\hat{x}_n) \leq  H \chi_{\hat{z}_n, \alpha r_n, i}(\hat{y}_n).
\]
Using \(\omega \in \mathcal{U}_{\hat{z}_n, r_n}\), we obtain 
\[
\frac{\Phi_{\hat{z}_n, \alpha r_n, i} (\hat{x}_n)}{\chi_{\hat{z}_n, \alpha r_n, i}(\hat{x}_n)} = 1 + \frac{\Phi_{\hat{z}_n, \alpha r_n, i} (\hat{x}_n) - \chi_{\hat{z}_n, \alpha r_n, i}(\hat{x}_n)}{\chi_{\hat{z}_n, \alpha r_n, i}(\hat{x}_n)} \leq 1 + \varepsilon.
\]
Similarly, we see that \(\Phi_{\hat{z}_n, \alpha r_n, i} (\hat{y}_n) \geq (1 - \varepsilon) \chi_{\hat{z}_n, \alpha r_n, i} (\hat{y}_n)\). Therefore, we obtain 
\begin{equation}\label{eq: help1}
\begin{split}
\Phi_{\hat{z}_n, \alpha r_n, i} (\hat{x}_n) &\leq (1 + \varepsilon) \chi_{\hat{z}_n, \alpha r_n, i}(\hat{x}_n) \\&\leq H (1 + \varepsilon)\chi_{\hat{z}_n, \alpha r_n, i}(\hat{y}_n) \\&\leq \frac{(1 + \varepsilon)H \Phi_{\hat{z}_n, \alpha r_n, i}(\hat{y}_n)}{1 - \varepsilon}.
\end{split}
\end{equation}
In the following we use the short notation \(A_{n, k} \triangleq \Theta^p(\alpha r_n A_k(\hat{z}_n))\).
\begin{lemma}\label{lem: existence k}
	There exists \(k \in \{1, \dots, N\}\) such that 
	\[
	\max_{A_{n , k}} u > \frac{\varepsilon u(\hat{x}_n)}{1 + 3 \varepsilon} , \quad \max_{A_{n , k}} u > \frac{(1 - \varepsilon) \min_{A_{n , k}} u}{ (1 + 3\varepsilon) H} \frac{u(\hat{x}_n)}{u(\hat{y}_n)}.
	\]
\end{lemma}
\begin{proof}
	We denote 
	\[
	\Theta \triangleq \Big\{k \in \{1, \dots, N\} \colon \max_{A_{n , k}} u > \frac{\varepsilon u(\hat{x}_n)}{1 + 3 \varepsilon}\Big\}.
	\]
	Note that 
	\begin{align}\label{eq: ineq1}
	\sum_{k \not \in \Theta} \max_{A_{k, n}} u\ \Phi_{z_n, \alpha r_n, k}(\hat{x}_n)  \leq \frac{\varepsilon}{1 + 3 \varepsilon} \sum_{k \not \in \Theta} u(\hat{x}_n) \Phi_{z_n, \alpha r_n, k} (\hat{x}_n) \leq \frac{\varepsilon u(\hat{x}_n)}{1 + 3 \varepsilon} .
	\end{align}
	This yields that
	\begin{align*}
	\sum_{k \in \Theta} \max_{A_{n , k}} u\ \Phi_{z_n, \alpha r_n, k} (\hat{x}_n) &= \sum_{k =  1}^N \max_{A_{n , k}} u\ \Phi_{z_n, \alpha r_n, k} (\hat{x}_n) - \sum_{k \not \in \Theta} \max_{A_{n , k}} u\ \Phi_{z_n, \alpha r_n, k} (\hat{x}_n) \\&\geq \frac{1 + 2\varepsilon}{1 + 3 \varepsilon} u(\hat{x}_n).
	\end{align*}
Since the last term is positive, we conclude that \(\Theta \not = \emptyset\). 
	
	For contradiction, assume that for all \(k \in \Theta\)
	\begin{align}\label{eq: ineq2}
	\ \max_{A_{n , k}} u \leq \frac{(1 - \varepsilon)\min_{A_{n , k}} u}{(1 + 3\varepsilon) H} \frac{u(\hat{x}_n)}{u(\hat{y}_n)}.
	\end{align}
	Using the optional stopping theorem, \eqref{eq: help1}, \eqref{eq: ineq1} and \eqref{eq: ineq2} yields that 
	\begin{align*}
	u(\hat{x}_n) 
	&\leq \sum_{k \in \Theta} \max_{A_{n , k}} u\ \Phi_{z_n, \alpha r_n, k} (\hat{x}_n) + \sum_{k  \not \in \Theta} \max_{A_{n , k}} u \ \Phi_{z_n, \alpha r_n, k} (\hat{x}_n)
	\\
	&\leq \sum_{k \in \Theta} \max_{A_{k, n}} u\ \Phi_{z_n, \alpha r_n, k} (\hat{x}_n) + \frac{\varepsilon u(\hat{x}_n)}{1 + 3 \varepsilon}
	\\&\leq \sum_{k \in \Theta} \frac{(1 - \varepsilon)\min_{A_{n , k}} u}{(1 + 3\varepsilon) H} \frac{u(\hat{x}_n)}{u(\hat{y}_n)} \Phi_{z_n, \alpha r_n, k} (\hat{x}_n) + \frac{\varepsilon u(\hat{x}_n)}{1 + 3 \varepsilon}
	\\&\leq \frac{(1 + \varepsilon) u(\hat{x}_n)}{(1 + 3 \varepsilon)u(\hat{y}_n)}\sum_{k \in \Theta} \min_{A_{n , k}} u\ \Phi_{z_n, \alpha r_n, k} (\hat{y}_n) + \frac{\varepsilon u(\hat{x}_n)}{1 + 3 \varepsilon}
	\\&\leq \frac{(1 + 2 \varepsilon) u(\hat{x}_n)}{1 + 3 \varepsilon}.
	\end{align*}
	This is a contradiction. The proof is complete.
\end{proof}
Let \(k \in \{1, \dots, N\}\) be as in \Cref{lem: existence k} and take \(\hat{z}_{n + 1} \in K_{2R}\) such that 
\[
A_{n, k} \subseteq K_{\alpha \nu r_n} (\hat{z}_{n + 1}).
\]
Due to \Cref{lem: existence k}, we have
\begin{align}\label{eq: bound from lemma}
\frac{(1 - \varepsilon) u(\hat{x}_n)}{(1 + 3\varepsilon) H u (\hat{y}_n)} \leq \frac{\max_{\Theta^p (K_{\alpha \nu r_n} (\hat{z}_{n + 1}))} u}{\min_{\Theta^p (K_{\alpha \nu r_n} (\hat{z}_{n + 1}))} u}.
\end{align}
Now, we explain how \(\nu\) has to be chosen. Namely, take \(\nu\) such that
\begin{align*}
\nu \cdot \alpha\cdot \zeta^{\frac{\log(\mathfrak{t})}{- \log (\gamma)}} \leq \inf_{i \in \mathbb{Z}_+} \Big(\frac{r_{i + 1}}{r_i}\Big),
\end{align*}
where \(\mathfrak{t}>1\) is a constant we determine later. With this choice of \(\nu\) we can apply the oscillation inequality and obtain that
\begin{align}\label{eq: useful osc bound}
\underset{\Theta^p (K_{r_{n + 1}} (\hat{z}_{n + 1}))} \osc u \geq \mathfrak{t} \underset{\Theta^p (K_{\alpha \nu r_n} (\hat{z}_{n + 1}))} \osc u.
\end{align}
Using \eqref{eq: bound from lemma} and \eqref{eq: useful osc bound}, we further obtain that
\begin{equation}\label{eq: main bound final lem pf}\begin{split}
\frac{\max_{\Theta^p (K_{r_{n + 1}} (\hat{z}_{n + 1}))} u}{\max_{A_{n, k}} u} &= \frac{\min_{\Theta^p (K_{r_{n + 1}} (\hat{z}_{n + 1}))} u}{\max_{A_{n, k}} u} + \frac{\osc_{\Theta^p (K_{r_{n + 1}} (\hat{z}_{n + 1}))} u}{\max_{A_{n, k}} u}
\\&\geq \frac{\mathfrak{t} \osc_{\Theta^p (K_{\alpha \nu r_n} (\hat{z}_{n + 1}))}u}{\max_{\Theta^p (K_{\alpha \nu r_n} (\hat{z}_{n + 1}))}u}
\\&=  \mathfrak{t}\cdot \Big(1 - \frac{\min_{\Theta^p (K_{\alpha \nu r_n} (\hat{z}_{n + 1}))} u}{\max_{\Theta^p (K_{\alpha \nu r_n} (\hat{z}_{n + 1}))} u}\Big)
\\&\geq  \mathfrak{t}\cdot \Big(1 - \frac{ (1 + 3 \varepsilon) H u(\hat{y}_n)}{(1 - \varepsilon) u(\hat{x}_n)} \Big).
\end{split}
\end{equation}
Let \(\hat{x}_{n + 1}\) be the point where \(u\) attains its maximum on \(\Theta^p (K_{r_{n + 1}} (\hat{z}_{n + 1}))\).

Next, we explain how \(\hat{y}_{n + 1}\) and \(\mathfrak{t}\) are chosen. At this point we also explain how \(\{\mathbb{C}_1,  \dots, \mathbb{C}_N\},\) \(\theta_1, \dots, \theta_N\) and \(\delta^*\) are chosen. Take  \(\mathbb{C}_k\) and \(\theta_k\) such that there is a cylinder \(K_{\theta_k r_n}(\hat{u})\) with \(\hat{y}_n \in K_{\theta_k r_n}(\hat{u})\) and \(\theta_k r_n C_k (\hat{u})\subset K_{r_{n + 1}} (\hat{z}_{n + 1}) - (0, 2 r^2_{n + 1})\), see \Cref{fig: FS explain}. Here, \(\theta_k r_n C_k(\hat{u})\) is defined in the same manner as for \(\{\A_1, \dots, \A_N\}\). Recalling that \(\omega \in \mathcal{U}^*_{\hat{u}, \theta_k r_n}\), we can take \(\delta^*\) small enough such that there exists a uniform constant \(\mathfrak{m} > 1\) such that \(\Phi^*_{\hat{u}, \theta_k r_n, k} (\hat{y}_n) \geq \mathfrak{m}^{-1}\).
Then, take \(\hat{y}_{n + 1}\) to be the point in \(\Theta^p(\theta_k r_n C_k(\hat{u}))\) where \(u\) attains its minimum.
The optional stopping theorem yields that 
\begin{align}\label{eq: y_n ineq}
u(\hat{y}_n) \geq \frac{u(\hat{y}_{n + 1})}{\mathfrak{m}}.
\end{align}
We now impose an assumption on \(\mathfrak{t}\):
\begin{align}\label{eq: t cond}
\mathfrak{t} \geq 
\frac{2 \mathfrak{m} (1 + 3 \varepsilon)}{\varepsilon^2}.
\end{align}
Using \Cref{lem: existence k} and \eqref{eq: main bound final lem pf}, we obtain that
\begin{equation}\label{eq: imp ineq}
\begin{split}
u(\hat{x}_{n + 1}) &\geq \mathfrak{t} \Big(1 - \frac{(1 + 3\varepsilon) H u(\hat{y}_n)}{(1 - \varepsilon) u(\hat{x}_n)} \Big) \max_{A_{n, k}} u
\\&\geq \frac{\mathfrak{t}\varepsilon}{1 + 3 \varepsilon}   \Big(1 - \frac{ (1 + 3\varepsilon) H u(\hat{y}_n)}{(1 - \varepsilon) u(\hat{x}_n)} \Big) u(\hat{x}_n).
\end{split}
\end{equation}

\begin{lemma}\label{lem: ind final proof}
	For \(n \in [\k - 1]\) we have
	\[
	1 - \frac{ (1 + 3\varepsilon) H u(\hat{y}_{n + 1})}{(1 - \varepsilon) u(\hat{x}_{n + 1})} \geq \varepsilon\quad \Big(\Leftrightarrow\quad u(\hat{x}_{n + 1}) \geq \frac{(1 + 3 \varepsilon) H u(\hat{y}_{n + 1})}{(1 - \varepsilon)^2}\Big).
	\]
\end{lemma}\begin{proof}
	We use induction.
	For \(n = 0\) the claim follows from \eqref{eq: cont assp PHI}. Suppose that the claim holds for \(n \in [\k - 2]\). 
	Together with \eqref{eq: y_n ineq}, the induction hypothesis yields that
	\[
	\frac{(1 + 3 \varepsilon)H u(\hat{y}_{n + 1})}{(1 - \varepsilon)^2 \mathfrak{m}}  \leq \frac{(1 + 3 \varepsilon)Hu(\hat{y}_n)}{(1 - \varepsilon)^2}  \leq u(\hat{x}_n).
	\]
	Using this bound, \eqref{eq: imp ineq} and the induction hypothesis again, we obtain that 
	\[
	u(\hat{x}_{n + 1}) \geq \frac{\mathfrak{t} \varepsilon^2 u(\hat{x}_{n})}{1 + 3 \varepsilon} \geq \frac{\mathfrak{t} \varepsilon^2 H u(\hat{y}_{n + 1})}{(1 - \varepsilon)^2 \mathfrak{m}}.
	\]
	The assumption \eqref{eq: t cond} implies the claim.
\end{proof}
Now, \eqref{eq: t cond}, \eqref{eq: imp ineq} and \Cref{lem: ind final proof} yield that
\[
u(\hat{x}_{n + 1}) \geq \frac{\mathfrak{t} \varepsilon^2}{1 + 3 \varepsilon} u(\hat{x}_n) \geq 2 \mathfrak{m} u(\hat{x}_n).
\]
Inductively, we see that 
\[
u(\hat{x}_\k) \geq 2^\k \mathfrak{m}^\k u(\hat{x}_0)
\]
and \eqref{eq: K cond} completes the proof of the second claim in \Cref{lem: key lemma PHI} with \(M = \mathfrak{m}^\k\). To see that the first claim holds, note that 
\[
u(\hat{y}_K) \leq \mathfrak{m}^\k u(\hat{y}_0) = \mathfrak{m}^\k u(\hat{y}^*).
\]
Thus, the first claim follows from the argument we used to generate \((y_n)_{n \in [\k]}\). The proof is complete.
\qed
\subsection{Proof of \Cref{lem: strong sink bound}}
By the first part of \Cref{lem: key lemma PHI} there exists a point \(\hat{y}\) of the same parity as \(\hat{x}^*\) such that its space coordinate is in \(B_{2R}\) and at most at distance \(R^\delta\) from those of \(\hat{z}\), the time coordinate of \(\hat{y}\) is at least at distance \(R^{2\delta}\) and at most at distance \(2R^{2 \delta}\) from those of \(\hat{z}\), and \(u(\hat{y}) \leq M u(\hat{y}^*)\).
We now distinguish two cases. 

First, if \(\hat{y} \in \mathcal{C}_{\widetilde{\omega}}\) we use \(\omega \in \mathcal{S}_R\) and the optional stopping theorem to obtain that
\(
u(\hat{y}) \geq u(\hat{z}) \kappa^{2R^{2 \delta}},
\)
provided \(R\) is large enough. This yields the claim.

Second, if \(\hat{y} \not \in \mathcal{C}_{\widetilde{\omega}}\) we guide the walk into \(\mathcal{C}_{\widetilde{\omega}}\). Since \(\omega \in \mathcal{J}_R\), the worst case is that \(\hat{y}\) is in a hole of \(\mathcal{C}_{\widetilde{\omega}}\) of radius \(\lfloor R^{\xi}\rfloor\). As \(\omega \in \mathsf{B}\), with probability at least \(\frac{1}{2d}\) the walk in \(\omega\) goes a step in direction of the boundary of the hole. Thus, with probability at least \((2d)^{-d \lfloor R^{\xi}\rfloor}\) the walk is in \(\mathcal{C}_{\widetilde{\omega}}\). Recalling that \(\xi < \delta\) and that \(\kappa < \frac{1}{2d}\), the claim follows as before.
\qed

\subsection{Proof of \Cref{lem: prob strong sink bound}}
Recall that \(\omega \in \mathcal{I}_R\). Thus, to be at time \(T\) not in \(\mathcal{C}_{\widetilde{\omega}}\), the walk may not leave the ball \(B_{\mathfrak{o}\sqrt{d}R^{\xi/4}} (x) \equiv B_{\c R^{\xi/4}} (x)\) before it leaves the cylinder \(K_{2R}\), which is necessarily via its time boundary when \(R\) is large enough. In other words, we have
\[
\big\{X_T \not \in \mathcal{C}_{\widetilde{\omega}}\big\} \subseteq \big\{S > \varepsilon R^2\big\},
\]
where 
\[
S \triangleq \inf(n \in \mathbb{Z}_+ \colon X_n \not \in B_{\c R^{\xi / 4}} (x)).\]
Set \(o \triangleq \lfloor R^{2 - \xi} \rfloor\). We show by induction that for \(n = 1, \dots, o\)
\begin{align}\label{eq: to show by induction}
\sup \big( P^y_\omega (S > n \lfloor \varepsilon R^{\xi} \rfloor ) \colon y \in B_{\c R^{\xi/4}}(x)\big) \leq \mathfrak{w}^{- 2 n}.
\end{align}
For the induction base note that for all \(y \in  B_{\c R^{\xi/4}}(x)\)
\[
P^y_\omega (S > \lfloor \varepsilon R^{\xi} \rfloor ) \leq  \frac{E^y_\omega [S]}{\lfloor \varepsilon R^{\xi} \rfloor} \leq \C R^{- \xi/ 2} < \mathfrak{w}^{-2},
\]
in case \(R\) is large enough. For the induction step assume that \eqref{eq: to show by induction} holds for \(n \in \{1, \dots, o - 1\}\). The Markov property of the walk yields that for all \(y \in  B_{\c R^{\xi/4}}(x)\)
\begin{align*}
P^y_\omega (S > (n + 1) \lfloor \varepsilon R^{\xi}\rfloor) &= P^y_\omega (S > (n + 1) \lfloor \varepsilon R^{\xi}\rfloor, S > \lfloor \varepsilon R^{\xi}\rfloor)
\\&= E^y_\omega \big[ P_\omega^{X(\lfloor \varepsilon R^{\xi}\rfloor)} (S > n \lfloor \varepsilon R^{\xi}\rfloor) \1_{\{S > \lfloor \varepsilon R^{\xi}\rfloor\}}\big]
\\&\leq \mathfrak{w}^{- 2n} P^y_\omega (S > \lfloor \varepsilon R^{\xi}\rfloor) 
\\&\leq \mathfrak{w}^{- 2(n + 1)}. 
\end{align*}
Using \eqref{eq: to show by induction} with \(n = o\) yields that 
\[
P^x_\omega (S > \varepsilon R^2) \leq P^x_\omega (S > o \lfloor \varepsilon R^{\xi}\rfloor) \leq \mathfrak{w}^{- R^{2 - \xi}}.
\]
The lemma is proven. \qed

\appendix

\section{Proof of \Cref{theo:trans}} \label{sec:pf trans}
The proof is similar to those of \cite[Theorem 3.3.22]{Zeitouni2004}. The only differences are that instead of the EHI \cite[Lemma 3.3.8]{Zeitouni2004} one has to use \cite[Theorem 1.6]{Berger18} and that instead of \cite[Eq.  3.3.23]{Zeitouni2004} one can use the martingale property of the walk and the optional stopping theorem,  see also the proof of \cite[Theorem 2 (i)]{guozeitouni12}.

We give some details: Let \(R_0 \geq 1\) be a large constant and set \(R_i \triangleq R_0^i\) and 
\[
B^i (z)\triangleq  \big\{ x \in \mathbb{Z}^d \colon \|x - z\|_\infty < R_i\big\}, \quad i \in \N. %,\quad E_i \triangleq \big\{ x \in \mathbb{Z}^d \colon \tfrac{R_i}{2} < \|x\|_\infty < \tfrac{3R_i}{2}\big\}.
\]
We shall also write \(B^i \triangleq B^i(0)\). % and \(\overline{B}^i \triangleq B^i \cup \partial B^i\).
Set 
\[
\tau_i \triangleq \inf (n \in \N \colon X_n \not \in B^i), \quad i \in \N.
\]
Due to \cite[Theorem 1.6]{Berger18}, provided \(R_0\) is large enough, there exist constants \(\gamma, \delta  > 0\) and a set \(G_i \in  \mathcal{F}\) such that for every \(\omega \in G_i\), every \(z \in \partial B^i\) and every \(x \in B^{i - 1}\) it holds that  
\[
\max_{y \in B^{i - 1} (z)} E_\omega^{y} \big[ \# \text{ visits of \(x\) before }  \tau_{i + 2} \big] \leq \gamma \min_{y \in B^{i - 1} (z)} E_\omega^{y} \big[ \# \text{ visits of \(x\) before }  \tau_{i + 2} \big],
\]
and
\[
P(G_i) \geq 1 - e^{- R_{i - 1}^{\delta}}.
\]
Let \((\theta^x)_{x \in \mathbb{Z}^d}\) be the canonical shifts on \(\Omega\), i.e. \((\theta^x \omega) (y, e) = \omega(x + y, e)\).
We obtain for every \(\omega \in G_i\) and all \(z \in \partial B^i\) that
\begin{align*}
\sum_{x \in B^{i-1}} E^{z}_{\theta^x \omega} \big[  \# \text{ visits of \(0\) before }  \tau_{i + 1} \big] 
&\leq \sum_{x \in B^{i-1}}\max_{y \in B^{i - 1} (z)} E^{y}_{\omega} \big[  \# \text{ visits of \(x\) before }  \tau_{i + 2} \big]
\\&\leq \sum_{x \in B^{i - 1}}\gamma \min_{y \in B^{i - 1} (z)} E^{y}_{\omega} \big[  \# \text{ visits of \(x\) before }  \tau_{i + 2} \big]
\\
&\leq \gamma E^{z}_{\omega} \big[  \# \text{ visits of \(B^{i-1}\) before }  \tau_{i + 2} \big]
\\&\leq \gamma E^{z}_\omega \big[ \tau_{i + 2} \big]
\\&\leq \gamma \c R_{i + 2}^2.
\end{align*}
Using the shift invariance of \(P\) and the Markov property of the walk yields that 
\begin{align*}
\int E^{0}_\omega &\big[ \# \text{ visits of 0 in } (\tau_{i}, \tau_{i + 1}] \big] P(d \omega) \\&= \frac{1}{|B^{i - 1}|} \int \sum_{y \in B^{i - 1}} E_{\theta^y \omega}^0 \big[ \# \text{ visits of 0 in } (\tau_{i}, \tau_{i + 1}] \big] P(d \omega) 
\\&\leq \frac{1}{|B^{i - 1}|} \int_{G_i} \sum_{y \in B^{i - 1}} E_{\theta^y \omega}^0 \big[E_{\theta^y \omega}^{X_{\tau_i}} \big[ \# \text{ visits of 0 before } \tau_{i + 1} \big]\big] P(d \omega) + \c R_{i + 1}^2 P(G_i^c)
\\&\leq \C \big(R_i^{2 - d} + R_{i + 1}^2 e^{- R^\delta_i}\big).
\end{align*}
Recalling that \(d \geq 3\) and summing over \(i\) shows that 
\[
\int E^{0}_\omega \big[ \# \text{ visits of 0} \big] P(d \omega) < \infty, 
\]
which implies that the walk is transient for \(P\)-a.a. environments. \qed

\section{Proof of \Cref{theo: ABP}} \label{app: ABP pf}
	We borrow ideas from the proofs of \cite[Theorem 3.1]{Berger2014} and \cite[Theorem 2.2]{deuschel2018}.
	Define
	\[
	M \triangleq \sup_{Q_R} u, \qquad \Theta \triangleq \big\{ (y, s) \in \mathbb{R}^{d + 1} \colon (2 + \sqrt{d}) R \|y\|_2 < s < \tfrac{M}{2}\big\}.
	\]
	W.l.o.g. we assume that \(M > 0\). 
	Note that 
	\begin{align*}
	\llambda(\Theta) = \int_0^{\frac{M}{2}} \Big( \int_{\mathbb{R}^d} \1_{\{\|x\|_2 < \frac{s}{2R}\}} dx\Big)ds 
	= \c \int_0^{\frac{M}{2}} \frac{s^d ds}{R^d}  = \frac{\c M^{d + 1}}{R^d},
	\end{align*}
	where \(\llambda\) denotes the Lebesgue measure.
	Consequently, we have 
	\[
	M = \c R^{\tfrac{d}{d + 1}} \llambda(\Theta)^{\tfrac{1}{d + 1}}.
	\]
	In other words, \eqref{eq: ABP} follows once we show that
	\begin{align}
	\llambda (\Theta) \leq \c \sum_{(y, s) \in \Gamma_u} \big|E^{y}_\omega \big[u(X_{T^{(k)}}, s + 1+ T^{(k)})\big] - u(y, s + 1) \big|^{d + 1}.
	\end{align}
	The proofs of the following lemmata are postponed till the proof of \Cref{theo: ABP} is complete.
	\begin{lemma}\label{lem: 1-C1}
		We have 
		\[
		\llambda(\Theta) \leq \sum_{(y, s) \in \Gamma_u}  (u(y, s) - u(y,s + 1)) \llambda(I(y, s)).
		\]
		(Note that \(u(y, s) - u(y, s + 1) \geq 0\) whenever \((y, s)\in \Gamma_u\).)
	\end{lemma}
	\begin{lemma}\label{lem: 2-C1}
		There exists an \(R_o > 0\) such that whenever \(R \geq R_o\) for all \((y, s) \in \Gamma_u\)
		\[
		\llambda(I(y, s)) \leq 4^{d}\big(E^{y}_\omega \big[u(X_{T^{(k)}}, s + 1+ T^{(k)})\big] - u(y, s) \big)_+^{d}.
		\]
	\end{lemma}
	Set 
	\[
	\Psi
	\triangleq \big\{ (y, s) \in \Gamma_u  \colon E^{y}_\omega \big[u(X_{T^{(k)}}, s + 1 + T^{(k)})\big] - u(y, s) > 0\big\}.
	\]
	Using \Cref{lem: 1-C1,lem: 2-C1} and the arithmetic-geometric mean inequality, we obtain for all \(R \geq R_o\)
	\begin{align*} \llambda(\Theta) &\leq \c \sum_{(y, s) \in \Psi}  (u(y, s) - u(y, s + 1)) \big(E^{y}_\omega \big[u(X_{T^{(k)}}, s + 1+T^{(k)})\big] - u(y, s) \big)^{d}
	\\&\leq \c\sum_{(y, s) \in \Psi}  \big(u(y, s) - u(y, s + 1) + d E^{y}_\omega \big[u(X_{T^{(k)}}, s + 1+ T^{(k)})\big] - du(y, s) \big)^{d + 1}\\
	&\leq \c \sum_{(y, s) \in \Psi}  \big(E^{y}_\omega \big[u(X_{T^{(k)}},s + 1+T^{(k)})\big] - u(y, s + 1) \big)^{d + 1}
	\\&\leq \c\sum_{(y, s) \in \Gamma_u}  \big|E^{y}_\omega \big[u(X_{T^{(k)}}, s + 1+T^{(k)})\big] - u(y, s + 1) \big|^{d + 1}.
	\end{align*}
	The claim of \Cref{theo: ABP} follows.
\qed
\\

It remains to prove the \Cref{lem: 1-C1,lem: 2-C1}.
\begin{proof}[Proof of \Cref{lem: 1-C1}] We borrow arguments from the proof of \cite[Theorem 2.2]{deuschel2018}.
	Set 
	\[
	\chi (y,  s) \triangleq \big\{ (p, q - \langle y, p\rangle) \colon p \in I_u (y, s) \text{ and } q \in [u(y, s + 1), u(y, s)]
	\big\} \subset \mathbb{R}^{d + 1}.
	\]
	The key observation is the following inclusion:
	\begin{align}\label{eq: incl-C1}
	\Theta \subseteq \chi(\Gamma_u) \triangleq \bigcup_{(y, s) \in \Gamma_u} \chi(y, s).
	\end{align}
	Let us accept \eqref{eq: incl-C1} for a moment. Then, using that the map \((y, z) \mapsto (y, z - \langle \beta, y\rangle)\) preserves volume, because it has determinant one, we obtain
	\begin{align*}
	\llambda (\Theta) \leq \llambda (\chi (\Gamma_u)) \leq \sum_{(y, s) \in \Gamma_u} (u(y, s) - u(y, s + 1)) \llambda (I_u(y,s )),
	\end{align*}
	which is the claim.
	
	It remains to prove \eqref{eq: incl-C1}.
	Let \((y, s) \in \Theta\) and define 
	\[
	\phi(x, t) \triangleq u(x, t) - \langle y, x\rangle - s, \quad (x, t) \in Q^k_R.
	\]
	Let \((y_0, s_0) \in Q_R\) be such that \(u(y_0, s_0) = M\). 
	Recalling the definition of \(\Theta\), we see that \(\phi(y_0, s_0) > 0\) and that
	\(\phi(x, t) < 0\) for all \((x, t) \in Q^k_R\) with \(u(x, t) \leq 0\).
	Let 
	\begin{align*}
	N_x &\triangleq \max  ( t \colon (x, t) \in Q^k_R \text{ and } \phi(x, t) \geq 0), \quad \max (\emptyset) \triangleq - \infty.
	\end{align*}
	Note that \(s_0 \leq N_{y_0}
	\leq s_1 \triangleq \max(N_x \colon x \in O_R \cup \partial^k O_R)= \max (N_x \colon x \in O_R) \leq \lfloor R^2 \rfloor - 1\). 
	Let \(y_1\) be such that \(s_1 = N_{y_1}\), and note that \((y_1, s_1) \in Q_R\).
	For all \((x, t) \in Q^k_R\) with \(t > s_1\) we have \(\phi (x, t) < 0\), which yields that 
	\(
	u(x, t) - \langle y, x\rangle < s\leq u(y_1, s_1) - \langle y, y_1\rangle,
	\)
	because \(\phi(y_1, s_1) \geq 0\).
	This implies that \(y\in I_u(y_1, s_1)\). 
	By definition of \(s_1\), we have 
	\(
	\phi(y_1, s_1 + 1) < 0, 
	\) and hence \(u(y_1, s_1 + 1) < \langle y, y_1\rangle + s\). We conclude that 
	\(
	u(y_1, s_1 + 1) < \langle y, y_1\rangle + s \leq u(y_1, s_1),
	\)
	which finally implies \((y, s) \in \chi (y_1, s_1)\) and thus \eqref{eq: incl-C1} holds.
\end{proof}

\begin{proof}[Proof of \Cref{lem: 2-C1}]
	We borrow the idea of the proof of \cite[Lemma 3.4]{Berger2014}. 
	Fix \((y, s) \in Q_R\). 
	For \(i = 1, \dots, d\) define
	\[
	u_i \triangleq \inf(n \in \N \colon \alpha(n) = i).
	\]
	Furthermore, we define the following events:
	\begin{align*}
	A^{(+)}_i &\triangleq \big\{X_{u_i} - X_{u_i - 1} = e_i, u_i \leq k\big\},\\
	A^{(-)}_i &\triangleq \big\{X_{u_i} - X_{u_i - 1} = - e_i, u_i \leq k\big\}.
	\end{align*}
	Let \(W\) be a random variable independent of the walk, which takes the values \(\pm 1\) with probability \(\frac{1}{2}\).
	Finally, define 
	\begin{align*}
	B^{(+)}_i &\triangleq A^{(+)}_i \cup \big(\{W = + 1\} \cap \{u_i > k\}\big),\\
	B^{(-)}_i &\triangleq A^{(-)}_i \cup \big(\{W = - 1\} \cap \{u_i > k\}\big).
	\end{align*}
	We note that \(B_i^{(+)}\) and \(B_i^{(+)}\) are disjoint and that the union is \(P_\omega^{y}\)-full. Thus, due to symmetry, we have 
	\[
	P_\omega^{y} (B^{(+)}_i) = P_\omega^{y} (B^{(-)}_i) = \tfrac{1}{2}.
	\]
	As \(\omega\in \mathsf{B}\), the walk \(X\) is a \(P_\omega^y\)-martingale and \(E^{y}_\omega[ X_{T^{(k)}}] = y\) follows from the optional stopping theorem.
	In summary, we obtain 
	\begin{align*}
	E_\omega^y \big[ X_{T^{(k)}} | B^{(+)}_i\big] &= 2 E_\omega^y \big[ X_{T^{(k)}} \1_{B^{(+)}_i}\big] 
	\\&= 2 \big(E_\omega^{y} \big[ X_{T^{(k)}}\big] - E_\omega^{y} \big[ X_{T^{(k)}} \1_{B^{(-)}_i}\big]\big)
	\\&= 2 y - E_\omega^{y} \big[ X_{T^{(k)}} |B^{(-)}_i\big].
	\end{align*}
	Hence, we have
	\[
	\mathcal{O}_i \triangleq E_\omega^{y} \big[ X_{T^{(k)}} | B^{(+)}_i\big] - y = y - E_\omega^{y} \big[ X_{T^{(k)}} |B^{(-)}_i\big].
	\]
	Take \(\beta \in I_u(y, s)\).
	Using the definition of \(I_u(y, s)\), we obtain
	\begin{align*}
	\langle \beta, \mathcal{O}_i \rangle &= \sum_{x \in O_R \cup \partial^k O_R} \langle \beta, x - y\rangle P_\omega^{y} (X_{T^{(k)}} = x | B^{(+)}_i)
	\\&= \sum_{(x, t) \in Q_R \cup \partial^k Q_R} \langle \beta, x - y\rangle P_\omega^{y} (X_{T^{(k)}} = x, T^{(k)} = t - s - 1 | B^{(+)}_i)
	\\&\geq \sum_{(x, t) \in Q_R \cup \partial^k Q_R} (u(x, t) - u(y, s)) P_\omega^{y} (X_{T^{(k)}} = x, T^{(k)} = t - s - 1 | B^{(+)}_i)
	\\&= E_\omega^{y} \big[ u(X_{T^{(k)}}, s + 1 + T^{(k)})| B^{(+)}_i\big] - u(y, s).
	\end{align*}
	Similarly, we see that 
	\begin{align*}
	\langle \beta, - \mathcal{O}_i \rangle 
	%\\&
	\geq E_\omega^{y} \big[ u(X_{T^{(k)}}, s + 1 + T^{(k)})| B^{(-)}_i\big] - u(y, s).
	\end{align*}
	Consequently, \(\langle \beta,\mathcal{O}_i\rangle\) lies in an interval which length is bounded by \(2 L\), where
	\begin{align*}
	L \triangleq u(y, s) - E^{y}_\omega \big[ u(X_{T^{(k)}}, s + 1+T^{(k)}) \big]. 
	\end{align*}
	We conclude that 
	\[
	\llambda(I_u(y, s)) \leq \llambda\big( \big\{ z \in \mathbb{R}^{d} \colon \forall_{i = 1, \dots, d} \ \  \big\langle z, \mathcal{O}_i\big\rangle \in [0, 2 L] \big\} \big).
	\]
	Note that
	\[
	\llambda \big( \big\{ z \in \mathbb{R}^{d} \colon \forall_{i = 1, \dots, d} \ \  \big\langle z, \mathcal{O}_i\big\rangle \in [0, 2 L] \big\} \big) = (2 L)^d |\textup{det } (M)|,
	\]
	where 
	\(
	M = \begin{pmatrix} \mathcal{O}_1&\cdots& \mathcal{O}_d\end{pmatrix}.
	\)
	Due to Hadamard's determinant inequality,
	we have 
	\[
	|\textup{det} (M)| \leq \prod_{i = 1}^{d} \big\| \mathcal{O}_i\big\|.
	\]
	For large enough \(R\), we deduce from \cite[Claim 3.5]{Berger2014} that
	\[
	\big\| e_i -\mathcal{O}_i\big\| < \exp \big(- (\log R)^2\big).
	\]
	Consequently, by the triangle inequality, we have 
	\[
	|\textup{det} (M)| \leq \big(1 + \exp \big( - (\log R)^2 \big)\big)^d.
	\]
	Now, the claim of the lemma follows.
\end{proof}

\bibliographystyle{plain}

%\bibliography{References}

\end{document}